\documentclass[11pt,oneside]{amsart}

\usepackage{amsmath,ifthen, amsfonts, amssymb,
srcltx, amsopn, color, enumerate} 
\usepackage[linktocpage=true]{hyperref}
\usepackage[cmtip,arrow]{xy}
\usepackage{pb-diagram, pb-xy}
\usepackage{overpic}

\usepackage[T1]{fontenc}

\usepackage{hieroglf}

\newcommand{\showcomments}{yes}

\newsavebox{\commentbox}

\newcounter{intronum}
\newcounter{ax}
\newtheorem{thm}{Theorem}[section]
\newtheorem{lem}[thm]{Lemma}

\newtheorem{cor}[thm]{Corollary}

\newtheorem{prop}[thm]{Proposition}

\newtheorem{thmi}{Theorem}
\newtheorem{cori}[thmi]{Corollary}

\theoremstyle{definition}
\newtheorem{defn}[thm]{Definition}
\newtheorem{rem}[thm]{Remark}
\newtheorem*{unrem}{Remark}
\newtheorem{remi}[intronum]{Remark}

\newtheorem{notation}[thm]{Notation}

\newtheorem{claim}{Claim}

\newtheorem{claim*}{Claim}

\newtheorem{cons}[thm]{Construction}

\DeclareMathOperator{\diam}{\textup{\textsf{diam}}}

\DeclareMathOperator{\asdim}{asdim}

\newcommand{\neb}{\mathcal N}
\def\MCG{\mathcal{MCG}}

\newcommand{\field}[1]{\mathbb{#1}}
\newcommand{\integers}{\ensuremath{\field{Z}}}

\newcommand{\naturals}{\ensuremath{\field{N}}}

\newcommand{\inv}{^{-1}}
\newcommand{\HHS}[2]{\ensuremath{(\cuco {#1},\frak {#2})}}

\newcommand{\pyramid}[1]{\mathrm{Pyr}({#1})}

\newcommand{\interior} [1] {{\ensuremath \text{\rm Int}(#1) }}

\newcommand{\hhhg}{\hookrightarrow_{_{hh}}}

\makeatletter

\newcommand{\Rmnum}[1]{\mathbf{{\expandafter\@slowromancap\romannumeral #1@}}}

\makeatother

\newcommand{\tup}[1]{\vec{#1}}

\let\oldmarginpar\marginpar
\renewcommand\marginpar[1]{\-\oldmarginpar[\raggedleft\footnotesize #1]{\raggedright\footnotesize #1}}

\newcommand{\tsh}[1]{\left\{\kern-.7ex\left\{#1\right\}\kern-.7ex\right\}}
\newcommand{\Tsh}[2]{\tsh{#2}_{#1}}
\newcommand{\ignore}[2]{\Tsh{#2}{#1}}
\newcommand{\retract}{\mathfrak l}

\newcommand{\co}{\colon}

\newcounter{enumitemp}

\newcommand{\dist}{\textup{\textsf{d}}}

\newcommand{\cuco}[1]{{\mathcal #1}}

\newcommand{\fontact}{{\mathcal C}}

\newcommand{\gate}{\mathfrak g}

\newcommand{\coneoff}[1]{\widehat{#1}}

 \usepackage{mathabx}
\newcommand{\propnest}{\sqsubsetneq}

\newcommand{\nest}{\sqsubseteq}
\newcommand{\orth}{\bot}
\newcommand{\transverse}{\pitchfork}

\newcommand{\relevant}{\mathbf{Rel}}

\newcommand{\cayley}{\mathrm{Cay}}
\newcommand{\cox}{\coneoff{\cuco X}}
\newcommand{\cone}{\mathfrak C}
\newcommand{\induced}{^{*}}
\newcommand{\inducedS}{^{\tiny{\diamondsuit}}}
\newcommand{\nclose}[1]{\widehat{#1}}

\begin{document}
\title[Asdim and small-cancellation for HHS]
{Asymptotic dimension and small-cancellation for hierarchically hyperbolic spaces and groups}

\author[J. Behrstock]{Jason Behrstock}
\address{Lehman College and The Graduate Center, CUNY, New York, New York, USA}
\email{jason.behrstock@lehman.cuny.edu}
\thanks{\flushleft{Behrstock was supported as a Simons Fellow.}}

\author[M.F. Hagen]{Mark F. Hagen}
\address{Dept. of Pure Maths and Math. Stat., University of Cambridge, Cambridge, UK}
\email{markfhagen@gmail.com}
\thanks{\flushleft {Hagen was supported by the EPSRC}}

\author[A. Sisto]{Alessandro Sisto}
\address{ETH, Z\"{u}rich, Switzerland}
\email{sisto@math.ethz.ch}
\thanks{\flushleft{Sisto was supported by the Swiss National Science Foundation project 144373}}

\maketitle

\begin{abstract}
We prove that all hierarchically hyperbolic groups have finite
asymptotic dimension. 
One application of this result is to obtain 
the sharpest known
bound on the asymptotic dimension of the mapping class group of a
finite type surface: improving the bound from exponential to at most 
quadratic in the 
complexity of the surface.  We also apply the main result to various other
hierarchically hyperbolic groups and spaces. We also prove a 
small-cancellation result namely: 
if $G$ is a hierarchically hyperbolic group, $H\leq G$ is a suitable
hyperbolically embedded subgroup, and $N\triangleleft H$ is
``sufficiently deep'' in $H$, then $G/\nclose N$ is a relatively
hierarchically hyperbolic group. This new class provides many new examples 
to which our asymptotic dimension bounds apply.
Along the way, we prove new results about the structure of HHSs,
for example: the associated hyperbolic spaces are always obtained,
up to quasi-isometry, by coning off canonical coarse product regions
in the original space (generalizing a relation established by 
Masur--Minsky between the complex 
of curves of a surface and Teichm\"{u}ller space).
\end{abstract}

\tableofcontents

\renewcommand{\qedsymbol}{$\Box$}

\section*{Introduction}\label{sec:intro}

Motivated by 
the observation that a suitable CAT(0) cube complex, equipped with a
collection of hyperbolic graphs encoding the relationship between its
hyperplanes, has many properties exactly parallel to those of the
mapping class group, equipped with the collection of curve graphs of
subsurfaces, we introduced  the
class of hierarchically hyperbolic spaces, abbreviated HHS, as 
a notion of ``coarse nonpositive
curvature'' which provides a framework for studying these two 
seemingly disparate classes
of spaces/groups.

The class of hierarchically hyperbolic spaces 
consists of metric spaces whose geometry
can be recovered, coarsely, from projections onto a specified
collection of hyperbolic metric spaces; the axioms governing these
spaces and projections are modelled on the relation between the 
subsurface projections between the curve graph of a surface and the 
curve graph of its subsurfaces, i.e., the mapping class group of a surface is
the archetypal HHS
(see~\cite{MasurMinsky:I,MasurMinsky:II,BKMM:consistency}).  The
hyperbolic spaces onto which one projects are partially ordered so
that there is a unique maximal element and numerous elements that are
minimal in every chain in which they appear. Relaxing the
hyperbolicity requirement for these minimal spaces, one obtains the 
notion of a \emph{relatively hierarchically hyperbolic space}, 
abbreviated RHHS.  These
notions are reviewed in Section~\ref{sec:preliminaries} of this paper 
and a detailed discussion can be found 
in~\cite{BehrstockHagenSisto:HHS_I} and~\cite{BehrstockHagenSisto:HHS_II}.

\subsection*{Asymptotic dimension}

The \emph{asymptotic dimension} of a metric space is a well-studied
quasi-isometry invariant, introduced by
Gromov~\cite{Gromov:asymptotic}, which provides a coarse version of
the topological dimension.  Early motivation for studying asymptotic
dimension was provided by Yu, who showed that groups with finite asymptotic
dimension satisfy both the coarse Baum-Connes and the Novikov 
conjectures \cite{Yu:novikov}. It is now known that asymptotic 
dimension provides coarse analogues for many properties of 
topological dimension, see \cite{BellDranishnikov:asdim_1} for a 
recent survey. Using very different techniques a number of groups and 
spaces have been shown to have finite asymptotic dimension, although 
good estimates on this dimension have proved difficult in many cases:  
curve graphs 
\cite{BellFujiwara:asdim_curve_graph, BestvinaBromberg:asdim}, 
mapping class groups 
\cite{BBF:quasi_tree}, cubulated groups 
\cite{Wright:asdim}, graph manifold groups \cite{Smirnov:asdimgraphmfld}, 
and groups hyperbolic relative to 
ones with finite asymptotic dimension \cite{Osin:asdim}.

One of our main results is the following very general result, which 
in addition to covering many new cases, provides a unified proof of finite asymptotic dimension for almost all the cases just mentioned:

\begin{thmi}\label{thmi:asdim}
Let $\cuco X$ be a uniformly proper HHS.  Then $\asdim\cuco 
X<\infty$. In particular, any HHG has finite asymptotic dimension.
\end{thmi}

Theorem~\ref{thmi:asdim} is proven in Section~\ref{sec:main_asdim_theorem}, where 
we establish 
the slightly stronger Theorem~\ref{thm:technical_asdim}, which 
obtains explicit bounds on the dimension.

In special cases where there is a hierarchical structure with a known
bound on the asymptotic dimension of the various $\fontact U$, we can
obtain fairly effective bounds on this dimension, as we now show in 
the case of the mapping class group.

Bestvina--Bromberg--Fujiwara established 
finiteness of the asymptotic dimension of mapping class group 
\cite{BBF:quasi_tree}, but without providing explicit 
bounds.   
Bestvina--Bromberg then improved on their prior work to obtain an 
explicit bound on the asymptotic 
dimension which is exponential in the complexity of the surface 
\cite{BestvinaBromberg:asdim}. Bestvina--Bromberg 
conjectured that the asymptotic 
dimension of the mapping class group is equal to 
its virtual cohomological dimension (and, in particular, linear in the 
complexity of the surface). 

Here, using the
hierarchically hyperbolic structure on the mapping class group of a
surface, constructed in~\cite{BehrstockHagenSisto:HHS_II},  
Section~\ref{sec:main_asdim_theorem} we show that a careful 
application of   
Theorem~\ref{thm:technical_asdim} yields the following  
which improves the sharpest bounds from exponential to quadratic:

\begin{cori}[Asymptotic dimension of $\MCG(S)$]\label{cori:MCG}
Let $S$ be a connected oriented surface of finite type of complexity
$\xi(S)\geq 2$.  Then $\asdim\MCG(S)\leq 5\xi(S)^2$.
\end{cori}

For convenience, we omit from the statement the case of connected oriented
surfaces of finite type and complexity at most one (i.e., $S^2$ with at
most 4 punctures and $S^1\times S^1$ with at most 1 puncture); this 
is not omitting any cases of interest since
such mapping class groups are either finite, $\mathbb Z$,
or virtually free, and hence their asymptotic dimensions are 0 or 1.

The notion of
a hierarchically hyperbolic structure plays a central role in
establishing the bound in Corollary~\ref{cori:MCG}. 
Indeed, our proof of Theorem~\ref{thmi:asdim} relies on the fact that
``coning off'' an appropriate collection of subspaces of an HHS yields
a new HHS of lower complexity, and the bound we eventually obtain is
in terms of a uniform bound on the asymptotic dimensions of the
hyperbolic spaces in the HHS structure, which is obtained separately
in Corollary~\ref{cor:finite_asdim_curve_graph}.  

To establish this bound for general (relatively) hierarchically
hyperbolic spaces, we generalize the ``tight geodesics'' strategy of
Bell--Fujiwara~\cite{BellFujiwara:asdim_curve_graph}, who proved that
the asymptotic dimension of the curve graph of a surface of finite
type is finite.  Bell--Fujiwara's work relies on a finiteness theorem of
Bowditch~\cite{Bowditch:tight}, which does not provide an explicit
bound on the asymptotic dimension. 

In the case of the mapping class group, we do not use such upper
bounds.  Instead, in the proof of Corollary~\ref{cori:MCG}, we
sidestep Section~\ref{sec:tight} and the ``tight geodesics'' method
completely and instead use the linear bound (in terms of complexity)
on the asymptotic dimension of the curve graph provided by
Bestvina-Bromberg in~\cite{BestvinaBromberg:asdim} when we invoke
Theorem~\ref{thm:technical_asdim}.  The structure of the proof of
Theorem~\ref{thm:technical_asdim} allows this, and depends in an
essential way on the notion of a hierarchically hyperbolic space. 
Interestingly, even though the notion of an HHS generalizes known 
structures in the mapping class group, the present work can not be 
employed in the mapping class group case without appeal to the full 
generality of hierarchically hyperbolic spaces. 
Roughly, this is because our approach involves coning off certain subspaces of
a (relatively) hierarchically hyperbolic space to produce a
hierarchically hyperbolic space of lower complexity, enabling
induction. Although this procedure keeps us in the category of being 
an HHS, being a mapping class group is not similarly
closed under this coning operation. 

In~\cite{BBF:quasi_tree}, the asymptotic dimension of the mapping
class group is shown to be finite as a consequence of the fact that
the asymptotic dimension of Teichm\"{u}ller space $\mathcal T(S)$ is
finite.  Our method gives an improved bound on
$\asdim(\mathcal T(S))$, where $\mathcal T(S)$ is given either the 
Teichm\"{u}ller metric or the Weil--Petersson metric:

\begin{cori}\label{cori:teich}
Let $S$ be a connected oriented surface of finite type of complexity $\xi(S)\geq 1$.  Then 
$\asdim(\mathcal T(S))\leq 5\xi(S)^2+\xi(S)$.
\end{cori}

Pre-existing bounds on the asymptotic dimension of the associated
collection of hyperbolic spaces can be used to bound the asymptotic
dimension of some other HHS without recourse to
Corollary~\ref{cor:finite_asdim_curve_graph}.  For example,
in~\cite{BehrstockHagenSisto:HHS_I}, we showed
that if $\cuco X$ is a CAT(0) cube complex admitting a collection of
convex subcomplexes called a \emph{factor system}, then $\cuco X$
admits a hierarchically hyperbolic structure in which the associated
hyperbolic spaces are uniformly quasi-isometric to simplicial trees, and thus have asymptotic dimension $\leq1$.
This holds in particular when $\cuco X$ embeds convexly in the
universal cover of the Salvetti complex of a right-angled Artin group
$A_\Gamma$ associated to a finite simplicial graph $\Gamma$.
Theorem~\ref{thm:technical_asdim} then provides $\asdim\cuco
X\leq\sum_{\ell=0}^{|\Gamma^{(0)}|}K_\ell$, where $K_\ell\leq\ell$ is
the maximum size of a clique appearing in a subgraph of $\Gamma$ with
$\ell$ vertices.  This reproves finiteness of the asymptotic dimension
for such complexes, as established by Wright~\cite{Wright:asdim}.

We also recover the following result of Osin~\cite{Osin:asdim}:

\begin{cori}[Asymptotic dimension of relatively hyperbolic groups]\label{cori:rel_hyp}
Let the group $G$ be hyperbolic relative to a finite collection $\mathcal P$ of peripheral subgroups such that $\asdim P<\infty$ for each $P\in\mathcal P$.  Then $\asdim G<\infty$.
\end{cori}

\begin{proof}
It is easy to verify that $(G,\mathfrak S)$ is a relatively hierarchically
hyperbolic space, where $\mathfrak S$ consists of $G$ together with
the set of all conjugates of elements of $\mathcal P$; each of these
conjugates is nested in $G$ and the conjugates are pairwise-transverse
(the orthogonality relation is empty);
see~\cite{BehrstockHagenSisto:HHS_II}.  The result then follows 
immediately from
Theorem~\ref{thm:technical_asdim}.
\end{proof}

\subsection*{Quotients of hierarchically hyperbolic groups}\label{subsec:intro_quotients}
As discussed above, the first examples of hierarchically hyperbolic
groups were mapping class groups and many cubical groups \cite{BehrstockHagenSisto:HHS_I}; many 
further constructions and combination theorems 
were then provided in \cite{BehrstockHagenSisto:HHS_II}. 
In
Section~\ref{sec:HHG_quotient} we provide many new examples of (relatively) hierarchically
hyperbolic groups, which arise as quotients of hierarchically
hyperbolic groups by suitable subgroups, using small-cancellation
techniques closely related to the theory developed in~\cite{DGO}. In the aforementioned paper, the authors introduced the notion of \emph{hyperbolically embedded subgroup} of a group and extended the relatively hyperbolic Dehn filling theorem \cite{Osin:peripheral, GrovesManning:dehn}, thereby constructing many interesting quotients of groups such as mapping class groups. In particular, they showed that mapping class groups are SQ-universal, i.e. for every hyperbolic surface $S$ and for every countable group $Q$ there exists a quotient of $\MCG(S)$ containing an isomorphic copy of $Q$. Roughly speaking, we prove that Dahmani-Guirardel-Osin's construction of quotients preserves (relative) hierarchical hyperbolicity when applied to a (relatively) hierarchically hyperbolic group.

We say that the group $H$ is \emph{hierarchically hyperbolically embedded} if $G$ can be generated by a set $\mathcal T$ so that: $\mathcal T\cap H$ generates $H$, and $\cayley(G,\mathcal T)$ is quasi-isometric to the $\nest$--maximal element of $\mathfrak S$, and $H$ is hyperbolically embedded in $(G,\mathcal T)$ in the sense of~\cite{DGO}.  Theorem~\ref{thmi:quotients} is a direct consequence of Theorem~\ref{thm:quotients} below.  The proof of Theorem~\ref{thmi:quotients} relies heavily on the ``small-cancellation'' methods of~\cite{DGO}.

\begin{thmi}\label{thmi:quotients}
Let $(G,\mathfrak S)$ be an HHG and let $H\hhhg(G,\mathfrak S)$ be \emph{hierarchically hyperbolically embedded}.  Then there exists a finite set $F\subset H-\{1\}$ such that for all $N\triangleleft H$ with $F\cap N=\emptyset$, the group $G/\nclose N$ is a relatively hierarchically hyperbolic group.  If, in addition, $H/N$ is hyperbolic, then $G/\nclose N$ is hierarchically hyperbolic.
\end{thmi}

Here $\nclose{N}$ denotes the normal closure of $N$ in $G$. We remark that acylindrically hyperbolic groups contain plenty of hyperbolically embedded subgroups, and in particular they contain hyperbolically embedded virtually $F_2$ subgroups \cite[Theorem 6.14]{DGO}. Moreover, hierarchically hyperbolic groups are ``usually'' acylindrically hyperbolic, in the sense that any non-elementary hierarchically hyperbolic group $G$ so that $\pi_S(G)$ is unbounded is acylindrically hyperbolic, where $S$ is the $\nest$--maximal $S\in\mathfrak S$.

Theorem~\ref{thmi:quotients} provides many new examples of
hierarchically hyperbolic groups, and hence, via
Theorem~\ref{thmi:asdim}, expands the class of groups known to have
finite asymptotic dimension.  For example:

\begin{cori}\label{cori:kill_pA}
Let $S$ be a surface of finite type and let $f\in\MCG(S)$ be a
pseudo-Anosov element.  Then there exists $N$ such that
$\MCG(S)/\langle\langle f^{kN}\rangle\rangle$ is a hierarchically
hyperbolic group for all integers $k\geq 1$ and has asymptotic dimension at 
most $5\xi(S)^2$.
\end{cori}

Corollary~\ref{cori:kill_pA} has an exact analogue in the world of cubical groups.

\begin{cori}\label{cori:kill_rank_one}
Let $X$ be a compact special cube complex with universal cover $\widetilde X$ and let $G$ act properly and cocompactly on $\widetilde X$.  Let $g\in G$ be a rank-one element, no nonzero power of which stabilizes a hyperplane.  Then there exists $N$ such that for all integers $k\geq 1$, the group $G/\langle\langle g^{kN}\rangle\rangle$ is hierarchically hyperbolic and hence has finite asymptotic dimension.
\end{cori}

Since the proofs of Corollaries \ref{cori:kill_pA} and \ref{cori:kill_rank_one} are very similar, we only give that of Corollary \ref{cori:kill_rank_one}.

\begin{proof}
As shown in~\cite{BehrstockHagenSisto:HHS_I,BehrstockHagenSisto:HHS_II}, $G$ is a hierarchically hyperbolic group, where the top-level associated hyperbolic space is quasi-isometric to the intersection graph of the hyperplane carriers in $\widetilde X$.  As shown in~\cite{Hagen:boundary}, the given $g$ acts loxodromically on this graph, and thus the maximal elementary subgroup containing $g$ is hierarchically hyperbolically embedded in $G$, see \cite[Corollary 3.11]{antolin2015commensurating} or \cite[Corollary 4.14]{hull2013small} (which both refine \cite[Theorem 6.8]{DGO}).  Finally, apply Theorem~\ref{thmi:quotients}.
\end{proof}

Note that $G$ need not be virtually special to satisfy the hypotheses of Corollary~\ref{cori:kill_rank_one} (for example, the non-virtually special examples of Burger-Mozes and Wise~\cite{BurgerMozes,Wise:CSC} act geometrically on the universal cover of the product of two finite graphs), so the corollary can not be proved via cubical small-cancellation theory or related techniques (see~\cite{Wise:quasiconvex_hierarchy}) followed by an application of the results of~\cite{BehrstockHagenSisto:HHS_I}.  Stronger versions of Corollary~\ref{cori:kill_pA} and Corollary~\ref{cori:kill_rank_one} exist, where one kills more complex subgroups.  

\begin{remi}After we posted the initial version of this paper, it was shown in~\cite{HagenSusse:cubical} that any proper CAT(0) cube complex admitting a proper, cocompact group action is a hierarchically hyperbolic space (and the group in question a hierarchically hyperbolic group).  Hence the conclusion of Corollary~\ref{cori:kill_rank_one} holds with the ``special'' hypothesis.  This is also true for the above-mentioned asymptotic dimension result.\end{remi}\label{remi:new_cubes}

\subsection*{Factored spaces}

In section~\ref{sec:factored} we give a construction which we call a 
\emph{factored space}. Roughly, the factored space, $\coneoff{\cuco 
X}$,  associated to a 
hierarchically hyperbolic space $(\cuco X,\mathfrak S)$ is obtained 
by collapsing particularly subsets of $\cuco X$ which are isomorphic 
to direct products. In Proposition~\ref{prop:cone_off_bottom} we 
prove that this construction yields a new HHS.

One particularly interesting consequence of this construction is the 
following corollary which is a special case of 
Corollary~\ref{cor:cone_off_all_F}. This result 
generalizes \cite[Theorem~1.2]{MasurMinsky:I} where it is proven that 
the the curve graph of a surface is quasi-isometric to 
Teichm\"{u}ller space after collapsing the thin parts, and, also, 
\cite[Theorem~1.3]{MasurMinsky:I} where the mapping class group 
is considered and the multicurve stabilizers are collapsed.

\begin{cori}[$\coneoff{\cuco X}$ is QI to $\pi_S(\cuco X)$]\label{thm:coarse_factors}
Let $(\cuco X,\mathfrak S)$ be hierarchically hyperbolic and 
let
$S\in\mathfrak S$ be the unique $\nest$--maximal element. 
Then $\coneoff{\cuco X}$ is quasi-isometric to 
$\pi_S(\cuco X)\subseteq\fontact S$.
\end{cori}

This result implies that if 
$(G,\mathfrak S)$ is a hierarchically hyperbolic group, then
$\fontact S$ is quasi-isometric to the coarse intersection graph of
the ``standard product regions.''

This result can be interpreted as stating that
hierarchically hyperbolic structures always arise from a coarse
version of the ``factor system'' construction used
in~\cite{BehrstockHagenSisto:HHS_I} to endow CAT(0) cube complexes
with hierarchically hyperbolic structures.

\subsection*{Structure of the paper}\label{subsec:structure}
In Section~\ref{sec:preliminaries}, we review basic facts about
asymptotic dimension and about hierarchical spaces and groups,
including (relatively) hierarchically hyperbolic ones.  In
Section~\ref{sec:factored}, we introduce a coning construction 
which shows that the
top-level hyperbolic space associated to a hierarchically hyperbolic
space is quasi-isometric to the space obtained by coning off the
``standard product regions.'' This construction, which we use in
the inductive proof of Theorem~\ref{thmi:asdim}, is
of independent interest and generalizes a construction we had 
originally 
included in the first version of \cite{BehrstockHagenSisto:HHS_II}.  
 Finiteness of the
asymptotic dimension of the hyperbolic spaces associated to a
relatively hierarchically hyperbolic space, is proved in
Section~\ref{sec:tight}.
In Section~\ref{sec:ball_preimage}, we
prove one of the key propositions needed in the induction argument for Theorem~\ref{thmi:asdim}.
In Section~\ref{sec:main_asdim_theorem}, we prove Theorem~\ref{thmi:asdim} and
its corollaries, and finally we prove 
Theorem~\ref{thmi:quotients} in Section~\ref{sec:HHG_quotient}.

\subsection*{Acknowledgments}\label{subsec:acknowledgments}
The authors thank the organizers of the conferences \emph{Manifolds
and groups} (Ventotene, September 2015) and the \emph{Th\'eorie
g\'eom\'etrique et asymptotique des groupes et applications} (CIRM,
September 2015), where our initial discussions took place.  The 
authors thank Carolyn Abbott for helpful feedback. The
authors also thank the referee for numerous comments that improved the
exposition.

\section{Preliminaries}\label{sec:preliminaries}

\subsection{Background on asymptotic dimension}\label{defn:asdim_background}
Let $(\cuco X,\dist)$ be a metric space.  There are several equivalent
definitions of the \emph{asymptotic dimension} of $\cuco X$ (see
e.g.~\cite{BellDranishnikov:Bedlewo} or 
\cite{BellDranishnikov:asdim_1} for comprehensive surveys).  We say
that $\asdim\cuco X \leq n$ if for each $D>0$, there exist $B\geq0$ 
and families
$\mathcal U_0,\ldots,\mathcal U_n$ of subsets which form a cover of $\cuco X$ such that:
\begin{enumerate}
 \item for all $i\leq n$ and all $U\in\mathcal U_i$, we have $\diam(U)\leq B$;\label{item:bounded}
 \item for all $i\leq n$ and all $U,U'\in\mathcal U_i$, if $U\neq U'$ then $\dist(U,U')>D$.\label{item:d_disjoint}
\end{enumerate}
A function $f\colon [0,\infty)\to[0,\infty)$ such that for each sufficiently large $D$, there is a cover of $\cuco X$ as above that satisfies part~\eqref{item:d_disjoint} for the given $D$ and satisfies part~\eqref{item:bounded} with $B=f(D)$ is an \emph{$n$--dimensional control function} for $\cuco X$.

We say  $\cuco X$ \emph{has asymptotic dimension $n$ (with control function $f$)}, when $n$ is minimal so that $\asdim\cuco X\leq n$ (and $f$ is an $n$--dimensional control function).

A family of metric spaces, $\{(\cuco X_i,\dist_i)\}_{i\in I}$, 
\emph{has $\asdim\cuco X_i \leq n$ uniformly}
if for all sufficiently large $D\geq0$, there exists $B\geq0$ such
that for each $i\in I$, there are sets $\mathcal
U_0^i,\ldots,\mathcal U^i_n$ of subsets of $\cuco X_i$, collectively covering $\cuco X_i$, so that: 
\begin{enumerate}
\item for all $i\in I$, all $0\leq j\leq n$, and all $U\in\mathcal U_j^i$, we have $\diam(U)\leq B$;
 \item for all $i\in I$, all $0\leq j<k\leq n$, and all $U,U'\in\mathcal U_j^i$, if $U\neq U'$ then $\dist(U,U')>D$.
\end{enumerate}
As above, $f\colon [0,\infty)\to[0,\infty)$ is an \emph{$n$--dimensional control function} for $\{\cuco X_i\}$ if for each $i$, and each sufficiently large $D$, we can choose the covers above so that if the second condition is satisfied for $D$, then the first is satisfied with $B=f(D)$.

Equivalently, $\asdim\cuco X\leq n$ if for all $r\geq0$ there exists a
uniformly bounded cover of $\cuco X$ such that any $r$--ball
intersects at most $n+1$ sets in the
cover~\cite{BellDranishnikov:asymptotic_groups}, and $\{\cuco X_i\}$
has $\asdim\cuco X_i\leq n$ uniformly if for each $r$ the covers can
be chosen to consist of sets bounded independently of~$i$.  We will
use this formulation in Section~\ref{sec:tight}.

We will require the following theorems of Bell--Dranishnikov:

\begin{thm}[Fibration theorem; \cite{BellDranishnikov:hurewicz}]\label{thm:fibration_theorem}
Let $\psi\co\cuco X\to\cuco Y$ be a Lipschitz map, with $\cuco X$ a geodesic space and $\cuco Y$ a metric space.  Suppose that for each $R>0$, the collection $\{\psi^{-1}(B(y,R))\}_{y\in\cuco Y}$ has $\asdim \psi^{-1}(B(y,R))\leq n$ uniformly.  Then $\asdim\cuco X\leq\asdim \cuco Y+n$.
\end{thm}

\begin{thm}[Union theorem; \cite{BellDranishnikov:asymptotic_groups}]\label{thm:union_theorem}
Let $\cuco X$ be a metric space and assume that $\cuco X=\bigcup_{i\in I}\cuco X_i$, where $\{\cuco X_i\}_{i\in I}$ satisfies $\asdim\cuco X_i\leq n$ uniformly.  Suppose that for each $R$ there exists $Y_R\subset\cuco X$, with $\asdim Y_R\leq n$, such that for all distinct $i,i'\in I$, we have $\dist(\cuco X_i-Y_R,\cuco X_{i'}-Y_R)\geq R$.  Then $\asdim\cuco X\leq n$. 
\end{thm}

\subsection{Background on hierarchical spaces}\label{subsec:hierarchical_prelim}
We recall our main definition from~\cite{BehrstockHagenSisto:HHS_II}:

\begin{notation}
In Definition~\ref{defn:space_with_distance_formula} below, we use the notation $\dist_W(-,-)$ to denote distance in a space $\fontact W$, where $W$ is in an index set $\mathfrak S$.  We will follow this convention where it will not cause confusion. However, in Section~\ref{sec:HHG_quotient}, where there are multiple HHS structures and spaces in play, we generally avoid this abbreviation. Similarly, where it will not cause confusion, we write, e.g. $\dist_W(x,y)$ to mean $\dist_W(\pi_W(x),\pi_W(y))$, where $x,y\in\cuco X$, and $W\in\mathfrak S$, and $\pi:\cuco X\to\fontact W$ is a projection.  We emphasise that, throughout the text, e.g. $\dist_W(x,y)$ and $\dist_{\fontact W}(\pi_W(x),\pi_W(y))$ mean the same thing.
\end{notation}

\begin{defn}[Hierarchical space, (relative) hierarchically hyperbolic space]\label{defn:space_with_distance_formula}
The $q$--quasigeodesic space  $(\cuco X,\dist)$ is a \emph{hierarchical space} if there exists an index set $\mathfrak S$, and a set $\{\fontact W:W\in\mathfrak S\}$ of geodesic spaces $(\fontact U,\dist_U)$,  such that the following conditions are satisfied:\begin{enumerate}
\item\textbf{(Projections.)}\label{item:dfs_curve_complexes} There is 
a set $\{\pi_W\co\cuco X\rightarrow2^{\fontact W}\mid W\in\mathfrak S\}$ of \emph{projections} sending points in $\cuco X$ to sets of diameter bounded by some $\xi\geq0$ in the various $\fontact W\in\mathfrak S$.  Moreover, there exists $K$ so that each $\pi_W$ is $(K,K)$--coarsely Lipschitz.

 \item \textbf{(Nesting.)} \label{item:dfs_nesting} $\mathfrak S$ is
 equipped with a partial order $\nest$, and either $\mathfrak
 S=\emptyset$ or $\mathfrak S$ contains a unique $\nest$--maximal
 element; when $V\nest W$, we say $V$ is \emph{nested} in $W$.  We
 require that $W\nest W$ for all $W\in\mathfrak S$.  For each
 $W\in\mathfrak S$, we denote by $\mathfrak S_W$ the set of
 $V\in\mathfrak S$ such that $V\nest W$.  Moreover, for all $V,W\in\mathfrak S$
 with $V\propnest W$ there is a specified subset
 $\rho^V_W\subset\fontact W$ with $\diam_{\fontact W}(\rho^V_W)\leq\xi$.
 There is also a \emph{projection} $\rho^W_V\colon \fontact
 W\rightarrow 2^{\fontact V}$.  (The similarity in notation is
 justified by viewing $\rho^V_W$ as a coarsely constant map $\fontact
 V\rightarrow 2^{\fontact W}$.)
 
 \item \textbf{(Orthogonality.)} 
 \label{item:dfs_orthogonal} $\mathfrak S$ has a symmetric and anti-reflexive relation called \emph{orthogonality}: we write $V\orth W$ when $V,W$ are orthogonal.  Also, whenever $V\nest W$ and $W\orth U$, we require that $V\orth U$.  Finally, we require that for each $T\in\mathfrak S$ and each $U\in\mathfrak S_T$ for which $\{V\in\mathfrak S_T:V\orth U\}\neq\emptyset$, there exists $W\in \mathfrak S_T-\{T\}$, so that whenever $V\orth U$ and $V\nest T$, we have $V\nest W$.  Finally, if $V\orth W$, then $V,W$ are not $\nest$--comparable.
 
 \item \textbf{(Transversality and consistency.)}
 \label{item:dfs_transversal} If $V,W\in\mathfrak S$ are not
 orthogonal and neither is nested in the other, then we say $V,W$ are
 \emph{transverse}, denoted $V\transverse W$.  There exists
 $\kappa_0\geq 0$ such that if $V\transverse W$, then there are
  sets $\rho^V_W\subseteq\fontact W$ and
 $\rho^W_V\subseteq\fontact V$ each of diameter at most $\xi$ and 
 satisfying: $$\min\left\{\dist_{
 W}(\pi_W(x),\rho^V_W),\dist_{
 V}(\pi_V(x),\rho^W_V)\right\}\leq\kappa_0$$ for all $x\in\cuco X$.
 
 For $V,W\in\mathfrak S$ satisfying $V\propnest W$ and for all
 $x\in\cuco X$, we have: $$\min\left\{\dist_{
 W}(\pi_W(x),\rho^V_W),\diam_{\fontact
 V}(\pi_V(x)\cup\rho^W_V(\pi_W(x)))\right\}\leq\kappa_0.$$ 
 
 The preceding two inequalities are the \emph{consistency inequalities} for points in $\cuco X$.
 
 Finally, if $U\propnest V$, then $\dist_W(\rho^U_W,\rho^V_W)\leq\kappa_0$ whenever $W\in\mathfrak S$ satisfies either $V\propnest W$ or $V\transverse W$ and $W\not\orth U$.
 
 \item \textbf{(Finite complexity.)} \label{item:dfs_complexity} There exists $n\geq0$, the \emph{complexity} of $\cuco X$ (with respect to $\mathfrak S$), so that any set of pairwise--$\nest$--comparable elements has cardinality at most $n$.
  
 \item \textbf{(Large links.)} \label{item:dfs_large_link_lemma} There
 exist $\lambda\geq1$ and $E\geq\max\{\xi,\kappa_0\}$ such that the following holds.
 Let $W\in\mathfrak S$ and let $x,x'\in\cuco X$.  Let
 $N=\lambda\dist_{_W}(\pi_W(x),\pi_W(x'))+\lambda$.  Then there exists $\{T_i\}_{i=1,\dots,\lfloor
 N\rfloor}\subseteq\mathfrak S_W-\{W\}$ such that for all $T\in\mathfrak
 S_W-\{W\}$, either $T\in\mathfrak S_{T_i}$ for some $i$, or $\dist_{
 T}(\pi_T(x),\pi_T(x'))<E$.  Also, $\dist_{
 W}(\pi_W(x),\rho^{T_i}_W)\leq N$ for each $i$.
 
 \item \textbf{(Bounded geodesic image.)} \label{item:dfs:bounded_geodesic_image} For all $W\in\mathfrak S$, all $V\in\mathfrak S_W-\{W\}$, and all geodesics $\gamma$ of $\fontact W$, either $\diam_{\fontact V}(\rho^W_V(\gamma))\leq E$ or $\gamma\cap\neb_E(\rho^V_W)\neq\emptyset$. 
 
 \item \textbf{(Partial Realization.)} \label{item:dfs_partial_realization} There exists a constant $\alpha$ with the following property. Let $\{V_j\}$ be a family of pairwise orthogonal elements of $\mathfrak S$, and let $p_j\in \pi_{V_j}(\cuco X)\subseteq \fontact V_j$. Then there exists $x\in \cuco X$ so that:
 \begin{itemize}
 \item $\dist_{V_j}(x,p_j)\leq \alpha$ for all $j$,
 \item for each $j$ and 
 each $V\in\mathfrak S$ with $V_j\propnest V$, we have 
 $\dist_{V}(x,\rho^{V_j}_V)\leq\alpha$, and
 \item if $W\transverse V_j$ for some $j$, then $\dist_W(x,\rho^{V_j}_W)\leq\alpha$.
 \end{itemize}

\item\textbf{(Uniqueness.)} For each $\kappa\geq 0$, there exists
$\theta_u=\theta_u(\kappa)$ such that if $x,y\in\cuco X$ and
$\dist(x,y)\geq\theta_u$, then there exists $V\in\mathfrak S$ such
that $\dist_V(x,y)\geq \kappa$.\label{item:dfs_uniqueness}
\end{enumerate}
If there exists $\delta\geq0$ such that $\fontact U$ is $\delta$--hyperbolic for all $U\in\mathfrak S$, then $(\cuco X,\mathfrak S)$ is \emph{hierarchically hyperbolic}.  If there exists $\delta$ so that $\fontact U$ is $\delta$--hyperbolic for all non--$\nest$--minimal $U\in\mathfrak S$, then $(\cuco X,\mathfrak S)$ is \emph{relatively hierarchically hyperbolic}.
\end{defn}

We require the following proposition from~\cite{BehrstockHagenSisto:HHS_II}:

\begin{prop}[$\rho$--consistency]\label{prop:rho_consistency}
There exists $\kappa_1$ so that the following holds.  Suppose that $U,V,W\in\mathfrak S$ satisfy both of the following conditions: $U\propnest V$ or $U\transverse V$; and $U\propnest W$ or $U\transverse W$.  Then, if $V\transverse W$, then 
 $$\min\left\{\dist_{
 W}(\rho^U_W,\rho^V_W),\dist_{
 V}(\rho^U_V,\rho^W_V)\right\}\leq\kappa_1$$
 and if $V\propnest W$, then 
 $$\min\left\{\dist_{
 W}(\rho^U_W,\rho^V_W),\diam_{\fontact
 V}(\rho^U_V\cup\rho^W_V(\rho^U_W))\right\}\leq\kappa_1.$$
\end{prop}

\begin{notation}\label{notation:E}
Given a hierarchical space $\HHS X S$, let $E$ be the maximum of all of the constants appearing in Definition~\ref{defn:space_with_distance_formula} and Proposition~\ref{prop:rho_consistency}.  Moreover, if $\HHS X S$ is $\delta$--(relatively) HHS, then we choose $E\geq\delta$ as well.
\end{notation}

\begin{notation}\label{notation:orth_set}
Let $\HHS X S$ be a hierarchical space and let $\mathcal U\subset\mathfrak S$.  Given $V\in\mathfrak S$, we write $V\orth\mathcal U$ to mean $V\orth U$ for all $U\in\mathcal U$.
\end{notation}

We can now prove the following lemma, analogous to Definition~\ref{defn:space_with_distance_formula}.\eqref{item:dfs_transversal}:

\begin{lem}\label{lem:orth_close}
Let $(\cuco X,\mathfrak S)$ be a hierarchical space and let $W\in\mathfrak S$ and let $U,V\in\mathfrak S_W-\{W\}$ satisfy $U\orth V$.  Then $\dist_W(\rho^U_W,\rho^V_W)\leq 2E$.
\end{lem}

\begin{proof}
Apply partial realization (Definition~\ref{defn:space_with_distance_formula}.\eqref{item:dfs_partial_realization}.
\end{proof}

The following lemma, which is~\cite[Lemma~2.5]{BehrstockHagenSisto:HHS_II}, is used in Section~\ref{sec:tight}:

\begin{lem}[Passing large projections up the $\nest$--lattice]\label{lem:passing_up}
 Let $(\cuco X,\mathfrak S)$ be a hierarchical space.  For every
 $C\geq0$ there exists $N$ with the following property.  Let
 $V\in\mathfrak S$, let $x,y\in\cuco X$, and let
 $\{S_i\}_{i=1}^{N}\subseteq \mathfrak S_V-\{V\}$ be distinct and
 satisfy $\dist_{S_i}(x,y)\geq E$.  Then there exists
 $S\in\mathfrak S_V$ and $i$ so that $S_i\propnest S$ and
 $\dist_{S}(x,y)\geq C$.
\end{lem}

In this paper, we primarily work with relatively HHS.  The main results from~\cite{BehrstockHagenSisto:HHS_II} that we will require are realization, the distance formula, and the existence of hierarchy paths (Theorems~\ref{thm:realization},\ref{thm:distance_formula},\ref{thm:hierarchy_paths} below), whose statements require the following definitions:

\begin{defn}[Consistent tuple]\label{defn:consistent_tuple}
Let $\kappa\geq0$ and let $\tup b\in\prod_{U\in\mathfrak S}2^{\fontact U}$ be a tuple such that for each $U\in\mathfrak S$, the $U$--coordinate $b_U$ has diameter $\leq\kappa$.  Then $\tup b$ is \emph{$\kappa$--consistent} if for all $V,W\in\mathfrak S$, we have $$\min\{\dist_V(b_V,\rho^W_V),\dist_W(b_W,\rho^V_W)\}\leq\kappa$$ whenever $V\transverse W$ and $$\min\{\dist_W(x,\rho^V_W),\diam_V(b_V\cup\rho^W_V)\}\leq\kappa$$ whenever $V\propnest W$.
\end{defn}

\begin{defn}[Hierarchy path]\label{defn:hierarchy_path}
A path $\gamma\colon I\to\cuco X$ is a \emph{$(D,D)$--hierarchy path} if $\gamma$ is a $(D,D)$--quasigeodesic and $\pi_U\circ\gamma$ is an unparameterized $(D,D)$--quasigeodesic for each $U\in\mathfrak S$.
\end{defn}

\begin{thm}[Realization]\label{thm:realization}
Let $\HHS X S$ be a hierarchical space. Then for each $\kappa\geq1$, there exists $\theta=\theta(\kappa)$ so that, for any $\kappa$--consistent tuple $\tup b\in\prod_{U\in\mathfrak S}2^{\fontact U}$, there exists $x\in\cuco X$ such that $\dist_V(x,b_V)\leq\theta$ for all $V\in\mathfrak S$.
\end{thm}

\noindent Observe that uniqueness (Definition~\eqref{item:dfs_uniqueness} implies that the \emph{realization point} for $\tup b$ provided by Theorem~\ref{thm:realization} is coarsely unique.  The following theorem is Theorem~6.7 in~\cite{BehrstockHagenSisto:HHS_II}, which is proved using the corresponding statement for hierarchically hyperbolic spaces (\cite[Theorem~4.5]{BehrstockHagenSisto:HHS_II}):

\begin{thm}[Distance formula for relatively HHS]\label{thm:distance_formula}
Let $\HHS X S$ be a relatively hierarchically hyperbolic space.  Then
there exists $s_0$ such that for all $s\geq s_0$, there exist $C,K$ so
that for all $x,y\in\cuco X$,
$$\dist(x,y)\asymp_{K,C}\sum_{U\in\mathfrak
S}\ignore{\dist_U(x,y)}{s}.$$
\end{thm}

\noindent (The notation $\ignore{A}{B}$ denotes the quantity which is $A$ if $A\geq B$ and $0$ otherwise.)

The following closely-related statement is Theorem~6.8 of~\cite{BehrstockHagenSisto:HHS_II}:

\begin{thm}[Hierarchy paths in relatively HHS]\label{thm:hierarchy_paths}
Let $\HHS X S$ be a relatively hierarchically hyperbolic space.  Then there exists $D\geq0$ such that for all $x,y\in\cuco X$, there is a $(D,D)$--hierarchy path in $\cuco X$ joining $x,y$.
\end{thm}

\subsubsection{Hierarchical quasiconvexity, gates, and standard product regions}\label{subsubsec:HQSP}
The next definition slightly generalizes Definition~5.1
of~\cite{BehrstockHagenSisto:HHS_II} (which it was stated for the 
case of 
hierarchically hyperbolic spaces):

\begin{defn}[Hierarchical quasiconvexity in relatively HHS]\label{defn:hier_quasi}
Let $\HHS X S$ be a $\delta$--relatively hierarchically hyperbolic space, and let $\cuco Y\subseteq\cuco X$.  Then $\cuco Y$ is \emph{hierarchically quasiconvex} if there exists a function $k\colon [0,\infty)\to[0,\infty)$ such that:
\begin{itemize}
 \item for each $U\in\mathfrak S$ with $\fontact U$ a $\delta$--hyperbolic space, the subspace $\pi_U(\cuco Y)\subseteq\fontact U$ is $k(0)$--quasiconvex;
 \item for each ($\nest$--minimal) $U\in\mathfrak S$ for which $\fontact U$ is not $\delta$--hyperbolic, either $\fontact U=\neb^{\fontact U}_{k(0)}(\pi_U(\cuco Y))$ or $\diam(\pi_U(\cuco Y))\leq k(0)$;
 \item for all $\kappa\geq0$ and all $\kappa$--consistent tuples $\tup
 b$ for which $b_U\subset\pi_U(\cuco Y)$ for all $U\in\mathfrak S$, each
 realization point $x\in\cuco X$ for which
 $\dist_U(\pi_U(x),b_U)\leq\theta(\kappa)$ satisfies $\dist(x,\cuco Y)\leq
 k(\kappa)$ (where $\theta(\kappa)$ is
 as in Theorem~\ref{thm:realization}) .
\end{itemize}
In this case, we say $\cuco Y$ is \emph{$k$--hierarchically quasiconvex} and refer to $k$ as a \emph{hierarchical quasiconvexity function} for $\cuco Y$.
\end{defn}

Let $\HHS X S$ be relatively hierarchically hyperbolic and let $\cuco Y\subseteq\cuco X$ be $k$--hierarchically quasiconvex.  Given $x\in\cuco X$ and $U\in\mathfrak S$, let $p_U(x)$ be defined as follows.  If $U$ is $\delta$--hyperbolic, then $p_U(x)$ is the coarse projection of $\pi_U(x)$ on $\pi_U(\cuco Y)$ (which is defined since $\pi_U(\cuco Y)$ is $k(0)$--quasiconvex).  If $\pi_U\colon \cuco Y\to\fontact U$ is $k(0)$--coarsely surjective, then $p_U(x)$ is the set of all $p\in\pi_U(\cuco Y)$ with $\dist_U(x,p)\leq k(0)$ (which is nonempty).  Otherwise, $\pi_U(\cuco Y)$ has diameter at most $k(0)$, and we let $p_U(x)=\pi_U(\cuco Y)$.  The tuple $(p_U(x))_{U\in\frak S}$ is easily checked to be $\kappa=\kappa(k(0))$--consistent, and we apply the realization theorem (Theorem~\ref{thm:realization}) and the uniqueness axiom to produce a coarsely well-defined point $\gate_{\cuco Y}(x)\in\cuco Y$ so that $\dist_U(\gate_{\cuco Y}(x),p_U(x))$ is bounded in terms of $k$ for all $U$.  The (coarsely well-defined) map $\gate_{\cuco Y}\colon\cuco X\to\cuco Y$ given by $x\mapsto\gate_{\cuco Y}(x)$ is the \emph{gate map} associated to $\cuco Y$.

Important examples of hierarchically quasiconvex subspaces of the relatively HHS $\HHS X S$ are the \emph{standard product regions} defined as follows (see~\cite[Section~5]{BehrstockHagenSisto:HHS_II} for more detail).  For each $U\in\mathfrak S$, let $\mathfrak S_U$ denote the set of $V\in\mathfrak S$ with $V\nest U$, and let $\mathfrak S_U^\orth$ denote the set of $V\in\mathfrak S$ such that $V\orth U$, together with some $A_U\in\mathfrak S$ such that $V\nest A_U$ for all $V$ with $V\orth U$.  Then there are uniformly hierarchically quasiconvex subspaces $\mathbf F_U,\mathbf E_U\subseteq\cuco X$ such that $(\mathbf F_U,\mathfrak S_U),(\mathbf E_U,\mathfrak S_U^\orth)$ are relatively hierarchically hyperbolic spaces and the inclusions $\mathbf F_U,\mathbf E_U\hookrightarrow\cuco X$ extend to a uniform quasi-isometric embedding $\mathbf F_U\times \mathbf E_U\to\cuco X$ whose image $\mathbf P_U$ is hierarchically quasiconvex.  We call $\mathbf P_U$ the \emph{standard product region} associated to $U$ and, for each $e\in \mathbf E_U$, the image of $\mathbf F_U\times\{e\}$ is a \emph{parallel copy of $\mathbf F_U$ (in $\cuco X$)}.  The relevant defining property of $\mathbf P_U$ is: there exists $\alpha$, depending only on $\cuco X,\mathfrak S$ and the output of the realization theorem, so that for all $x\in \mathbf P_U$ (and hence each parallel copy of $\mathbf F_U$), we have $\dist_V(x,\rho^U_V)\leq\alpha$ whenever $U\propnest V$ or $U\transverse V$.  Moreover we can choose $\alpha$ so that, if $U\orth V$, then $\diam(\pi_U(\mathbf F_V\times\{e\}))\leq\alpha$ for all $e\in\mathbf E_V$.

\begin{rem}[Gates in standard product regions and their factors]\label{rem:gates}
Let $\HHS X S$ be a relatively HHS and let $U\in\mathfrak S$.  The
gate map $\gate_{\mathbf P_U}\colon \cuco X\to \mathbf P_U$ can be
described as follows.  For each $x\in\cuco X$ and $V\in\mathfrak S$,
we have:
\begin{itemize}
 \item $\dist_V(\pi_V(\gate_{\mathbf P_U}(x)),\rho^U_V)\leq\alpha$ if $V\transverse U$ or $U\propnest V$;
 \item $\pi_V(\gate_{\mathbf P_U}(x))=\pi_V(x)$ otherwise.
\end{itemize}
For each $e\in \mathbf E_U$, the gate map $\gate_{\mathbf F_U\times\{e\}}\colon \cuco X\to \mathbf F_U\times\{e\}$ is described by:
\begin{itemize}
 \item $\dist_V(\pi_V(\gate_{\mathbf F_U\times\{e\}}(x)),\rho^U_V)\leq\alpha$ if $V\transverse U$ or $U\propnest V$;
 \item $\pi_V(\gate_{\mathbf F_U\times\{e\}}(x))=\pi_V(x)$ if $V\nest U$;
 \item $\dist_V(\pi_V(\gate_{\mathbf F_U\times\{e\}}(x)),\pi_V(e))\leq\alpha$ if $V\orth U$.
\end{itemize}
Likewise, for each $f\in \mathbf F_U$, the gate map $\gate_{\{f\}\times \mathbf E_U}\colon \cuco X\to \{f\}\times \mathbf E_U$ is described by:
\begin{itemize}
 \item $\dist_V(\pi_V(\gate_{\{f\}\times \mathbf E_U}(x)),\rho^U_V)\leq\alpha$ if $V\transverse U$ or $U\propnest V$;
 \item $\pi_V(\gate_{\{f\}\times \mathbf E_U}(x))=\pi_V(x)$ if $V\orth U$;
 \item $\dist_V(\pi_V(\gate_{\{f\}\times \mathbf E_U}(x)),\pi_V(f))\leq\alpha$ if $V\nest U$.
\end{itemize}
\end{rem}

\begin{rem}[Standard product regions in relatively HHS]\label{rem:standard_product_regions}
In~\cite[Section~5.2]{BehrstockHagenSisto:HHS_II}, standard product regions are constructed in the context of hierarchically hyperbolic spaces.  However, the construction uses only the hierarchical space axioms and realization (Theorem~\ref{thm:realization}), so that $\mathbf F_U,\mathbf E_U,\mathbf P_U$ can be constructed in an arbitrary hierarchical space.  The way we have defined things, the assertion that these subspaces are hierarchically quasiconvex requires $\HHS X S$ to be relatively hierarchically hyperbolic.  It is easy to see, from the definition, that the explanation of hierarchical quasiconvexity from~\cite{BehrstockHagenSisto:HHS_II} (for HHS) works in the more general setting of relatively HHS.
\end{rem}

\begin{defn}[Totally orthogonal]\label{defn:tot_orth}
Given a hierarchical space $\HHS X S$, we say that $\mathcal U\subset\mathfrak S$ is \emph{totally orthogonal} if $U\orth V$ for all distinct $U,V\in\mathcal U$.
\end{defn}

Recall from~\cite[Lemma~2.1]{BehrstockHagenSisto:HHS_II} that there is a uniform bound, namely the complexity, on the size of totally orthogonal subsets of $\mathfrak S$.  Observe that if $\mathcal U$ is a totally orthogonal set in the relatively HHS $\HHS X S$, then $\bigcap_{U\in\mathcal U}\mathbf P_U$ coarsely contains $\prod_{U\in\mathcal U}\mathbf F_U$.

\subsubsection{Partially ordering relevant domains}\label{subsubsec:partial_order}
Given a hierarchical space $\HHS X S$, a constant $K\geq0$, and $x,y\in\cuco X$, we say that $U\in\mathfrak S$ is \emph{$K$--relevant (for $x,y$)} if $\dist_U(x,y)\geq K$.  In Section~2 of~\cite{BehrstockHagenSisto:HHS_II}, it is shown that when $K\geq100E$, then any set $\relevant_{max}(x,y,K)$ of pairwise $\nest$--incomparable $K$--relevant elements of $\mathfrak S$ can be partially ordered as follows: if $U,V\in\relevant(x,y,K)$, then $U\preceq V$ if $U=V$ or if $U\transverse V$ and $\dist_U(\rho^V_U,y)\leq E$.  (This was done in~\cite{BehrstockHagenSisto:HHS_II} in the context of hierarchically hyperbolic spaces, but the arguments do not use hyperbolicity and thus hold for arbitrary hierarchical spaces.)

\subsubsection{Automorphisms and (relatively) hierarchically hyperbolic groups}\label{subsubsec:aut}
Let $(\cuco X,\mathfrak S)$ be a hierarchical space.  An \emph{automorphism} $g$ of $(\cuco X,\mathfrak S)$ is a map $g\colon \cuco X\to\cuco X$, together with a bijection $g\inducedS\colon \mathfrak S\to\mathfrak S$ and, for each $U\in\mathfrak S$, an isometry $g\induced(U)\colon \fontact U\to\fontact U$ so that the following diagrams coarsely commute whenever the maps in question are defined (i.e., when $U,V$ are not orthogonal):
\begin{center}
$
\begin{diagram}
\node{\cuco 
X}\arrow[3]{e,t}{g}\arrow{s,l}{\pi_U}\node{}\node{}\node{\cuco 
X'}\arrow{s,r}{\pi_{g\inducedS(U)}}\\
\node{\fontact (U)}\arrow[3]{e,t}{g\induced(U)}\node{}\node{}\node{\fontact 
(g\inducedS(U))}
\end{diagram}
$
\end{center}
and
\begin{center}
$
\begin{diagram}
\node{\fontact 
U}\arrow[4]{s,l}{\rho^U_V}\arrow[7]{e,t}{g\induced(U)}\node{}\node{}\node{}\node{}\node{}\node{}\node{}\node{}\node{}\node{\fontact(g\inducedS(U))}\arrow[4]{s,r}{\rho^{g\inducedS(U)}_{g\inducedS(V)}}\\
\node{}\node{}\node{}\node{}\node{}\node{}\node{}\node{}\\
\node{}\node{}\node{}\node{}\node{}\node{}\node{}\node{}\\
\node{}\node{}\node{}\node{}\node{}\node{}\node{}\node{}\\
\node{\fontact 
V}\arrow[7]{e,t}{g\induced(V)}\node{}\node{}\node{}\node{}\node{}\node{}\node{}\node{}\node{}\node{\fontact (g\inducedS (V))}
\end{diagram}
$
\end{center}
The finitely generated group $G$ is \emph{hierarchical} if there is a hierarchical structure $(G,\mathfrak S)$ on $G$ (equipped with a word-metric) so that the action of $G$ on itself by left multiplication is an action by HS automorphisms (with the above diagrams uniformly coarsely commuting).  If $(G,\mathfrak S)$ is a (relatively) hierarchically hyperbolic space, we say that $(G,\mathfrak S)$ (or just $G$) is a \emph{(relatively) hierarchically hyperbolic group [(R)HHG]}.

\subsection{Very rotating families}\label{subsec:rotating_families}
In Section~\ref{sec:HHG_quotient}, we will make use of the \emph{very rotating families} technology introduced in~\cite{DGO}.  All of the notions we need in that section are defined there, and we refer the reader to~\cite{DGO} or~\cite{Guirardel_notes} for additional background.

\section{Factored spaces}\label{sec:factored}
Given a hierarchical space $\HHS X S$, we say $\mathfrak U\subseteq\mathfrak S$ is \emph{closed under nesting} if for all $U\in\mathfrak U$, if $V\in\mathfrak S-\mathfrak U$, then $V\not\nest U$.

\begin{defn}[Factored space]\label{defn:factored_space}
Let $(\cuco X,\mathfrak S)$ be a hierarchical space. 
A \emph{factored space} $\coneoff{\cuco X}_{\mathfrak U}$ is constructed by
defining a new metric $\hat\dist$ on $\cuco X$ depending on a 
given subset $\mathfrak U \subset\mathfrak S$ which is closed under nesting. 
First, for each $U\in\mathfrak U$, for each pair $x,y\in\cuco X$ for 
which there exists $e\in {\bf E}_{U}$ such that $x,y\in {\bf 
F}_{U}\times \{e\}$, we set 
$\dist'(x,y)=\min\{1,\dist(x,y)\}$.  
For any pair $x,y\in\cuco X$ for which there does not exists such an 
$e$ we set $\dist'(x,y)=\dist(x,y)$.  We now
define the distance $\hat\dist$ on $\coneoff{\cuco X}_{\mathfrak U}$.  Given a
sequence $x_0,x_1,\ldots,x_k\in\coneoff{\cuco X}_{\mathfrak U}$, define its length
to be $\sum_{i=1}^{k-1}\dist'(x_i,x_{i+1})$.  Given
$x,x'\in\coneoff{\cuco X}_{\mathfrak U}$, let $\hat\dist(x,x')$ be the infimum of
the lengths of such sequences $x=x_0,\ldots,x_k=x'$.  
\end{defn}

Given a hierarchical space $\HHS X S$, and a set $\mathfrak
U\subseteq\mathfrak S$ closed under nesting, let $\psi \co \cuco X \to
\coneoff{\cuco X}_{\mathfrak U}$ be the  set-theoretic identity map.  Observe that:

\begin{prop}\label{prop:identity_Lipschitz}
The map $\psi\co \cuco X \to \coneoff{\cuco X}_{\mathfrak U}$ is Lipschitz.
\end{prop}

\begin{proof}
This follows from the definition of $\hat\dist$ and the fact that
$\cuco X$ is a quasigeodesic space.
\end{proof}

\begin{defn}[Hat space]\label{defn:hat_space} 
Let $\mathfrak U_{1}$ denote the set of $\nest$--minimal 
elements of $\mathfrak S$. The \emph{hat space}  $\coneoff{\cuco X}=\coneoff{\cuco X}_{\mathfrak U_1}$
is the factored space associated to the set $\mathfrak U_{1}$.  
\end{defn}

Recall that a $\delta$--relatively HHS is an HS \HHS X S, such that for all $U\in\mathfrak S$, either $\fontact U$ is $\delta$--hyperbolic 
    or $U\in\mathfrak U_{1}$.

\begin{prop}\label{prop:cone_off_bottom}
    Fix a $\delta$--relatively HHS, \HHS X S, and let $\mathfrak
    U\subset \mathfrak S$ be closed under $\nest$ and contain each
    $U\in\frak U_1$ for which $\fontact U$ is not
    $\delta$--hyperbolic.  The space $(\coneoff{\cuco X}_{\mathfrak U},\mathfrak S-
    \mathfrak U)$ is an HHS,
    where the associated $\fontact(*), \pi_{*}, \rho_{*}^{*}, 
    \nest,\orth,\transverse$ are the 
    same as in the original structure.
\end{prop}

\begin{proof}
We must verify each of the requirements of Definition~\ref{defn:space_with_distance_formula}.  First observe that by the definition of $\hat\dist$ and the fact that $(\cuco X,\dist)$ is a quasigeodesic space, $(\cox_{\mathfrak U},\hat\dist)$ is also a $(K,K)$--quasigeodesic space for some $K$.

\textbf{Projections:} 
By our
hypothesis on $\mathfrak U$, we have that $\fontact U$ is
$\delta$--hyperbolic for each $U\in\mathfrak S-\mathfrak U$, so it
remains to check that $(\coneoff{\cuco X},\mathfrak S-\frak U)$ is a
hierarchical space.  The projections $\pi_U\colon \cox_{\mathfrak
U}\to 2^{\fontact U}$ are as before (more precisely, they are
compositions of the original projections $\pi_U\colon \cuco X\to\fontact U$
with the  set-theoretic identity $\cox_{\mathfrak U}\to\cuco X$, but 
we will abuse
notation and call them $\pi_U$).

Fix $U\in\mathfrak S-\mathfrak U$.  By Definition~\ref{defn:space_with_distance_formula}.\eqref{item:dfs_curve_complexes}, there exists $K$, independent of $U$, so that $\pi_U$ is $(K,K)$--coarsely Lipschitz.  Let $x,y\in\cox_{\mathfrak U}$ and let $x=x_0,\ldots,x_\ell=y$ be a sequence with $\hat\dist(x,y)\geq\sum_{i=0}^{\ell-1}\dist'(x_i,x_{i+1})-1$. Note that $\dist_U(x,y)\leq\sum_{i=0}^{\ell-1}\dist_U(x_i,x_{i+1})$.  

Let $I_1$ be the set of $i\in\{0,\ldots,\ell-1\}$ such that $\dist'(x_i,x_{i+1})=\dist(x_i,x_{i+1})$, let $I_2$ be the set of $i$ for which $x_i,x_{i+1}$ lie in a common parallel copy of $\mathbf F_V$, where $V\transverse U$ or $V\propnest U$ and $V\in\mathfrak U$, and let $I_3$ be the set of $i$ so that $x_i,x_{i+1}$ lie an a common parallel copy of $\mathbf F_W$, where $W\orth U$ and $W\in\mathfrak U$.  Note that we do not need to consider the case where $W\in\mathfrak U$ and $U\nest W$, since $W\in\mathfrak U$ and $\mathfrak U$ is closed under nesting.  Then $$\dist_U(x,y)\leq\sum_{i\in I_1}\left[K\dist'(x_i,x_{i+1})+K\right]+ 2\alpha|I_2|+\alpha|I_3|.$$ 

The third term comes from the fact that, given $i\in I_3$ and $W\orth U$ the associated element of $\mathfrak U$ with $x_i,x_{i+1}\in\mathbf F_W\times\{e\}$ for some $e\in\mathbf E_W$, we have that $\pi_U(\mathbf F_W)$ has diameter at most $\alpha$, so $\dist_U(x_i,x_{i+1})\leq\alpha$.  Combining the above provides the desired coarse Lipschitz constant $C$.

\textbf{Nesting, orthogonality, transversality, finite complexity:}  The parts of Definition~\ref{defn:space_with_distance_formula} that only concern $\mathfrak S$ and the relations $\nest,\orth,\transverse$ continue to hold with $\mathfrak S$ replaced by $\mathfrak S-\mathfrak U$.  The complexity of $(\cox_{\mathfrak U},\frak S-\frak U)$ is obviously bounded by that of $\HHS X S$.  (Note that the fact that $\mathfrak U$ is closed under nesting is needed to ensure that for all $W\in\frak S-\frak U$ and $U\propnest W$, there exists $V\propnest W$ so that $T\nest V$ for each $T$ with $T\propnest W$ and $T\orth U$.)

\textbf{Consistency:}  Since the projections $\pi_*$ and relative projections $\rho^*_*$ have not changed, consistency holds for $(\cox_{\mathfrak U},\mathfrak S-\frak U)$ since it holds for $\HHS X S$.

\textbf{Bounded geodesic image and large links:} The bounded geodesic image axiom holds for $(\cox_{\mathfrak U},\mathfrak S-\frak U)$ since it holds for $\HHS X S$ and is phrased purely in terms of geodesics in the various $\fontact (*)$ and relative projections $\rho^*_*$.  The same applies to the large link axiom.

\textbf{Partial realization:} Since for each $U\in\mathfrak S$, we have $\pi_U(\cuco X)=\pi_U(\cox_{\mathfrak U})$, and since we have not changed any of the projections $\pi_*$ or relative projections $\rho^*_*$, the partial realization axiom for $\HHS X S$ implies that for $(\cox_{\mathfrak U},\frak S-\frak U)$.

\textbf{Uniqueness:} This is done in Lemma~\ref{lem:cone_off_uniqueness} below.
\end{proof}
%
%

\begin{defn}[Friendly]\label{defn:friendly}
For $U,V\in\mathfrak S$, we say that $U$ is \emph{friendly} to $V$ if
$U\nest V$ or $U\orth V$.  Note that when $U$ is \emph{not} friendly
to $V$, then $\rho^V_U$ is a uniformly bounded subset of $\fontact U$.
\end{defn}


In the proof of Lemma \ref{lem:cone_off_uniqueness}, we will need to ``efficiently'' jump between product regions $P_U, P_V$. Heuristically, the pairs of points that are ``closest in every $\fontact W$'' are of the form $p,q$ for some $p\in\gate_{\mathbf P_U}(\mathbf P_V)$ and $q=\gate_{\mathbf P_V}(p)$, and these are the ones we study in the following lemma. In particular, we are interested in the distance formula terms for such pairs $p,q$.

\begin{lem}[Knowing who your friends are]\label{lem:friendly}
Let $U,V\in\mathfrak S$ and let $p\in\gate_{\mathbf P_U}(\mathbf P_V)$. For $q=\gate_{\mathbf P_V}(p)$, the following holds.  If $W\in\mathfrak S$ satisfies $\dist_W(p,q)\geq10^3\alpha E$ then $W$ is not friendly to either of $U$ or $V$, and $\dist_W(\rho^U_W,\rho^V_W)\geq500\alpha E$. 
\end{lem}

\begin{proof}
If $W\nest V$ or $W\orth V$, then $\pi_W(p),\pi_W(q)$ coarsely coincide by the definition of gates.  Hence $W$ is not friendly to $V$.  Suppose now that $W$ is friendly to $U$.  Choose $p_0\in \mathbf P_V$ so that $p=\gate_{\mathbf P_U}(p_0)$.  Since $W$ is not friendly to $V$ and $p_0,p\in \mathbf P_V$, the $W$--coordinates of $p_0,q$ both coarsely coincide with $\rho^V_W$.  Hence, since $W$ is friendly to $U$, the $W$--coordinate of $p$ also coarsely coincides with $\rho^V_W$, contradicting $\dist_W(p,q)\geq10^3\alpha E$.  Hence $W$ is not friendly to $U$.  The final assertion follows from the fact that $\dist_W(p,\rho^U_W),\dist_W(q,\rho^V_W)\leq E$.
\end{proof}

\begin{lem}\label{lem:irrelevant_and_unfriendly}
Suppose that $W\in\mathfrak S$ and $x,y\in\cuco X$ satisfy $\dist_W(x,y)\leq100\alpha$, while $\dist_V(x,y)\geq100\alpha$ for some $V\in\mathfrak S$.  Suppose that $W$ is not friendly to $V$.  Then $\dist_W(\rho^V_W,x)\leq200\alpha E$.  
\end{lem}

\begin{proof}
First suppose $W\transverse V$.  The lower bound on $\dist_V(x,y)$ implies that either $\dist_V(x,\rho^W_V)>E$ or $\dist_V(y,\rho^W_V)>E$.  In the first case, an application of consistency yields the desired conclusion.  In the second case, apply consistency and the upper bound on $\dist_W(x,y)$.

Next suppose $V\propnest W$.  If $\dist_W(\rho^V_W,x)\leq200\alpha E$, then the bound on $\dist_W(x,y)$ implies that geodesics from $\pi_W(x)$ to $\pi_W(y)$ would remain $E$--far from $\rho^V_W$.  But then consistency and bounded geodesic image would imply that $\pi_V(x),\pi_V(y)$ lie $\leq E$--close, a contradiction.
\end{proof}

\begin{lem}[Uniqueness]\label{lem:cone_off_uniqueness}
For all $\kappa\geq0$, there exists $\theta=\theta(\kappa)$ such that for all $x,y\in\cox_{\mathfrak U}$ with $\hat\dist(x,y)\geq\theta$, there exists $U\in\mathfrak S-\mathfrak U$ such that $\dist_U(x,y)\geq\kappa$.
\end{lem}

\begin{proof}
Let $x,y\in\cuco X$ and let $M=\max_{V\in\mathfrak S-\mathfrak U}\dist_V(x,y)+1$.  We may assume $\alpha\geq E$.

We declare $U\in\mathfrak S$ to be \emph{relevant} if $\dist_U(x,y)\geq100\alpha$.  Let $\mathfrak R^{max}$ be the set of relevant $T\in\mathfrak U$ not properly nested into any relevant element of $\mathfrak U$.

\textbf{Counting and ordering relevant elements:}  By Lemma~\ref{lem:passing_up}, there exists $N=N(100\alpha+\kappa)$ so that if $V_1,\ldots,V_{N+1}\in\mathfrak R^{max}$, then there exists $T\in\mathfrak S$ so that $\dist_T(x,y)\geq100\alpha+\kappa$ and $V_i\propnest T$ for some $i$.  The latter property would ensure that $T\in\mathfrak S-\mathfrak U$, since $\mathfrak R^{max}$ consists of maximal relevant elements of $\mathfrak U$.  Now, if there is such a $T$, then we are done: we have found $T\in\mathfrak S-\mathfrak U$ with $\dist_T(x,y)\geq\kappa$.  Hence we may assume that $|\mathfrak R^{max}|\leq N$, where $N\geq1$ depends only on $(\cuco X,\mathfrak S)$, the constant $\alpha$, and the desired $\kappa$.  

By definition, if $U,V\in\mathfrak R^{max}$, then $U\transverse V$ or $U\orth V$.  Hence, let $U\preceq V$ if either $U=V$ or $U\transverse V$ and $\dist_U(y,\rho^U_V)\leq E$.  As discussed in Section~\ref{subsubsec:partial_order}, $\preceq$ is a partial ordering on $\mathfrak R^{max}$, and $U,V$ are $\preceq$--incomparable if and only if they are orthogonal.  Let $V_1,\ldots,V_k$, with $k\leq N$, be the elements of $\mathfrak R^{max}$, numbered so that $i\leq j$ if $V_i\preceq V_j$.

\textbf{A sequence to estimate $\hat\dist(x,y)$:} The idea is to jump from $x$ to $P_{V_1}$, then from $P_{V_1}$ to $P_{V_2}$ and so on until we get to $y$. The most ``efficient'' way of jumping between product regions is described in Lemma \ref{lem:friendly}, which justifies the definition of the following sequence of points.\footnote{It seems natural to take $x_i=\pi_{V_i}(y)$, $x'_i=\pi_{V_{i+1}}(x)$, but in certain situations this would create extraneous distance formula terms between $x_i,x'_i$, namely when there exists $U\nest V_i,V_{i+1}$.} Let $x_0=x,x_0'=\gate_{\mathbf P_{V_1}}(x),x_k=\gate_{\mathbf P_{V_k}}(x),x'_k=\gate_{\mathbf P_{V_k}}(y)$ and, for $1\leq i\leq k-1$, let $x_i\in\gate_{\mathbf P_{V_i}}(\mathbf P_{V_{i+1}})$ and $x'_i=\gate_{\mathbf P_{V_{i+1}}}(x_i)$. 

There exists $e\in\mathbf E_{V_{i+1}}$ and $f\in\mathbf F_{V_{i+1}}$ so that $x'_i\in\mathbf F_{V_{i+1}}\times\{e\}$ and $x_{i+1}\in\{f\}\times\mathbf E_{V_{i+1}}$.  Let $z_{i+1}=\gate_{\{f\}\times\mathbf E_{V_{i+1}}}(x'_i)$.  Observe that $\dist'(x'_i,z_{i+1})=1$.

\textbf{Bounding $\dist_U(x_i,x'_i)$:} If $\dist_U(x_i,x_{i}')\geq10^3\theta E$ for some $U\in\mathfrak S$, then $V_i\transverse V_{i+1}$ by Lemma~\ref{lem:friendly} and~\cite[Lemma~2.11]{BehrstockHagenSisto:HHS_I}.  We now bound $\dist_U(x_i,x'_i)$ for each of the possible types of $U\in\mathfrak S$.

First suppose that $U\in\mathfrak S-\mathfrak U$.  Then $U\not\nest V_i$ and $U\not\nest V_{i+1}$ since $\mathfrak U$ is closed under nesting.  If $V_i\propnest U$, then since $V_i$ is $100\alpha>E$--relevant for $x,y$, consistency and bounded geodesic image imply that $\rho^{V_i}_U$ lies $E$--close to any geodesic in $\fontact U$ from $\pi_U(x)$ to $\pi_U(y)$.  The same is true for $\rho^{V_{i+1}}_U$ if $V_{i+1}\propnest U$.  If $V_i\transverse U$, then consistency implies that $\rho^{V_i}_U$ lies $E$--close to $\pi_U(x)$ or $\pi_U(y)$, so that $\rho^{V_i}_U$ again lies $E$--close to any geodesic from $\pi_U(x)$ to $\pi_U(y)$.  Hence, if neither of $V_i,V_{i+1}$ is orthogonal to $U$, then $\dist_U(x_i,x'_i)\leq 2(E+\alpha)+M\leq10^3\alpha EM$.  If $U\orth V_i$, then Lemma~\ref{lem:friendly} implies that $\dist_U(x'_i,x_i)\leq10^3\alpha E$.  We conclude that $\dist_U(x'_i,x_i)\leq10^3\alpha EM$ whenever $U\in\mathfrak S-\mathfrak U$.

Next, suppose $U\in\mathfrak U$ and $1\leq i\leq k-1$.  If $\dist_U(x_i,x'_i)>10^3\alpha E$, then Lemma~\ref{lem:friendly} implies that $U$ is not friendly to $V_i$ or $V_{i+1}$.  Moreover, since $\dist_U(\rho^{V_i}_U,\rho^{V_{i+1}}_U)\geq500\alpha E$ by the same lemma, Lemma~\ref{lem:irrelevant_and_unfriendly} implies that $U$ is relevant, so $U\nest U'$ for some $\nest$--maximal relevant $U'\in\mathfrak U$.  Now, $U\transverse V_i,V_{i+1}$ and $U\nest U'$, so $U'\not\in\{V_i,V_{i+1}\}$.  Similarly, we cannot have $U'\orth V_i,V_{i+1}$.  Finally, $V_i,V_{i+1},U'$ are pairwise $\nest$--incomparable, so all are in $\mathfrak R^{max}$ and are pairwise $\preceq$--comparable.  Note that we can extend $\preceq$ to $\mathfrak R^{max}$ and observe that $\preceq$ is a partial order on $\{U,V_i,V_{i+1}\}$.

If $U'\preceq V_i$, then by definition $\dist_{U'}(y,\rho^{V_i}_{U'})\leq E$.  Since $U'$ is relevant, we have $\dist_{U'}(x,\rho^{V_i}_{U'})>E$, so consistency implies that $\dist_{V_i}(x,\rho^{U'}_{V_i})\leq E$.  Hence $\dist_{V_i}(\rho^{U'}_{V_i},y)\geq50E$.  Definition~\ref{defn:space_with_distance_formula}.\eqref{item:dfs_transversal} implies that $\dist_{V_i}(\rho^{U'}_{V_i},\rho^U_{V_i})\leq E$, so consistency implies that $U\prec V_i$.  Since $U\prec V_{i+1}$, by transitivity of $\prec$, we have that $\rho^{V_{i+1}}_U$ coarsely coincides with $\pi_U(y)$.  But then $\dist_{U}(\rho^{V_i}_U,\rho^{V_{i+1}}_U)\leq 2E$, a contradiction.  A similar argument rules out $V_{i+1}\prec U'$, whence $V_i\prec U'\prec V_{i+1}$.  However, this contradicts the way we numbered the elements of $\mathfrak R^{max}$.  Thus $\dist_U(x_i,x_i')\leq 10^3\alpha E$, as desired.

It remains to bound $\dist_U(x_0,x'_0)$ for $U\in\mathfrak U$ (the case $i=k$ is identical to the case $i=0$).  Suppose that $\dist_U(x_0,x'_0)\geq10^3\alpha E$.  The definition of the gate ensures that we cannot have $U\nest V_1$ or $U\orth V_1$, so $U$ is not friendly to $V_1$.  Moreover, $\dist_U(x_0,\rho^{V_1}_U)\geq500\alpha E$ since $\rho^{V_1}_U$ coarsely coincides with $\pi_U(x_0')$.  If $U$ is irrelevant, then Lemma~\ref{lem:irrelevant_and_unfriendly} implies that $\dist_U(x_0,\rho^{V_1}_U)\leq200\alpha E$, a contradiction, so $U$ is relevant and $U\prec V_1$.  Also, $U$ is nested in some $U'\in\mathfrak R^{max}$.  Since $U\nest U'$ and $U$ is not friendly to $V_1$, we have that $U,U'\transverse V_1$ and $U',V_1$ are $\prec$--comparable.  Thus $V_1\prec U$.  Since $\rho^U_{V_1},\rho^{U'}_{V_1}$ coarsely coincide, we have that $\rho^{U'}_{V_1}$ is far from $x_0$, so that $U'\prec V_1$, which is impossible.  Hence $\dist_U(x_0,x'_0)\leq 10^3\alpha E$.

For any $\mu'\geq 10^3\alpha E$, we have shown that $\dist_U(x_i,x'_i)\leq\mu' M$ when $U\in\mathfrak S-\mathfrak U$ and $\dist_U(x_i,x'_i)\leq\mu'$ when $U\in\mathfrak U$, for $0\leq i\leq k$.

\textbf{Bounding $\hat\dist(x,y)$:} We have produced a uniform constant $\mu'$ so that $\dist_U(x_i,x'_i)\leq\mu' M$ when $U\in\mathfrak S-\mathfrak U$ and $\dist_U(x_i,x'_i)\leq\mu'$ when $U\in\mathfrak U$, for $0\leq i\leq k$.  Hence, by the distance formula (Theorem~\ref{thm:distance_formula}) with threshold $\mu'+1$, we have $\dist'(x_i,x'_i)\leq\mu NM$ for some uniform $\mu$.  Thus $\sum_{i=0}^k\dist'(x_i,x'_i)\leq (k+1)\mu NM\leq 2\mu N^2M.$  (Recall that $N$ depends only on $(\cuco X,\mathfrak S)$, the set $\mathfrak U$, and the input $\kappa$.)  

Now, 
\begin{eqnarray*}
\hat\dist(x,y)&\leq&\sum_{i=0}^k\dist'(x_i,x'_i)+\sum_{i=0}^{k-1}[\dist'(x'_i,z_{i+1})+\hat\dist(z_{i+1},x_{i+1})]\\
      &\leq& 2\mu N^2M + N+\sum_{i=0}^{k-1}\hat\dist(z_{i+1},x_{i+1}).\\
\end{eqnarray*}

Fix $0\leq i\leq k-1$, let $\mathbf E=\mathbf E_{V_{i+1}}$ and $\mathbf F=\mathbf F_{V_{i+1}}$, for convenience.  Consider the hierarchical space $(\mathbf E,\mathfrak S^\orth_{V_{i+1}})$, where $\mathfrak S^{\orth}_{V_{i+1}}$ consists of all those $U\in\mathfrak S$ with $U\orth V_{i+1}$ together with some $A\in\mathfrak S$ such that $A\propnest S$ and each $U$ orthogonal to $V_{i+1}$ satisfies $U\nest A$.  Let $\mathfrak U_{\mathbf E}=\mathfrak U\cap\mathfrak S_{V_{i+1}}^\orth$ and consider the factored space $(\coneoff{\mathbf E}_{\mathfrak U_{\mathbf E}},\mathfrak S^\orth_{V_{i+1}}-\mathfrak U_{\mathbf E})$, whose metric we denote $\hat\dist_{\mathbf E}$.  Observe that $\hat\dist(z_{i+1},x_{i+1})\leq \epsilon\hat\dist_{\mathbf E}(z_{i+1},x_{i+1})+\epsilon$ for some uniform $\epsilon$, since $\mathbf E\to\cuco X$ is a uniform quasi-isometric embedding and any two points lying on a parallel copy of some $\mathbf F_U$ in $\mathbf E$ also lie on such a parallel copy in $\cuco X$.

Now, by induction on complexity, there exists a ``uniqueness function'' $f:\naturals\to\naturals$, independent of $i$, so that $(\coneoff{\mathbf E}_{\mathfrak U_{\mathbf E}},\mathfrak S^\orth_{V_{i+1}}-\mathfrak U_{\mathbf E})$ has the following property: if $e,e'\in\mathbf E$, then $$\hat\dist(e,e')\leq \epsilon f\left(\max_{U\in\mathfrak S^\orth_{V_{i+1}}-\mathfrak U_{\mathbf E}}\dist_U(e,e')\right)+\epsilon.$$
Indeed, in the base case, either $\mathfrak S^\orth_{V_{i+1}}-\mathfrak U_{\mathbf E}=\emptyset$, and $\mathbf E$ is uniformly bounded in $\cuco X$ (and hence its $\hat\dist$--diameter is uniformly bounded) or $\mathfrak U_{\mathbf E}=\emptyset$ and $\hat\dist_{\mathbf E}$ coarsely coincides with $\dist$ on $\mathbf E$, whence $f$ exists by uniqueness in $\mathbf E$ (with metric $\dist$ and HS structure $\mathfrak S^\orth_{V_{i+1}}$).

\renewcommand{\qedsymbol}{$\blacksquare$}
\begin{claim}\label{claim:baby_hierarchy}
There exists $\eta=\eta(\HHS X S)$ such that for all $U\in\mathfrak S-\mathfrak U$ and $1\leq i\leq k-1$, there exists $x_i\in\gate_{P_{V_i}}(P_{V_{i+1}})$ so that $\pi_U(x_i)$ lies $\eta$--close to a geodesic from $\pi_U(x)$ to $\pi_U(y)$.
\end{claim}

\begin{proof}[Proof of Claim~\ref{claim:baby_hierarchy}]
Theorem~\ref{thm:hierarchy_paths} provides a $D$--discrete $D$--hierarchy path $\gamma$ joining $y$ to $x$, where $D$ depends only on $\HHS X S$.  We may assume that $D\leq\alpha$, since $\alpha$ was chosen in advance in terms of $\HHS X S$ only.  The proof of Proposition~5.16 of~\cite{BehrstockHagenSisto:HHS_II} (which does not use hyperbolicity of the various $\fontact U$), provides a constant $\eta'$ so that, for each $i$, there exists a maximal subpath $\gamma_i$ of $\gamma$ lying in $\neb_{\eta'}(\mathbf P_{V_i})$, with initial point $x'_i$.  Moreover, $\pi_U(x'_i)$ uniformly coarsely coincides with $\rho^{V_i}_U$ when $V_i\propnest U$ or $V_i\transverse U$.  Let $y_i\in\mathbf P_{V_{i+1}}$ lie $\eta''$--close to the terminal point of $\gamma_{i+1}$.  We claim that $\dist(\gate_{\mathbf P_{V_i}}(y_i),x'_i)$ is uniformly bounded.  Indeed, by definition $\pi_U(\gate_{\mathbf P_{V_i}}(y_i))$ coarsely coincides with $\rho^{V_i}_U$, and hence with $\pi_U(x)$, when $U\transverse V_i$ or $V_i\nest U$, and coincides with $\pi_U(y_i)$ when $U\nest V_i$ or $U\orth V_i$.  Our choice of $x'_i,y_i$ ensures that $\dist_U(x'_i,y_i)$ is uniformly bounded for such $U$, so our claim follows from the distance formula.  Taking $x_i=\gate_{\mathbf P_{V_i}}(y_i)$ completes the proof, since $x'_i$, and hence $x_i$, lies uniformly close to any geodesic from $\pi_U(x)$ to $\pi_U(y)$ in any $\delta$--hyperbolic $\fontact U$, by the definition of a hierarchy path.
\end{proof}

We now choose specific values of $x_i,x'_i,z_i$ satisfying the above defining conditions.  First, as before, $x_0=x,x'_0=\gate_{\mathbf P_{V_1}}(x)$, while $x_k=\gate_{\mathbf P_{V_k}}(y)$ and $x'_k=y$.  For $1\leq i\leq k-1$, let $x_i$ be a point provided by Claim~\ref{claim:baby_hierarchy}.  Then let $x'_i=\gate_{P_{V_{i+1}}}(x_i)$ for $1\leq i\leq k-1$, as before, and define the points $z_i$ as above.

\begin{claim}\label{claim:gate_bound}
There exists a function $f'\co\naturals\to\naturals$, independent of $i$, so that $\dist_U(z_{i+1},x_{i+1})\leq f'(M)$ for all $U\in\mathfrak S^\orth_{V_{i+1}}-\mathfrak U_{\mathbf E}$.
\end{claim}

\begin{proof}[Proof of Claim~\ref{claim:gate_bound}]
Let $U\in\mathfrak S^\orth_{V_{i+1}}-\mathfrak U_{\mathbf E}$.  By the definition of gates, for $1\leq i\leq k-1$, we have $\dist_U(z_{i+1},x_i)\leq\alpha$ so, since $\pi_U(x_i)$ lies $\eta$--close to a geodesic from $\pi_U(x)$ to $\pi_U(y)$, we have $\dist_U(z_{i+1},\{x,y\})\leq \eta+\alpha$.  Likewise, $\dist_U(x_{i+1},\{x,y\})\leq \eta$.  Hence $\dist_U(z_{i+1},x_{i+1})\leq M+2\eta+\alpha$.
\end{proof}

\renewcommand{\qedsymbol}{$\Box$}

Claim~\ref{claim:gate_bound} and the above discussion imply that, if $M\leq\kappa$, we have $$\hat\dist(x,y)\leq 2\mu N^2\kappa + N + \epsilon Nf(f'(\kappa)) + N\epsilon,$$ which completes the proof.
\end{proof}

\setcounter{claim}{0}

\subsection{$\fontact(*)$ as a coarse intersection graph}\label{subsec:coarse_inter}
We conclude this section by highlighting 
a particularly interesting application of 
Proposition~\ref{prop:cone_off_bottom}, one of the tools we developed 
for proving the results about asymptotic dimension.
\footnote{A proof of the results in this subsection appeared in the first version of 
\cite{BehrstockHagenSisto:HHS_II} using techniques which we have now  
generalized to prove Proposition~\ref{prop:cone_off_bottom}.}

\begin{cor}\label{cor:cone_off_all_F}
    Given a relatively HHS $\HHS X S$, the space $\coneoff{\cuco 
    X}_{\mathfrak S - \{S\}}$ is quasi-isometric to $\pi_S(\cuco X)\subseteq\fontact S$, 
    where $S\in\mathfrak S$ is $\nest$--maximal. 
\end{cor}

\begin{proof}
By Proposition~\ref{prop:cone_off_bottom}, $(\cox_{\mathfrak S-\{S\}},\{S\})$ is a hierarchically hyperbolic space, and the claim follows from the distance formula (Theorem~\ref{thm:distance_formula}).
\end{proof}

\begin{rem}[Coning off $\mathbf P_U$]\label{rem:cone_off_product_regions}
If we had constructed $\cox_{\mathfrak S-\{S\}}$ by ``coning off'' $\mathbf P_U$ for each $U\in\mathfrak S-\{S\}$, instead of coning off each parallel copy of each $\mathbf F_U$, then Corollary~\ref{cor:cone_off_all_F} would continue to hold.
\end{rem}

In many examples of interest, $\pi_S$ is coarsely surjective, so that
Corollary~\ref{cor:cone_off_all_F} yields a quasi-isometry
$\cox_{\mathfrak S-\{S\}}\to\fontact S$.  Moreover, if $\HHS X S$ is
an HHS, then (as described
in~\cite{BehrstockHagenSisto:HHS_II,DurhamHagenSisto:HHS_IV}), $\cuco
X$ admits an HHS structure obtained by replacing each $\fontact U$
with a hyperbolic space quasi-isometric to $\pi_U(\cuco X)$, so in
particular $\fontact S$ becomes quasi-isometric to the space obtained
by coning off each parallel copy of each $\mathbf F_U, U\neq S$.  If,
as is the case for hierarchically hyperbolic \emph{groups} $\HHS G S$,
the parallel copies of the various $\mathbf F_U$ coarsely cover $\cuco
X$, this provides a hierarchically hyperbolic structure in which
$\fontact S$ is a coarse intersection graph of the set of $\mathbf
F_U$ with $U\propnest S$ for which there is no $V$ with $U\propnest
V\propnest S$.  This is a coarse version of what happens, for example,
when $\cuco X$ is a CAT(0) cube complex with a factor system and we
can take $\fontact S$ to be the \emph{contact graph} of $\cuco X$
(see~\cite{Hagen:quasi_arb,BehrstockHagenSisto:HHS_I}).

\setcounter{claim}{0}

\section{Asymptotic dimension of the $\fontact U$}\label{sec:tight}
In this section $\HHS X S$ is a relatively hierarchically hyperbolic
space with the additional property that $\cuco X$ is a \emph{uniformly
locally-finite discrete geodesic space}, i.e.,
\begin{enumerate}
 \item there exists $r_0>0$ so that $\dist(x,y)\geq r_0$ for all distinct $x,y\in\cuco X$;\label{item:discrete_1}
 \item there is a function $p\colon [0,\infty)\to[0,\infty)$ so that $|B(x,r)|\leq p(r)$ for all $x\in\cuco X$;\label{item:proper_1}
 \item there exists $r_1$ so that for all $x,y\in\cuco X$, there
 exists $n$ and $\gamma\colon \{0,n\}\to\cuco X$ so that
 $\gamma(0)=x,\gamma(n)=y$, and
 $\dist(x,y)=\sum_{i=0}^{n-1}\dist(\gamma(i),\gamma(i+1))$, and
 $\dist(\gamma(i),\gamma(i+1))\leq r_1$ for all
 $i$.\label{item:geodesic_1}
\end{enumerate}
If $\cuco X$ satisfies \eqref{item:discrete_1} and \eqref{item:geodesic_1} (but not necessarily \eqref{item:proper_1}), then $\cuco X$ is a \emph{$(r_0,r_1)$--discrete geodesic space}.

The following notion is motivated by work of Bowditch; see~\cite[Section 3]{Bowditch:tight}.

\begin{defn}[Tight space]\label{defn:tight}
The $\delta$--hyperbolic space $F$ is \emph{$(C,K)$--tight} if there 
exists a map $\beta\colon F^2\to 2^F$ so that:
\begin{enumerate}
\item for every $x,y\in F$, we have $\dist_{Haus}([x,y],\beta(x,y))\leq C$, where $[x,y]$ is any geodesic from $x$ to $y$;\label{item:beta_close_to_geodesic}
\item for every $x,y,z\in F$ with $y$ on a geodesic from $x$ to $z$, if $\dist(y,\{x,z\})\geq r+C$ for some $r\in\mathbb R^+$ then $B(y,2\delta+2C)\cap \bigcup_{x'\in B(x,r),z'\in B(z,r)} \beta(x',z')$ has cardinality at most $K$.\label{item:property_B}
\end{enumerate}
\end{defn}

Notice that for $x,y,z,r$ as in \ref{item:property_B}, if $x'\in B(x,r),z'\in B(z,r)$ then $\beta(x',z')$ intersects $B(y,2\delta+2C)$ by $\delta$--hyperbolicity and part~\ref{item:beta_close_to_geodesic} of Definition~\ref{defn:tight}.

Let $\xi$ denote the complexity of $\mathfrak S$, i.e., the length of a longest $\nest$--chain.  By Definition~\ref{defn:space_with_distance_formula}.\eqref{item:dfs_complexity}, $\xi<\infty$.  The aim of this section is to prove:

\begin{thm}[Existence of tight geodesics]\label{thm:tight_geodesics}
Let $(\cuco X,\mathfrak S)$ be a $\delta$--relatively hierarchically
hyperbolic space, where $\cuco X$ is a uniformly locally-finite
discrete geodesic space.  Suppose moreover that $\pi_U:\cuco X\to\fontact U$ is uniformly coarsely surjective, where $U$ varies over all elements of $\mathfrak S$ with $\fontact U$ a $\delta$--hyperbolic space.  Then there exist $C,K\geq 0$ so that
$\fontact U$ is $(C,K)$--tight for every $U\in\mathfrak S$ for which
$\fontact U$ is $\delta$--hyperbolic.
\end{thm}

\begin{proof}
Throughout this proof we use the 
identification of $\fontact S$ with the coned-off space $\cox$, as 
established in Proposition~\ref{prop:cone_off_bottom}.  Our assumption on coarse surjectivity of the projections $\pi_U$ implies that, for each $U\in\mathfrak S$ with $\fontact U$ a $\delta$--hyperbolic space, we may (by an initial change in the constants from Definition~\ref{defn:space_with_distance_formula}), assume that $\pi_U$ is actually surjective.  

Fix a constant $D$, as provided by Theorem~\ref{thm:hierarchy_paths}, 
so that every pair of points of $\cuco X$ can be joined by
a $D$--hierarchy path.  For
$S$ the $\nest$--maximal element of $\mathfrak S$, we will show that
$\fontact S$ is $(C,K)$--tight, where $C,K$ depend only on $D,E$, the
complexity $\xi$ of $\mathfrak S$, and the function $p$ which 
quantifies the local-finiteness.  To see that
this suffices, recall that for each $U\in\mathfrak S$, there is a
hierarchically quasiconvex subspace $\mathbf F_U$ with a relatively
hierarchically hyperbolic structure $(\mathbf F_U,\mathfrak S_U)$ in
which $U$ is $\nest$--maximal.

For $M\geq 0$ and $x,y\in \cuco X$, define
$$\beta_M(x,y)=\{z\in\cuco X: \hat\dist(z,[x,y])\leq M, \dist_U(z,\{x,y\})\leq M\ \forall U\in\mathfrak S- \{S\}\}.$$

For sufficiently large $M$, the map $\beta_M\colon (\fontact S)^2\to 2^{\fontact S}$ satisfies property~\eqref{item:beta_close_to_geodesic},  by the definition and Claim~\ref{claim:geodesic_close_to_beta} below.

\begin{claim}\label{claim:geodesic_close_to_beta}
 For each sufficiently large $M$ there exists $K_1$ so that for each $x,y\in\cuco X$ and $z\in[x,y]$ we have $\hat\dist(z,\beta_M(x,y))\leq K_1$.
\end{claim}

\renewcommand{\qedsymbol}{$\blacksquare$}

\begin{proof}
 We distinguish two cases.  First, suppose that there exists
 $U\in\relevant(x,y,10E)$ for which $\hat\dist(\rho^U_S,z)\leq 10DE$, and
 take a $\nest$--maximal $U$ with such property.  Then consider
 $z'=\gate_{P_U}(x)$.  Clearly, $\hat\dist(z,z')$ is uniformly
 bounded, and we now show $z'\in\beta_M(x,y)$, provided that $M$ is
 large enough.  In order to do so, we uniformly bound
 $\dist_V(\{x,y\},z)$ for each $V\in\mathfrak S- \{S\}$.  If $V$ is
 either nested into $U$ or orthogonal to $U$, then we are done by the
 definition of gate.  Otherwise, $\pi_V(z')$ coarsely coincides with
 $\rho^U_V$, so we have to show that either $\pi_V(x)$ or $\pi_V(y)$
 coarsely coincides with $\rho^U_V$.  Suppose by contradiction that
 this is not the case.  If $U\transverse V$, then consistency implies
 that $\pi_U(x),\pi_U(y)$ both coarsely coincide with $\rho^V_U$, so
 that $\dist_U(x,y)\leq 3E$ contradicting the choice of $U$.  If
 $U\propnest V$, then any geodesic from $\pi_V(x)$ to $\pi_V(y)$ stays
 far from $\rho^U_V$ since, by maximality of $U$, we have
 $\dist_V(x,y)<10E$.  In particular, by bounded geodesic image and
 consistency we have $\dist_U(x,y)<10E$, a contradiction.
 
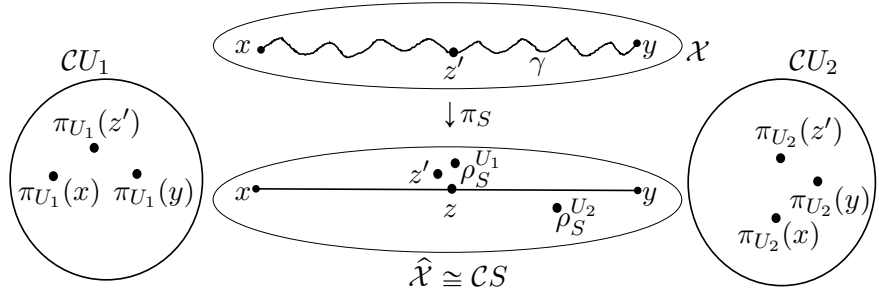
\begin{figure}[h]
\begin{overpic}[width=0.75\textwidth]{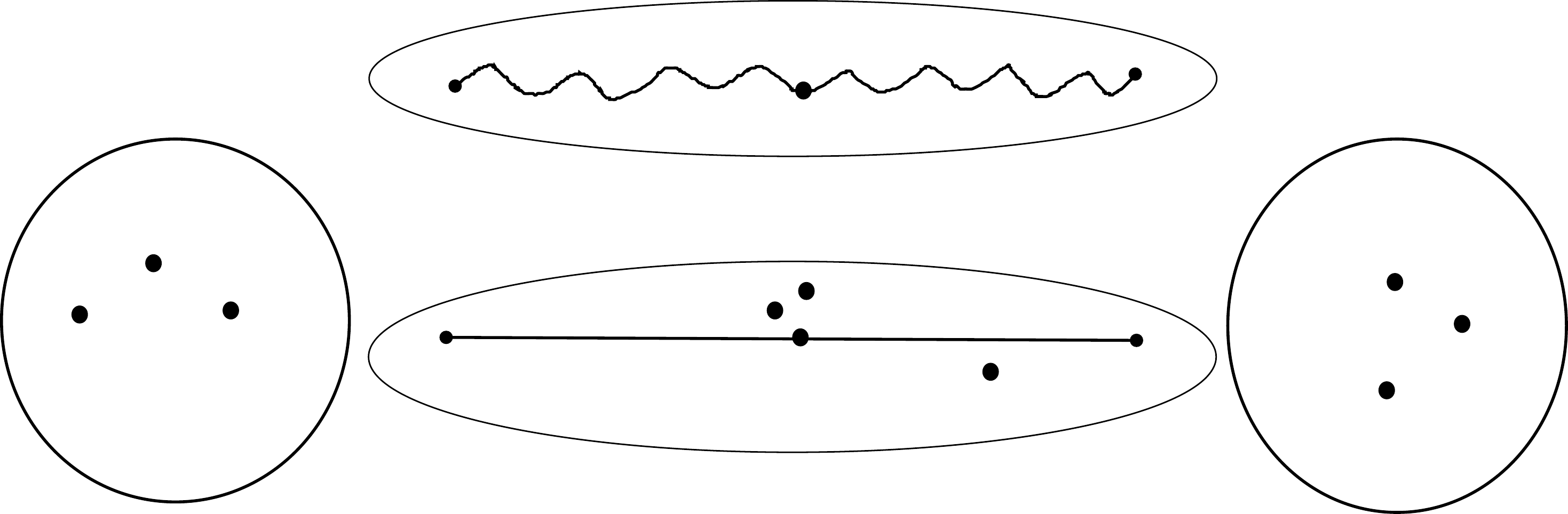}
\put(50,19){$\downarrow$}
\put(78,26){$\cuco X$}
\put(6,25){$\fontact U_1$}
\put(90,25){$\fontact U_2$}
\put(46,0){$\cox\cong \fontact S$}
\put(52,19){$\pi_S$}
\put(26,27){$x$}
\put(50,24){$z'$}
\put(73,27){$y$}
\put(60,25){$\gamma$}
\put(73,10){$y$}
\put(26,10){$x$}
\put(50,8){$z$}
\put(46,12){$z'$}
\put(52,12.5){$\rho^{U_1}_S$}
\put(63,7){$\rho^{U_2}_S$}
\put(1,10){$\pi_{U_1}(x)$}
\put(12,10){$\pi_{U_1}(y)$}
\put(5,18){$\pi_{U_1}(z')$}
\put(84,5){$\pi_{U_2}(x)$}
\put(90,9){$\pi_{U_2}(y)$}
\put(86,17){$\pi_{U_2}(z')$}
\end{overpic}
\caption{The second case in the proof of Claim~\ref{claim:geodesic_close_to_beta} when each relevant $U$ has $\rho^U_S$ far from $z$.  The subcases in the proof correspond to the $U_1,U_2\in\mathfrak S$ shown.}\label{fig:tight_claim_1_case_2} 
\end{figure}

 Suppose now that there does not exist any $U\in\relevant(x,y,10E)$ so that $\hat\dist(\rho^U_S,z)\leq 10DE$. Consider a hierarchy path $\gamma$ from $x$ to $y$. Since $\pi_S$ is coarsely Lipschitz, there exists $z'\in\gamma$ so that $\hat\dist(z,z')\leq 5DE$. We claim $z'\in\beta_M(x,y)$ for sufficiently large $M$. In fact, for any $U\in\mathfrak S-\{S\}$ (notice that if such $U$ exists then $\hat {\cuco X}$ is $\delta$--hyperbolic) either $\hat\dist(\rho^U_S,z')\leq 5E$, so $\hat\dist(\rho^U_S,z)\leq 5DE+5E\leq 10DE$, in which case $U$ is irrelevant and we are done by hypothesis, or $\rho^U_S$ lies $5E$--far from a geodesic from $\pi_S(z')$ to one of $\pi_S(x),\pi_S(y)$.  In this case, we can apply bounded geodesic image to conclude.  See Figure~\ref{fig:tight_claim_1_case_2}.
\end{proof}

We now prove Property \ref{item:property_B}. Let $x,y,z\in \cuco X$ with $y$ on a (discrete) $\hat d$--geodesic from $x$ to $z$, and suppose $\hat\dist(y,\{x,z\})\geq r+K_1$ for some $r\geq K_1+2\delta+100E$.  Let $$T=B^{\cox}(y,2\delta+2K_1)\cap \left(\bigcup_{x'\in B^{\hat{\cuco X}}(x,r),z'\in B^{\hat{\cuco X}}(z,r)} \beta_M(x',z')\right).$$
 
\noindent Moreover, let $\mathcal U=\mathcal U(x,y,z)$ be the set of all $U\in\mathfrak S$ satisfying the following conditions:
\begin{itemize}
 \item $\dist_U(x,z)\geq 10M+10^3E$,
 \item $U\neq S$ and $\dist_S(\rho^U_S,y)\leq 2\delta+2K_1+50E$,
 \item whenever $V\in\mathfrak S$ satisfies $\dist_V(x,z)\geq 10M+10^3E$ and $U\propnest V \propnest S$, we have $\dist_V(\rho^U_V,\{x,z\})\leq 100E+10M$.
\end{itemize}
We will see in what follows that for $y'\in T$ it is sufficient to have information about $\pi_U(y')$ for $U\in\mathcal U$ to coarsely reconstruct all $\pi_V(y')$. Moreover, we will bound the cardinality of $\mathcal U$.

For $y'\in T$, denote $\mathcal U(y')=\{U\in\mathcal U: \dist_U(x,y')\leq M+10E\}$, and let $\mathcal B$ be the set of all subsets of $\mathfrak S$ of the form $\mathcal U(y'),y'\in T$.

\begin{claim}\label{claim:bound_U_i}
There is $K_3=K_3(E,K_1,M)$ with $|\{y'\in T:\mathcal U(y')=\mathcal U_i\}|\leq p(K_3)$ for all $\mathcal U_i\in\mathcal B$.
\end{claim}

\begin{proof}[Proof of Claim~\ref{claim:bound_U_i}]
Fix $y',y''\in T$ with $\mathcal U(y')=\mathcal U(y'')$. We bound $\dist(y',y'')$ by bounding $\dist_U(y',y'')$ for all $U\in \mathfrak S$.

Let $U\in\mathfrak S$.  If $U\in\mathcal U$, then $\dist_U(y',y'')\leq 2M+10E$ since $U(y')=U(y'')$ (and the fact that each of $\pi_U(y'),\pi_U(y'')$ coarsely coincides with either $\pi_U(x)$ or $\pi_U(z)$ in $\fontact U$). We now analyze the other cases, but for technical reasons we change the constants from the definition of $\mathcal U$.

If $U=S$, then $\dist_U(y',y'')\leq 4\delta+4K_1$. Assume $U\neq S$ from now on.

If $\dist_U(x,z)< 10M+150E$, then $\dist_U(y',y'')\leq 20M+200E$.

If $\dist_S(\rho^U_S,y)> 2\delta+2K_1+10E$, then no geodesic from $y'$ to $y''$ in $\cox$ passes $E$--close to $\rho^U_S$, so consistency and bounded geodesic image again yield $\dist_U(y',y'')\leq 10E$.

Finally, suppose that $\dist_U(x,z)\geq 10M+150E$ and $\dist_S(\rho^U_S,y)\leq 2\delta+2K_1+10E$, but that there exists $V\in\mathfrak S$ satisfying $\dist_V(x,z)\geq 10M+10^3E$, $U\propnest V \propnest S$ and $\dist_V(\rho^U_V,\{x,z\})> 50E+10M$. Consider a $\nest$-maximal $V$ with this property. We claim that $V\in\mathcal U$. In fact, $\rho^V_S$ coarsely coincides with $\rho^U_S$, yielding the second condition in the definition of $\mathcal U$. Moreover, for any $W\in\mathfrak S$ with $V\propnest W\propnest S$, we have $\dist_W(\rho^V_W,\{x,z\})\leq 100E+10M$, for otherwise $V$ would not be maximal (we are once again using that $\rho^V_W$ coarsely coincides with $\rho^U_W$).

Now, since $\mathcal U(y')=\mathcal U(y'')$ we have that $\pi_V(y'),\pi_V(y'')$ are both close to one of $\pi_V(x),\pi_V(z)$. Hence, $\pi_U(y')$ must $10E$--coarsely coincide with $\pi_U(y'')$ because geodesics in $\fontact V$ from $\pi_V(y')$ to $\pi_V(y'')$ stay $E$--far from $\rho^U_V$.

We conclude that for all $U\in\mathfrak S$, we have, say, $\dist_U(y',y'')\leq 500ME\delta K_1$, so the distance formula (Theorem~\ref{thm:distance_formula}) provides a uniform $K_3=K_3(M,E,K_1)$ so that $\dist_{\cuco X}(y',y'')\leq K_3$, and the claim follows from the definition of $p$.
\end{proof}

\begin{claim}\label{claim:tight_claim_2}
There exists $K_4=K_4(E,M)$ so that $|\mathcal B|\leq 2^{K_4^\xi}$.
\end{claim}

\begin{proof}
By definition, if $U\in\mathcal U$, then $\dist_S(\rho^U_S,y)\leq 2\delta+2K_1+50E$.  Choose $a,b\in\cuco X$ such that $a$ is $E$--close to a $\cox$--geodesic from $y$ to $z$ and satisfies $\hat\dist(y,a)\in[2\delta+2K_1+60E,2\delta+2K_1+80E]$, while $b$ is $E$--close to a $\cox$--geodesic from $x$ to $y$ and satisfies $\hat\dist(y,b)\in[2\delta+2K_1+60E,2\delta+2K_1+80E]$.  Then $\rho^U_S$ does not lie $E$--close to a geodesic from $a$ to $z$ or from $x$ to $b$, so consistency and bounded geodesic image yield $\dist_U(a,z)\leq E$ and $\dist_U(x,b)\leq E$.  Thus $\dist_U(b,a)\geq10M+98E\geq E$, so Lemma~\ref{lem:passing_up} yields $K_5=K_5(M,E)$ so that there are at most $K_5$ such $U$ that are $\nest$--maximal.

Fix $U\in\mathcal U-\{S\}$ to be $\nest$--maximal and consider all $V\in\mathfrak S$ such that $V\propnest U$ and such that the following conditions are satisfied:
\begin{itemize}
 \item $\dist_V(x,z)\geq 10M+10^3E$,
 \item $\dist_S(\rho^V_S,y)\leq 2\delta+2K_1+50E$,
 \item $\dist_U(\rho^V_U,x)\leq 100E+10M$.
\end{itemize}
(Note that any $V\in\mathcal U$ nested into $U$ satisfies either these conditions or the same conditions with $z$ replacing $x$ in the third one.)  By partial realization and the first and third conditions above, there exists $a_U\in\cuco X$ such that $\dist_U(a_U,x)\leq 10M+200E$ and $\rho^U_V$ fails to come $E$--close to every geodesic from $\pi_U(x)$ to $\pi_U(a_U)$ for all $V$ as above.  Hence, for each $V$ as above, bounded geodesic image and consistency imply that $\dist_U(a_U,x)\leq 10M+200E$ and $\dist_V(a_U,x)\geq10M+9E$.  Let $\mathcal V_0$ be the set of all $\nest$--maximal $V\propnest U$ contained in $\mathcal U$.  By Lemma~\ref{lem:passing_up} and maximality, there exists $K_6=K_6(10M+200E)$ so that $|\mathcal V_0|\leq K_6$.  

For each $V\in\mathcal V_0\cap\mathcal U$, choose $a_V$ as above, so that $\dist_V(a_V,x)\leq 10M+200E$ and, any element $W$ of $\mathcal U$ properly nested into $V$ and satisfying $\dist_V(\rho^W_V,x)\leq 100E+10M$, we have $\dist_W(a_V,x)\geq10M+9E$.  Then, exactly as above, we find that there are at most $K_6$ such $W$ that are $\nest$--maximal and properly nested in $V$.  Proceeding inductively, we see that there are at most $K_6^{\xi-1}$ elements $T$ of $\mathcal U$ satisfying $\dist_U(\rho^T_U,x)\leq 100E+10M$ and properly nested into $U$, while an identical discussion bounds the set of $T\in\mathcal U$ properly nested in $U$ and satisfying $\dist_U(\rho^T_U,z)\leq 100E+10M$.  Hence $|\mathcal U|\leq 2K_5K_6^{\xi-1}$, so, letting $K_4=\max\{2K_5,K_6\}$, we get $|\mathcal B|\leq 2^{K_4^\xi}$.  
\end{proof}

\renewcommand{\qedsymbol}{$\Box$}

Claim~\ref{claim:bound_U_i} and Claim~\ref{claim:tight_claim_2} together imply that $|T|\leq p(K_3)\cdot 2^{K_4^\xi}$ for uniform $K_4$, so $\fontact S$ is $(K_1,p(K_3)\cdot 2^{K_4^\xi})$--tight.
\end{proof}

\begin{cor}\label{cor:finite_asdim_curve_graph}
Let $\cuco X$ be a uniformly locally-finite discrete geodesic space and let $(\cuco X,\mathfrak S)$ be $\delta$--relatively hierarchically hyperbolic.  Suppose moreover that $\pi_U:\cuco X\to\fontact U$ is uniformly coarsely surjective, where $U$ varies over all elements of $\mathfrak S$ with $\fontact U$ a $\delta$--hyperbolic space.  Then there exist $\lambda,\mu=\lambda(\delta,D,E,p),\mu(\delta,D,E,p)$, independent of $\xi$, so that $\asdim\fontact U\leq\lambda\cdot 2^{\mu^\xi}$ uniformly, whenever $U\in\mathfrak S$ has the property that $\fontact U$ is $\delta$--hyperbolic. 
\end{cor}

\begin{proof}
By Theorem~\ref{thm:tight_geodesics}, there exist $C,\lambda',\mu$ so that $\fontact U$ is $(C,\lambda'\cdot 2^{\mu^\xi})$--tight for each $U\in\mathfrak S$ with $\fontact U$ a $\delta$--hyperbolic space.  We now argue roughly as in the proof of~\cite[Theorem~1]{BellFujiwara:asdim_curve_graph}.

Fix $r\in\naturals,x_0\in\fontact U,$ and $\ell\in\naturals$ with
$\ell\geq10\delta+C$.  For each $n\geq1$, let $$A_n=\{x\in\fontact
U:10(n-1)(r+\ell)\leq\dist_U(x,x_0)\leq10n(r+\ell)\},$$ so that
$\cup_nA_n=\fontact U$.  Let $S_n=\{x\in\fontact
U:\dist_U(x,x_0)=10n(r+\ell)\}$.  Given $n\geq 3$, define subsets
$B^i_n\subset A_n$ as follows: and for each $s_i\in S_{n-2}$ and 
$x\in A_n$, set $x\in B^i_n$ if and only if there exists
$\beta(x,x_0)$ lying at Hausdorff distance $C$ from a geodesic joining
$x_0$ to $x$ that passes through $s_i$.  For $n\in\{1,2\}$, define 
$B_n^1=A_n$.  Then the sets $\{B_n^i\}$ cover $\fontact U$.  Observe that
$\diam B^i_n\leq100C(r+\ell)$ for each $n,i$.

Hence $\{B^i_n\}$ is a uniformly bounded cover of $\fontact U$.  We
now check that its $\frac{r}{2}$--multiplicity is at most
$2\lambda'\cdot 2^{\mu^\xi}$, from which the claim follows.  Our claim
is a consequence of the following statement: for $n\geq1$, there are
$\leq\lambda'\cdot 2^{\mu^\xi}$ elements of $\{B_n^i\}$ intersecting
any $\frac{r}{2}$--ball $\mathcal B$ in $\fontact U$.

Indeed, the claim is clear when $n\leq 2$, so assume $n\geq3$.  Choose
distinct $B_n^i,B_n^j$ intersecting $\mathcal B$, with the
intersections respectively containing points $y_i,y_j$.  Hence we have
geodesics $[x_0,y_i],[x_0,y_j]$, respectively joining $x_0$ to
$y_i,y_j$ and passing through $s_i,s_j$, and paths
$\beta_i=\beta(x_0,y_i)$ and $\beta_j=\beta(x_0,y_j)$ at Hausdorff
distance $\leq C$ from $[x_0,y_i],[x_0,y_j]$.  Choose
$s_i'\in\beta_i,s_j'\in\beta_j$ with
$\dist_U(s_i,s_i'),\dist_U(s_j,s'_j)\leq C$.  Since
$\dist_U(y_i,y_j)\leq r$, because $y_i,y_j\in\mathcal B$, we see that
$s'_i,s'_j$ both lie in $\bigcup_{x'\in B(x_0,r),z'\in
B(y_i,r)}\beta(x',z')$.

Now, $s_i$ lies on a geodesic from $x_0$ to $y_i$ and
$\dist_U(s_i,\{x_0,y_i\})\geq r+C$ by definition.  Hyperbolicity and
the definition of $S_{n-2}$ imply that $\dist_U(s_i,s_j)\leq 2\delta$,
whence $\dist_U(s_j',si),\dist_U(s_i',s_i)\leq 2\delta+2C$, i.e.,
$s_i',s_j'\in B(s_i,\delta+2C)\cap\bigcup_{x'\in B(x_0,r),z'\in
B(y_i,r)}\beta(x',z')$.  This intersection is bounded by 
Property~\eqref{item:property_B} of tightness, thus 
completing the proof.
\end{proof}

\setcounter{claim}{0}

\section{Asymptotic dimension of ball-preimages under $\psi\co\cuco X\to\cox$}\label{sec:ball_preimage}
Fix a relatively hierarchically hyperbolic space $\HHS X S$.  Let $\mathcal O$ be the set of totally orthogonal subsets of $\mathfrak 
U_{1}\subseteq\mathfrak S$. We set $\asdim\mathcal O$ to be the minimal uniform asymptotic 
dimension of the collection $\{\Pi_{V\in\mathfrak V} \mathbf F_{V}: \mathfrak 
    V\in\mathcal O\}$.  Our goal is to prove:

\begin{prop}\label{prop:build_A}
    For every $R\geq 0$, there exists a 
    hierarchical quasiconvexity function $k$  
    and a control function $f$, 
    so that for every 
    $x_{0}\in\coneoff{\cuco X}$ there exists 
    $A\subset \cuco X$ satisfying:
    \begin{itemize}
        \item $A$ is $k$--hierarchically quasiconvex in $\cuco X$, 
    
         \item $A\supseteq \psi\inv(B(x_{0},R))$,
    
        \item $\asdim A\leq \asdim\mathcal O$ with control function $f$.
    \end{itemize}
\end{prop}

\begin{proof}
 For convenience, fix $x_0\in\coneoff{\cuco X}$ and stipulate that all constants, hierarchical quasiconvexity functions, and control functions below are chosen independently of $x_0$.  Fix $D$ large enough that any pair of points in $\cuco X$ are joined by a $(D,D)$--quasigeodesic.
 
 Set $A_0=\{x_0\}$. We will inductively construct $A_n,k_n,f_n$ so that
 \begin{enumerate}
  \item $A_n\subseteq \cuco X$ is $k_n$--hierarchically quasiconvex,\label{item:hqc}
  \item $\asdim A_n\leq \asdim\mathcal O$ with control function $f_n$,\label{item:asdim_bound}
  \item for each $n\geq 1$, the set $A_n$ contains $\neb_{10D}(A_{n-1})$,\label{item:contains_nhood}
  \item for each $n\geq 1$ and $U\in\mathfrak U_{1}$, 
  the set $A_n$ contains each
  parallel copy of $\mathbf F_U$ which intersects
  $A_{n-1}$.\label{item:parallel_copies}
 \end{enumerate}

 Assuming for the moment that we can construct such $A_n$, then the  
 proposition follows. Indeed, given $R\geq1$, we now show that 
 the set $A_n$ contains
 $\psi\inv(B(x_{0},R))$ for $n\geq 10D^2R$. Let $x\in
 \psi\inv(B(x_{0},R))$, whence, by definition of $\dist'$, there exists a sequence
 $x_0,\dots,x_k=x$ for which $\sum_{i=1}^{k-1}\dist'(x_i,x_{i+1})\leq
 2R$.  Using that $\cuco X$ is a quasigeodesic space, we can suitably
 interpolate between consecutive points in the sequence and find
 another sequence $x_0=y_0,\dots,y_{k'}=x$ so that
 $\sum_{i=1}^{k'-1}\dist'(y_i,y_{i+1})\leq 10D^2R$ and for each $i$
 either $y_i,y_{i+1}$ lie in a common $\mathbf F_U$ for $U\in\mathfrak U_1$ or
 $\dist(y_i,y_{i+1})\leq 10D$.  It is readily shown
 inductively that $y_i$ lies in $A_i$ for each $i$, so that in
 particular $x\in A_n$, as required.
 
 Fix  $n\geq 1$ and assume we have constructed $A_{n-1}$ with the desired 
 properties; we now construct $A_{n}$. Let $A'_{n-1}$ be the
 $10D$--neighborhood of $A_{n-1}$.  Note that, for suitable
 $k'_{n-1},f'_{n-1}$ depending on $k_{n-1},f_{n-1}$, and $D$, the space $A'_{n-1}$ is $k'_{n-1}$--hierarchically
 quasiconvex and has asymptotic dimension at most $\asdim \mathcal O$
 with control function $f'_{n-1}$.
 
 Let $M$ be a constant to be chosen later. We say 
 $U\in\mathfrak U_1$ is \emph{admissible} if whenever $V\in\mathfrak
 S$ satisfies either $U\transverse V$ or $U\propnest V$ we have
 $\dist_V(\rho^U_V, A_{n-1})\leq M$.  The totally orthogonal collection
 $\mathcal U\subseteq \mathfrak U_1$ is admissible if each $U\in\mathcal U$ is admissible.  For an
 admissible collection $\mathcal U$ we define $B_{\mathcal U}\subseteq
 \cuco X$ to be the set of all $x\in \cuco X$ so that
 \begin{itemize}
  \item $\dist_V(x, A'_{n-1})\leq M$ for each $V\perp\mathcal U$;
  \item $\dist_V(x,\rho^{U}_V)\leq M$ whenever $V\in\mathfrak S$ and $U\in\mathcal U$ satisfy $U\transverse V$ or $U\propnest V$.
 \end{itemize}
 (Roughly, $\mathcal B_{\mathcal U}$ is the set of partial 
 realization points for $\mathcal U$ whose projections 
 lie close to $A_{n-1}'$ 
 except possibly in $\fontact U$ for $U\in\mathcal U$.)
 
 Since $A_{n-1}$ is hierarchically quasiconvex, $\mathcal B_{\emptyset}$ coarsely coincides with $A_{n-1}$. Let $A_n=A'_{n-1}\cup \bigcup_{\mathcal U} B_{\mathcal U}$, where
the union is taken over all admissible collections $\mathcal U$. 
The desired properties of $A_n$, for suitable $k_n,f_n$, will be checked in the following claims.

\renewcommand{\qedsymbol}{$\blacksquare$}

\begin{claim}\label{claim:parallel_copies}
If $U\in\mathfrak U_1$ has the property that $\mathbf F_U\times\{e\}$ intersects $A_{n-1}$ for some $e\in \mathbf E_U$, then $\mathbf F_U\times\{e\}\subseteq A_n$. 
\end{claim}

\begin{proof}
 For $U$ as in the statement and $M$ large enough,
 $\diam_V(\rho^U_V\cup\pi_V(\mathbf F_U))\leq M$ whenever
 $U\transverse V$ or $U\propnest V$, by the definition of $\mathbf
 F_U$.  Hence, since $\pi_V$ is coarsely Lipschitz
 and  $\pi_V(\mathbf F_U)$
  intersects $\pi_V(A_{n-1})$, we also have that 
 $\dist_V(\rho^U_V,A_{n-1})\leq M$. Thus $U$ is admissible.  Any
 point $x$ in a parallel copy of $\mathbf F_U$ which intersects
 $A_{n-1}$ is readily seen to be in $\mathcal B_{\{U\}}\subseteq A_n$,
 as required.
\end{proof}

\begin{claim}\label{claim:image_proj}
 For each admissible $U\in\mathfrak U_1$, the projections $\pi_U(\cuco X)$ and $\pi_U(A_{n})$ coarsely coincide. 
\end{claim}

\begin{proof}
 Suppose that $U$ is admissible and let $p\in \pi_U(\cuco X)$. Fix $y\in A_{n-1}$ and define the tuple $\tup b$ in the following way. Let $b_V=\pi_V(y)$ if $V\perp U$, $b_V=\rho^U_V$ if $V\in\mathfrak S$ satisfies $U\transverse V$ or $U\propnest V$, and $b_U=p$. It is easy to check that $\tup b$ is consistent, using the definitions and Proposition~\ref{prop:rho_consistency}. Hence, for $M$ large enough, realization (Theorem~\ref{thm:realization}) provides us with a point $x\in\cuco X$ which, by the definition of $\tup b$, is contained in $\mathcal B_{\{U\}}$ and has $\dist_U(x,p)\leq M$. This completes the proof of the claim.
\end{proof}

\begin{claim}\label{claim:hqc}
 $A_n$ is $k_n$-hierarchically quasiconvex.
\end{claim}

\begin{proof}
Claim \ref{claim:image_proj} provides a constant $C$ so that $\pi_U(\cuco X)$ and $\pi_U(A_{n})$ $C$--coarsely coincide.  Also, $\pi_U(A_{n})$ $C$--coarsely coincides with $\pi_U(A'_{n-1})$ when $U$ is not admissible.  This verifies the part of the definition of hierarchical quasiconvexity governing projections; it remains to check the part governing realization points.

Consider $x\in \cuco X$ so that $\dist_V(x,A_n)\leq t$ for each $V\in\mathfrak S$ and some $t\geq0$. Let $\mathcal U$ be the collection of all $U\in\mathfrak U_1$ so that $\dist_U(x,A_{n-1})\geq 10ME$. Note that each $U\in\mathcal U$ is admissible. We now show that $\mathcal U$ is totally orthogonal, thereby showing that $\mathcal U$ is admissible.

Since each $U\in\mathcal U$ is $\nest$--minimal, distinct elements of $\mathcal U$ are $\nest$--incomparable. Hence we only have to rule out the existence of $U_1,U_2\in\mathcal U$ so that $U_1\transverse U_2$. If such $U_i$ existed then, up to switching $U_1,U_2$, we would have $\dist_{U_2}(x,\rho^{U_1}_{U_2})\leq E$ by consistency and hence $\dist_{U_2}(\rho^{U_1}_{U_2},A_{n-1})\geq 10ME-2E>M$, contradicting the admissibility of $U_1$.

Let us now define the tuple $\tup b$ in the following way. For each 
$V\in\mathfrak S$, let $\pi_{V,A_{n-1}}$ be a coarse closest point
projection $\fontact V\to \pi_V(A_{n-1})$.  Since $A_{n-1}$ is
hierarchically quasiconvex, this map is defined in the usual way when
$\fontact V$ is hyperbolic, and is either coarsely constant or
coarsely the identity otherwise.  We set
$b_V=\pi_{V,A_{n-1}}(\pi_V(x))$ if $V\perp \mathcal U$.  Also, we
require $\dist_V(b_V,\rho^U_V)\leq 10E$ if $V\in\mathfrak S$ satisfies
$U\transverse V$ or $U\propnest V$.  Finally, we set $b_U=\pi_U(x)$
whenever $U\in\mathcal U$.  It is easy to check that $\tup b$ is
consistent, allowing us to invoke realization
(Theorem~\ref{thm:realization}) to get a point $x'\in \cuco X$.  In
fact, we have $x'\in B_{\mathcal U}\subseteq A_n$.

We now bound $\dist(x,x')$. In order to do so, by the uniqueness axiom, we can instead uniformly bound $\dist_V(x,b_V)$ for each $V\in\mathfrak S$. We consider the following cases.
\begin{itemize}
 \item If $V\in\mathcal U$, then $b_V=\pi_V(x)$, as required.
 \item If $V\perp\mathcal U$ then $\pi_V(x)$ is within distance $t+10EMD$ of $A_{n-1}$, and thus within distance $t+10EMD$ of $b_V=\pi_{V,A_{n-1}}(\pi_V(x))$.
 \item Suppose that there exists $U\in\mathcal U$ with $U\transverse V$. If, by contradiction, we had $\dist_V(x,\rho^U_V)\geq 100ME$, then by consistency we would have $\dist_U(x,\rho^V_U)\leq E$. Since by definition of $\mathcal U$ we have that $\pi_U(x)$ is far from $\pi_U(A_{n-1})$, we then get $\dist_U(A_{n-1},\rho^V_U)> E$. By consistency applied to any $y\in A_{n-1}$, we get $\diam_V(\pi_V(A_{n-1})\cup \rho^U_V)\leq 2E$. In particular, $\dist_V(x,A_{n-1})\geq 10ME$, i.e., $V\in\mathcal U$, a contradiction.
 \item Suppose that there exists $U\in\mathcal U$ with $U\propnest V$. If, by contradiction, we had $\dist_V(x,\rho^U_V)\geq 100ME$, then any geodesic from $\pi_V(x)$ to $\pi_{V,A_{n-1}}(x)$ stays $E$--far from $\rho^U_V$ since $\dist_V(x,A_{n-1})\leq 10ME$. Let $y\in A_{n-1}$ be any point so that $\pi_V(y)=\pi_{V,A_{n-1}}(x)$. By bounded geodesic image and consistency (for $x$ and $y$) we have $\dist_U(x,A_{n-1})\leq \dist_U(x,y)\leq 10E$, contradicting $U\in\mathcal U$.
\end{itemize}
This completes the proof of the claim.
\end{proof}

\begin{claim}\label{claim:asdim_bound}
 $A_n$ has asymptotic dimension $\leq \asdim\mathcal O$ with control function $f_n$.
\end{claim}

\begin{proof}
For $i\geq 0$, we define $A^{(i)}_n=A'_{n-1}\cup\bigcup_{\mathcal U} B_{\mathcal U}$, where the union is taken over all admissible collections $\mathcal U$ of cardinality at most $i$.  Observe that there exists $I$, depending only on the complexity of $(\cuco X,\mathfrak S)$, so that $A_n^{(i)}=A_n$ for all $i\geq I$.  Hence it suffices to show, by induction on $i$, that for each $i\geq0$, we have that $\asdim A_n^{(i)}\leq\asdim\mathcal O$, with some control function $f_n^{(i)}$.     

When $i=0$, we have $A_n^{(0)}=A_{n-1}'$, which we saw above has asymptotic dimension $\leq\asdim\mathcal O$ with control function $f_n^{(0)}=f_{n-1}'$.  Suppose for some $i\geq0$ that we have $\asdim A_n^{(i)}\leq\asdim\mathcal O$ with control function $f_n^{(i)}$.  

Write $A_n^{(i+1)}=A_n^{(i)}\cup\bigcup_{\mathcal U}B_{\mathcal U}$,
where $\mathcal U$ varies over all admissible totally orthogonal
subsets in $\frak U_1$ with $|\mathcal U|\leq i+1$.  Note that $\asdim
B_{\mathcal U}\leq\asdim\mathcal O$ uniformly, since $B_{\mathcal U}$
is (uniformly) coarsely contained in $\prod_{U\in\mathcal U}\mathbf
F_U$.  By our induction hypothesis, for each $r\geq0$, we have that
$\neb_r(A_n^{(i)})$ has asymptotic dimension
$\leq\asdim\mathcal O$.  Below, we will establish the following: there
exists $K$ independent of $r$ so that for all $r\geq 0$, and all
admissible sets $\mathcal U\neq\mathcal U'$ of cardinality at most
$i$, we have $\dist(B_{\mathcal U}-Y_r,B_{\mathcal U'}-Y_r)\geq r$,
where $Y_r=\neb_{Kr+K}(A_n^{(i)})$.  Given this, the Union
Theorem~\cite[Theorem~1]{BellDranishnikov:asymptotic_groups} proves
the claim, where the control function is $f_n=f_n^{I}$.

In other words, it suffices to prove that there exists $K\geq0$ such that, if $\mathcal U,\mathcal U'\subseteq\frak U_1$ are admissible totally orthogonal subsets of cardinality at most $i$, then $$\neb_r(B_{\mathcal U})\cap \neb_r(B_{\mathcal U'})\subseteq \neb_{Kr+K}(A_{n-1}')\cup \neb_{Kr+K}(B_{\mathcal U\cap\mathcal U'}).$$  Let $x\in \neb_r(B_{\mathcal U})\cap \neb_r(B_{\mathcal U'})$.  

Let $V\in\mathfrak S$.  If $V\transverse U$ or $U\nest V$ for some $U\in\mathcal U\cup\mathcal U'$, then $\dist_V(x,\rho^U_V)\leq \lambda r+\mu$, where $\lambda,\mu$ depend only on $M$ and the coarse Lipschitz constants for $\pi_V$ (which are independent of $V$).
Since each $U\in\mathcal U\cup\mathcal U'$ is admissible, we then have $\dist_V(\rho^U_V,A_{n-1}')\leq M$, so that $\dist_V(x,A_{n-1}')\leq\lambda r+\mu+M$.  If $V\orth U$ for all $U\in\mathcal U$ or $V\orth U$ for all $U\in\mathcal U'$, then $\dist_V(x,A_{n-1}')\leq \lambda r+\mu+M$ by admissibility.  

Suppose that $V$ is $10(\lambda r+\mu_M)$--\emph{relevant} for $x,\gate_{A'_{n-1}}(x)$, i.e., $\dist_V(x,\gate_{A'_{n-1}}(x))\geq 10(\lambda r+\mu_M)$.  If $V\not\in\mathcal U$ (respectively, $\mathcal U'$), then the preceding discussion shows that for all $U\in\mathcal U\cup\mathcal U'$, we have $V\not\transverse U$ and $U\not\nest V$.  On the other hand, there must exist $U\in\mathcal U\cup\mathcal U'$ with $U\not\orth V$, whence $V\propnest U$, which is impossible since $\mathcal U\cup\mathcal U'\subseteq\frak U_1$.  Hence any such relevant $V$ lies in $\mathcal U\cap\mathcal U'$.  

Let $K'=\max\{\lambda,\mu+M\}$ and let $\mathcal U\cap\mathcal U'=\{U_j\}$.  We have established that:
\begin{itemize}
 \item $\dist_V(x,A_{n-1}')\leq K'r+K'$ for all $V$ for which there exists $j$ with $V\transverse U_j$ or $U_j\nest V$, and indeed $\dist_V(x,\rho^{U_j}_V)\leq K'r+K'$;
 \item $\dist_V(x,A_{n-1}')\leq K'r+K'$ for all $V$ such that $V\orth U_j$ for all $j$;
\end{itemize}
Define a tuple $\tup t\in\prod_{V\in\mathfrak S}2^{\fontact V}$ by $t_{U_i}=\pi_{U_i}(x)$ for all $i$, and so that $t_V=\rho^{U_i}_V$ if there exists $i$ with $U_i\transverse V$ or $U_i\propnest V$ (for some arbitrarily-chosen $i$ with that property if there are many -- the $\rho^{U_i}_V$ all $10E$--coarsely coincide, as can be seen by considering $\pi_V(\prod_i\mathbf F_{U_i})$), and $t_V=\pi_V(\gate_{A'_{n-1}}(x))$ otherwise.  

We claim that $\tup t$ is $100E$--consistent.  Indeed, let $V,W\in\mathfrak S$ with $V\transverse W$ or $V\propnest W$ or $W\propnest V$.  If $t_V=\pi_V(\gate_{A'_{n-1}}(x))$ and $t_W=\pi_W(\gate_{A'_{n-1}}(x))$, then we are done by the consistency axiom.  If $V=U_i$ for some $i$, then $t_W$ $10E$--coarsely coincides with $\rho^{U_i}_W=\rho^V_W$ as required.  If $t_W=\pi_W(\gate_{A'_{n-1}}(x))=\pi_W(g)$ and $V\transverse U_i$ or $U_i\propnest V$ for some $U_i$, then $t_V=\rho^{U_i}_V$.  For each $j$, either $W\orth U_j$ or $W\nest U_j$, by definition.  The latter is impossible since $U_j\in\mathfrak U_1$, and hence $W\orth U_i$.  If $V\transverse W$ or $W\propnest V$, it follows that $\rho^{U_i}_V$ is $10E$--coincident with $\rho^W_V$.  If $V\propnest W$, then $V\orth U_i$, a contradiction.  Finally consider the case where there exist $i,j$ so that $t_V=\rho^{U_i}_V$ and $t_W=\rho^{U_j}_W$.  If $V\propnest W$, then $W\not\orth U_i$, so we can take $U_j=U_i$.  Then the claim follows by $\rho$--consistency (Proposition~\ref{prop:rho_consistency}).  If $V\transverse W$ and $\dist_W(\rho^{U_j},\rho^V_W)>100E$, then $\dist_W(z,\rho^V_W)$ for all $z\in P_{U_j}$, whence $\dist_V(z,\rho^W_V)\leq E$ by consistency.  But $\pi_V(z)$ $10E$--coarsely coincides with $\rho^{U_i}_V$ since $U_i\orth U_j$.  

Hence, by realization (Theorem~\ref{thm:realization}), there exists $y\in\cuco X$ with $\dist_V(y,t_V)\leq\theta$ for all $V$.  If $M\geq\theta$, then $y\in B_{\{U_i\}}$.  The distance formula (with fixed threshold $\theta$, independent of $r$) thus provides $K$ so that $\dist(x,B_{\{U_i\}})\leq\dist(x,y)\leq Kr+K$, since $\dist_V(x,y)\leq K'r+K'+\theta$ for each $V$ that is irrelevant for $x,\gate_{A_{n-1}'}(x)$ and $\dist_{U_i}(x,y)\leq\theta$ for each $i$.  
\end{proof}

\renewcommand{\qedsymbol}{$\Box$}
Assertion~\eqref{item:parallel_copies} follows from Claim~\ref{claim:parallel_copies}, assertion~\eqref{item:hqc} follows from Claim~\ref{claim:hqc}, and assertion~\eqref{item:asdim_bound} follows from Claim~\ref{claim:asdim_bound}.  Assertion~\eqref{item:contains_nhood} holds by definition.  This completes the proof.
\end{proof}

\section{Proof of Theorem~\ref{thmi:asdim} and 
Corollaries~\ref{cori:MCG} and~\ref{cori:teich}}\label{sec:main_asdim_theorem}
In this section, we fix a uniformly locally-finite discrete geodesic
space $\cuco X$ admitting a $\delta$--relatively hierarchically
hyperbolic structure $(\cuco X,\mathfrak S)$ of complexity $\xi$.  Let
$\mathfrak U_1$ be the set of $\nest$--minimal elements, so that
$\fontact U$ is $\delta$--hyperbolic for each $U\in\mathfrak
S-\mathfrak U_1$.  Let $\mathcal O$ be the set of totally orthogonal
subsets of $\mathfrak U_1$ and let $\asdim\mathcal O$ denote the
minimal uniform asymptotic dimension of $\{\prod_{V\in\mathfrak
V}\mathbf F_V:\mathfrak V\in\mathcal O\}$.  
For $U\in\mathfrak U_1$, by~\cite[Lemma~2.1]{BehrstockHagenSisto:HHS_II}
and~\cite[Theorem~32]{BellDranishnikov:asdim_1}, 
we have that $\fontact U$ and $\mathbf F_U$ are uniformly
quasi-isometric and for $\mathfrak V\in\mathcal O$ we have $\asdim\prod_{V\in\mathfrak V}\fontact
V\leq\xi\max\{\asdim\fontact V:V\in\mathfrak V\}$.
Hence $\asdim\mathcal 
O\leq\xi n$, where $n$ is the minimal uniform asymptotic dimension of
$\{\fontact U:U\in\mathfrak U_1\}$.  

\begin{unrem}
In order to apply Corollary~\ref{cor:finite_asdim_curve_graph}, we must assume that for each $U\in\mathfrak S$ with $\fontact U$ a $\delta$--hyperbolic space, the projection $\pi_U$ is uniformly coarsely surjective.  By the proof of Proposition~1.16 of~\cite{DurhamHagenSisto:HHS_IV}, we can always assume that this holds (see also Remark 1.3 of~\cite{BehrstockHagenSisto:HHS_II}).  Hence, in the proof of Theorem~\ref{thm:technical_asdim}, we can make this coarse surjectivity assumption and thus apply Corollary~\ref{cor:finite_asdim_curve_graph}.
\end{unrem}

\begin{defn}[Level, $P_\ell$, $\Delta_\ell$]\label{defn:p_Delta}
Define the \emph{level} of $U\in\mathfrak S$ to be $1$ if $U$ is $\nest$--minimal and inductively define the level of $U\in\mathfrak S$ to be $\ell$ if $\ell-1$ is the maximal integer such that there exists $V\in\mathfrak S$ of level $\ell-1$ with $V\propnest U$.  Let $P_\ell$ be the maximal cardinality of pairwise-orthogonal sets in $\mathfrak S$ each of whose elements has level $\ell$.  Let $\Delta_\ell$ be the maximal uniform asymptotic dimension of $\fontact U$ with $U$ of level $\ell$ and $\fontact U$ hyperbolic.
\end{defn}

\begin{thm}\label{thm:technical_asdim}
Let $\cuco X$ be a uniformly locally-finite discrete geodesic space admitting a $\delta$--relatively HHS structure $(\cuco X,\mathfrak S)$ and let $\xi,n,$ and the $\Delta_\ell$ be as above.  Assume that $n<\infty$.  Then $\asdim\cuco X\leq n\xi+\sum_{\ell=2}^\xi P_\ell\Delta_\ell<\infty$.  In particular, if $(\cuco X,\mathfrak S)$ is an HHS, then $\asdim\cuco X\leq\sum_{\ell=1}^\xi P_\ell\Delta_\ell<\infty$.
\end{thm}

Observe that when $(\cuco X,\mathfrak S)$ is actually an HHS, Corollary~\ref{cor:finite_asdim_curve_graph} automatically gives $n<\infty$.

Before proving Theorem~\ref{thm:technical_asdim}, we record a 
lemma whose proof is an immediate consequence of \cite[Lemma~3.B.6]{CornulierdelaHarpe}.

\begin{lem}\label{lem:DGS_persists}
Let $\cuco X$ be a (not necessarily locally-finite) $(r_0,r_1)$--discrete geodesic space with a $\delta$-relatively hierarchically hyperbolic structure $(\cuco X,\mathfrak S)$.  Let $\mathfrak U\subseteq\mathfrak S$ be closed under nesting.  Then $\cox_{\mathfrak U}$ is quasi-isometric to a connected graph $\Gamma$, with constants independent of $\mathfrak U$.
\end{lem}

\begin{proof}[Proof of Theorem~\ref{thm:technical_asdim}]
By Proposition~\ref{prop:build_A}, for $R\geq0,x\in\cox_{\mathfrak
U_1}$, we have $\asdim\psi^{-1}(B^{\cox_{\mathfrak
U_1}}(x,R))\leq\asdim\mathcal O$ uniformly, where $\psi\colon \cuco
X\to\cox_{\mathfrak U_1}$ is the Lipschitz map provided by
Proposition~\ref{prop:identity_Lipschitz}.  
The space 
$\cuco X$ is a geodesic space, by Lemma~\ref{lem:DGS_persists}.
Hence, we may apply 
Theorem~\ref{thm:fibration_theorem}, which yields  $\asdim\cuco
X\leq\asdim\mathcal O+\asdim\cox_{\mathfrak U_1}$.

Now, by Proposition~\ref{prop:cone_off_bottom}, $(\cox_{\mathfrak U_1},\mathfrak S-\mathfrak U_1)$ is an HHS of complexity $\xi-1$, and $\cox_{\mathfrak U_1}$ is uniformly quasi-isometric to a geodesic space (a graph) by Lemma~\ref{lem:DGS_persists}.  Observe that for each $\nest$--minimal $U\in\mathfrak S-\mathfrak U_1$, we have that the associated subspace $\mathbf F_U\subset\cox_{\mathfrak U_1}$ is uniformly quasi-isometric to $\fontact U$, and thus, by induction,
$$\asdim\cuco X\leq n\xi+\sum_{\ell=2}^\xi P_\ell\Delta_\ell,$$
which is finite since $P_\ell\leq\xi$ for all $\ell$ by~\cite[Lemma~2.1]{BehrstockHagenSisto:HHS_II} and $\Delta_\ell<\infty$ by Corollary~\ref{cor:finite_asdim_curve_graph}.  This yields the desired bound for $(\cuco X,\mathfrak S)$ an HHS, since then we may take $n=\Delta_1$.
\end{proof}

In the case of the mapping class group, sharper bounds on the 
asymptotic dimensions of curve graphs are known. 
Webb has a combinatorial argument which gives a bound which is 
exponential in the complexity~\cite{Webb:tight}. We will make use of 
a much tighter bound due to Bestvina--Bromberg 
\cite{BestvinaBromberg:asdim}. Using this we will now prove 
Corollary~\ref{cori:MCG}:

\begin{proof}[Proof of Corollary~\ref{cori:MCG}]
We use the hierarchically hyperbolic structure on $\MCG(S)$ from~\cite[Section~11]{BehrstockHagenSisto:HHS_II}, where $S$ is a surface of complexity $\xi(S)$, and $\mathfrak S$ is the set of essential subsurfaces up to isotopy, and for each $U\in\mathfrak S$, the space $\fontact U$ is the curve graph.  Let $L$ be the maximum level of a subsurface, i.e., the maximal $\ell$ so that $\mathfrak S$ has a chain $U_1\propnest U_2\propnest\ldots\propnest U_\ell=S$.  

Let $\{U_1,\ldots,U_k\}$ be a collection of pairwise-disjoint
subsurfaces, each of level exactly $\ell>1$.  Then each $U_i$ contains
a subsurface $U_i'$ of level exactly $\ell-1$, and the complement in
$U_i$ of $U'_i$ has level $1$ unless it is a degenerate subsurface. 
Hence $U_i$ contains at least $\ell-1$ disjoint
subsurfaces of level $1$, so $(\ell-1) P_\ell\leq L$ for all 
$\ell\geq 2$ while $P_{1}=L=\xi(S)$  
(note, $U_{i}$ contains at most $\ell$ disjoint subsurfaces of level 
$1$).
As shown in~\cite{BestvinaBromberg:asdim}, $\asdim\fontact U\leq
2\ell+3$ uniformly.  Thus $\Delta_\ell\leq2\ell+3$, and
Theorem~\ref{thm:technical_asdim} gives:
\begin{eqnarray*}
\asdim(\MCG(S))&\leq& 
5L+L\sum_{\ell=2}^{L}\frac{(2\ell+3)}{\ell-1}\leq 2L^2+3L\log L+8L.\\
\end{eqnarray*}
Observing that $L\leq\xi(S)$ provided $\xi(S)\geq2$ completes the proof.
\end{proof}

We now prove Corollary~\ref{cori:teich}:

\begin{proof}[Proof of Corollary~\ref{cori:teich}]
    As noted in 
    \cite{BehrstockHagenSisto:HHS_I}, Teichm\"{u}ller space with 
    either of the two metrics mentioned is a hierarchically 
    hyperbolic space; for details see the corresponding
    discussion for the mapping class group in~\cite[Section
    11]{BehrstockHagenSisto:HHS_II} which applies {\it mutatis 
    mutandis} in the present context. 
    The
    index-set $\mathfrak S$ consists of all isotopy classes of essential
    subsurfaces (only non-annular ones in the case of the 
    Weil--Petersson metric).  For each $U\in\mathfrak S$ which is not an annulus,
    $\fontact U$ is the curve graph of $U$; For annular $U$ the space $\fontact U$ is a
    combinatorial horoball over the annular curve graph. Observing that $\asdim\fontact U\leq
    2$ when $\fontact U$ is the horoball over an annular curve graph, the
    claim now follows as in the proof of
    Corollary~\ref{cori:MCG}, albeit with an extra additive term of 
    $\xi(S)$ since the lowest complexity terms have asymptotic 
    dimension 2 in the present case instead of 1. (Since 
    Weil--Petersson doesn't have contributions from annuli at all, 
    one obtains a sharper bound than we record here.) 
\end{proof}

\section{Quotients of HHG}\label{sec:HHG_quotient}
In this section, we study quotients of hierarchically hyperbolic groups. 
It will be a standing assumption 
throughout this section that $H$ is a \emph{hyperbolically embedded 
subgroup} of the HHG $(G,\mathfrak S)$, as defined by the following, 
cf.\ \cite{DGO}:

\begin{defn}[Hierarchically hyperbolically embedded subgroup]\label{defn:hhhg}
Let $(G,\mathfrak S)$ be an HHG, let $S\in\mathfrak S$ be
$\nest$--maximal and let $H\leq G$.  Given a (possibly infinite) generating set
$\mathcal T$ of $G$, then we write $H\hhhg(G,\mathfrak S)$ if the 
following hold:
\begin{itemize}
 \item $\fontact S$ is the Cayley graph of $G$ with respect to $\mathcal T$ and $\pi_S$ is the inclusion;
 \item $\mathcal T\cap H$ generates $H$;
 
 \item $H$ is hyperbolically embedded in $(G,\mathcal T)$.  Recall 
 that this means that $\cayley(G, H\cup \mathcal T)$ is hyperbolic 
 and $H$ is proper with respect to the metric $\hat\dist$ obtained 
 from measuring the length of a shortest  
 path $\gamma\subset \cayley(G, H\cup \mathcal T)$ with the 
 property that between pairs of vertices in $H\cap\gamma$ the only edges 
 allowed are those from $\mathcal T$.
\end{itemize}
\end{defn}

Throughout this section we let $N$ denote a subgroup $N\triangleleft H\leq 
G$, and we let $\nclose N$ denote its normal closure in 
$G$. When dealing with different HHS structures $\mathfrak S,\mathfrak T$, and when it is necessary to distinguish between the two, we will use the notation $\fontact_{\mathfrak S}U,\fontact_{\mathfrak T}V$ for $U\in\mathfrak S,V\in\mathfrak V$ instead of $\fontact U,\fontact V$. Our main result is:

\begin{thm}\label{thm:quotients}
Let $(G,\mathfrak S)$ be an HHG and let $H\hhhg(G,\mathfrak S)$.  Then there exists a finite set $F\subset H-\{1\}$ such that for all $N\triangleleft H$ with $F\cap N=\emptyset$, the group $G/\nclose N$ admits a relative HHG structure $(G/\nclose N,\mathfrak S_N)$ where:
\begin{itemize}
 \item $\mathfrak S_N=(\mathfrak S\cup\{gH\}_{g\in G})/\nclose N$ and hence $\xi(\mathfrak S_N)\leq\max\{\xi(\mathfrak S),2\}$;
 \item $\fontact_{\mathfrak S_N} S_N=\cayley(G/\nclose N,\mathcal T/\nclose N\cup H/N)$, where $S_N\in\mathfrak S_N$ is $\nest$--maximal and $\mathcal T$ is as in Definition~\ref{defn:hhhg};
 \item for each $U\in\mathfrak S_N-\{S_N\}$, either $\fontact_{\mathfrak S_N} U$ is isometric to $\fontact_{\mathfrak S}U'$ for some $U'\in\mathfrak S$, or $\fontact_{\mathfrak S_N} U$ is isometric to a Cayley graph of $H/N$.
\end{itemize}
In particular, if $N$ avoids $F$ and $H/N$ is hyperbolic, then $G/\nclose N$ is an HHG.
\end{thm}

We postpone the proof until after explaining the necessary tools.  The definition of the index set $\mathfrak S_N$ in the above theorem means the following: by the definition of a hierarchically hyperbolic group structure, $G$, and hence $\nclose N\leq G$, acts on $\mathfrak S$ --- see Section~\ref{subsubsec:aut}.  Also, $\nclose N$ acts on the set $\{gH\}_{g\in G}$ in the obvious way.  Thus $\nclose N$ acts on the set $\mathfrak S\cup\{gH\}_{g\in G}$, and $\mathfrak S_N$ is the quotient of $\mathfrak S\cup\{gH\}_{g\in G}$ by this action.

\begin{rem}[Geometric separation and quasiconvexity]\label{rem:geom_sep}
Since $H\hookrightarrow_h(G,\mathcal T)$, we have that $H$ is \emph{geometrically separated} (with respect to $\fontact S$), i.e., for each $\epsilon\geq0$ there exists $M(\epsilon)$ so that, if $\diam_{\fontact S}(H\cap\neb_\epsilon(gH))\geq M(\epsilon))$, then $g\in H$.  Indeed, for any $\epsilon>0$, if we can choose $g$ so that $\diam_{\fontact S}(H\cap\neb_\epsilon(gH))$ is arbitrarily large, then $H$ contains infinitely many elements $h'$ with $\hat\dist(1,h')\leq 2\epsilon+1$, which is impossible.  In fact,~\cite[Theorem~6.4]{Sisto:metric_relhyp} shows that $\cayley(G,\mathcal T)$ is (metrically) hyperbolic relative to $\{gH\}$, and in particular there exists $M$ such that $H\subset\cayley(G,\mathcal T)$ is $M$--quasiconvex, by~\cite[Lemma~4.3]{DrutuSapir}.
\end{rem}

\begin{rem}[$H$ is hyperbolic]\label{rem:H_is_hyperbolic}
In our situation, the hyperbolically embedded subgroup 
$H\hookrightarrow_h(G,\mathcal T)$ must be hyperbolic. This holds 
since 
$(G,\mathcal T)$ is already hyperbolic (even before adding $H$ to the
generating set); thus, hyperbolicity of $H$ follows from the fact 
that $H\subset (G,\mathcal T)$ is quasiconvex and
$H$ acts properly (indeed, the word-metric $\dist_{\mathcal T}$
restricted to $H$ is bounded below by the auxillary metric on $H$ 
from Definition~\ref{defn:hhhg}, which is proper). 
This hyperbolicity is not a significant restriction on the subgroup, 
since it holds in the cases of interest, including the case where $G$ is the mapping class group of a surface
and $H\cong\integers$ is generated by a pseudo-Anosov element.
\end{rem}

\subsection{Pyramid spaces}\label{subsec:pyramid}
In this section we define the \emph{pyramid spaces} which are hyperbolic spaces,
associated to $G$ and $G/\nclose N$. We also describe a new 
hierarchically hyperbolic structure on~$G$ in which the pyramid space associated to $G$ replaces $\cayley(G,\mathcal T)$.  We begin by recalling from~\cite{GrovesManning:dehn} the notion of a combinatorial horoball, and the attendant hyperbolic cone construction from~\cite{DGO}.

\begin{defn}[Combinatorial horoball, hyperbolic cone]\label{defn:combinatorial_horoball}
Let $\Gamma$ be a graph.  The \emph{combinatorial horoball} $\mathcal H(\Gamma)$ is the graph formed from $\Gamma\times(\naturals\cup\{0\})$ by adding the following edges: for each $r\in\naturals\cup\{0\}$ and each vertex $v\in\Gamma$, join $(v,r),(v,r+1)$ by an edge.  For each $r$ and each $v,v'\in\Gamma$, join $(v,r),(v',r)$ if $\dist_\Gamma(v,v')\in(0,2^r]$.  

Given $r\in\naturals$, the \emph{hyperbolic cone} $\cone(\Gamma,r)$ of \emph{radius $r$} over $\Gamma$ is obtained from $\mathcal H(\Gamma)$ by adding a vertex $v$, called the \emph{apex} of the cone, and joining $v$ to each vertex $(w,s)$ of $\mathcal H(\Gamma)$ for which $s\geq r$.  We endow $\cone(\Gamma,r)$ with the usual graph metric. 
\end{defn}

Lemma~6.43 of~\cite{DGO} says that for any choice of $\Gamma$ and any $r\in\naturals$, the graph $\cone(\Gamma,r)$ is $\delta$--hyperbolic, where $\delta$ may be chosen independently of $\Gamma$ and $r$.

\begin{defn}[Pyramid spaces $\pyramid{G}_r$ and $\pyramid{G/\nclose N}_r$]\label{defn:coned_spaces}
For $r\geq1$, the \emph{pyramid space} $\pyramid{G}_r$ associated to
$(G,\mathcal T)$ and $H$ is obtained from $\cayley(G,\mathcal T)$ by
attaching the hyperbolic cone of radius $r$ over each coset $gH$, with
apex $v_{gH}$. 
Let $\dist_\triangle$ be the graph metric on $\pyramid{G}_r$.
Likewise, $\pyramid{G/\widehat N}_r$ is obtained from $\cayley(G/\nclose
N,\mathcal T/\nclose N)$ by attaching a radius--$r$ hyperbolic cone
over each coset of $H/N$.
\end{defn}

\begin{prop}\label{prop:hyperbolic_pyramids}
There exists $\delta\geq 1$ with the following properties:
\begin{enumerate}
 \item \label{item:g_pyr_hyp} $\pyramid{G}_r$ is $\delta$--hyperbolic for all sufficiently large $r$;
 \item \label{item:gN_pyr_hyp} for each sufficiently large $r>0$,
 there exists a finite set $F_r\subset H-\{1\}$ such that if $N\cap
 F_r=\emptyset$, then $\pyramid{G/\nclose N}_r$ is
 $\delta$--hyperbolic.
\end{enumerate}
\end{prop}

\begin{proof}
Assertion~\eqref{item:g_pyr_hyp} holds by~\cite[Lemma~6.45]{DGO}.  To
prove assertion~\eqref{item:gN_pyr_hyp}, consider the action of $G$ on
the $\delta$--hyperbolic space $\pyramid{G}_r$.  By construction, the
set $\{v_{gH}\}$ of apices is $G$--invariant, $2r$--separated, and
each $v_{gH}$ is fixed by the \emph{rotation subgroup} $H^g$, i.e., 
$(\{v_{gH}\},\{H^g\})$ is a $2r$--separated \emph{rotating family} in
the sense of~\cite[Definition~5.1]{DGO}.  Observe
that~\cite[Corollary~6.36]{DGO} provides a finite set $F_r\subset
H-\{1\}$ such that $\{N^g\}$ is a $\frac{2r}{\delta}$--rotating
family, with respect to the action of $\pyramid{G}_r$, provided
$N\triangleleft H$ avoids $F_r$.  (This means that
$(\{v_{gH}\},\{N^g\})$ is a $2r$--separated rotating family that also
satisfies the \emph{very rotating condition}
of~\cite[Definition~5.1]{DGO}.)  Proposition~5.28 of~\cite{DGO}
provides $r_0$ so that $\pyramid{G/\nclose N}_r$ is uniformly
hyperbolic when $r\geq r_0$.  Enlarging $\delta$ completes the proof.
\end{proof}

\begin{lem}\label{lem:quasiconvex_cones}
 There exist $Q,r_0$ so that $\cone(H,r)\subset\pyramid{G}_r$ is $Q$--quasiconvex if $r\geq r_0$.
\end{lem}
\setcounter{claim}{0}
\begin{proof}
Choose $r$ sufficiently large so that $\pyramid{G}_r$ is $\delta$--hyperbolic, using Proposition~\ref{prop:hyperbolic_pyramids}, and let $\cone$ be the hyperbolic cone of radius $r$ over $H$.  

Define a map $\retract\colon \pyramid{G}_r\to\cone$ as follows.  First, for each $x\in\cone$, let $\retract(x)=x$.  Next, for each $g\in G$, let $\retract(g)$ be some $h\in H$ so that $\dist_{\fontact S}(g,H)=\dist_{\fontact S}(g,h)$.  (This is coarsely unique since $H$ is $M$--quasiconvex in $\fontact S$ by Remark~\ref{rem:geom_sep}.)  If $gH\neq H$ and $(y,n)\in\cone(gH,r)$ is not the apex $v_{gH}$, then let $\retract(y,n)=\retract(y)$.  By Claim~\ref{claim:cosets_bounded}, $\retract$ sends $gH$ to a uniformly bounded subset of $H$, and we define $\retract(v_{gH})$ to be an arbitrarily-chosen point in $\retract(gH)$.

\begin{claim}\label{claim:cosets_bounded}
There exists a constant $K_1$ so that $\diam_{\fontact S}(\retract(gH))\leq K_1$ for all $gH\neq H$.
\end{claim}

\renewcommand{\qedsymbol}{$\blacksquare$}
\begin{proof}[Proof of Claim~\ref{claim:cosets_bounded}]
Apply $M$--quasiconvexity and separation of $gH,H$ (Remark~\ref{rem:geom_sep}).
\end{proof}

\begin{claim}\label{claim:Lipschitz_1}
There exists a constant $K_2$ so that $\dist_{\fontact S}(\retract(x),\retract(y))\leq K_2\dist_{\fontact S}(x,y)$ for all $x,y\in G$.
\end{claim}

\begin{proof}[Proof of Claim~\ref{claim:Lipschitz_1}]
Apply quasiconvexity and the definition of $\mathfrak l$.
\end{proof}

\begin{claim}\label{claim:Lipschitz_2}
There exists $K_3$ so that, if $r$ is sufficiently large, then $\dist_\triangle(\retract(x),\retract(y))\leq K_3\dist_\triangle(x,y)+K_3$ for all $x,y\in\pyramid{G}_r$.
\end{claim}

\begin{proof}[Proof of Claim~\ref{claim:Lipschitz_2}]
Let $[x,y]_\triangle$ be a $\pyramid{G}_r$--geodesic from $x\in G$ to $y\in G$, so that $[x,y]_\triangle=\alpha_1\beta_1\cdots\alpha_n\beta_n\alpha_{n+1}$, where each $\beta_i$ lies in some hyperbolic cone $\cone(g_iH,r)$ and each $\alpha_i$ avoids $v_{gH}\cup(gH\times[1,\infty))$ for each coset $gH$ (i.e., $\alpha_i$ is a possibly trivial path in $\fontact S$).

For each $i$, let $\check\beta_i$ be a geodesic in $\fontact S$ joining the endpoints $a_i,b_i$ of $\beta_i$ (which lie in $g_iH$).  By Claim~\ref{claim:cosets_bounded}, we have $\dist_\triangle(\retract(a_i),\retract(b_i))\leq\dist_{\fontact S}(\retract(a_i),\retract(b_i))\leq K_1$.

By Claim~\ref{claim:Lipschitz_1}, we have $\dist_{\fontact S}(\retract(b_{i-1}),\retract(a_i))\leq K_2\dist_{\fontact S}(b_{i-1},a_i)$.  Now, since $\alpha_i$ is a $\pyramid{G}$--geodesic, we have $\dist_{\fontact S}(b_{i-1},a_i)=\dist_{\triangle}(b_{i-1},a_i)$, whence $\dist_{\fontact S}(\retract(b_{i-1}),\retract(a_i))\leq K_2\dist_{\triangle}(b_{i-1},a_i)$.  Finally, $\dist_{\fontact S}(\retract(b_{i-1}),\retract(a_i))\geq\dist_{\triangle}(\retract(b_{i-1}),\retract(a_i))$, so $\dist_{\triangle}(\retract(b_{i-1}),\retract(a_i))\leq K_2\dist_{\triangle}(b_{i-1},a_i)$.  Hence $\dist_{\triangle}(\retract(x),\retract(y))\leq K_2\sum_i|\alpha_i|+K_1n\leq\max\{K_1,K_2\}\dist_\triangle(x,y)$, and we are done.

The other possibilities are that $x\in\cone(g_0H),y\in\cone(g_{n+1}H)$ or $x\in\cone(g_0H),y\in G$, so that we consider paths of the form $\beta_0\alpha_1\beta_1\cdots\alpha_n\beta_n\alpha_{n+1}$ or $\beta_0\alpha_1\beta_1\cdots\alpha_n\beta_n\alpha_{n+1}\beta_{n+1}$.  We now argue as above and conclude that $\dist_\triangle(\retract(x),\retract(y))\leq 2K_1+\max\{K_1,K_2\}\dist_\triangle(x,y)$.
\end{proof}
\renewcommand{\qedsymbol}{$\Box$}
The lemma follows easily from Claim~\ref{claim:Lipschitz_2}.
\end{proof}

\begin{rem}[Choosing $r$]\label{rem:constants}
Fix $r$ as in Proposition \ref{prop:hyperbolic_pyramids} and assume $r\geq10^9\delta EQ$, where $\delta$ is as in Proposition \ref{prop:hyperbolic_pyramids} and $Q$ as in Lemma \ref{lem:quasiconvex_cones}. Let $\pyramid{G}=\pyramid{G}_r$ and, for each $gH$, let $\cone(gH)=\cone(gH,r)\subset\pyramid{G}$.  Later, we will impose additional assumptions on the size of $r$.
\end{rem}

\begin{defn}[Push-off]\label{defn:push_off}
Let $\alpha$ be a geodesic in $\pyramid{G}$ with endpoints in $G$.  The path $\gamma$ in $\fontact S=\cayley(G,\mathcal T)$ is a \emph{push-off} of $\alpha$ if it is obtained by replacing sub-geodesics in hyperbolic cones with geodesics in $\fontact S$; we call these new subpaths \emph{replacement paths}.
\end{defn}

\begin{lem}\label{lem:quasiconvex_hull}
Let $Z$ be a $\delta$--hyperbolic space and let $\mathcal Q$ be a collection of $Q$--quasiconvex subspaces.  Let $\alpha$ be a geodesic of $Z$ and let $H\subseteq Z$ be the union of $\alpha$ and every $Y\in\mathcal Q$ with $\alpha\cap Y\neq\emptyset$.  Then $H$ is $(Q+2\delta)$--quasiconvex.
\end{lem}

\begin{proof}
Let $\gamma$ be a geodesic joining points $h,h'\in H$.  If $h,h'\in\alpha$, then we are done since $\gamma\subset\neb_{2\delta}(\alpha)$.  Next suppose $h\in Y,h'\in Y'$ for some $Y,Y'\in\mathcal Q$.  Consider closest points $p,p'$ of $Y\cap\alpha,Y'\cap\alpha$ to $h,h'$ respectively, and consider the geodesic quadrilateral $\gamma\rho'\beta\rho^{-1}$, where $\beta$ is the subgeodesic of $\alpha$ from $p'$ to $p$ and $\rho',\rho$ join $h,h'$ to $p,p'$.  Then every point of $\gamma$ lies either $2\delta$--close to a point of $\beta$ or lies $2\delta$--close to a point of $\rho\cup\rho'$ and hence $(Q+2\delta)$--close to a point of $Y\cup Y'$.  A virtually identical argument works if $h\in\alpha$ and $h'\in Y'\in\mathcal Q$.
\end{proof}

\begin{lem}\label{lem:geodesic_fellow_travel_pyramid}
There exists $C=C(\delta,Q)$ so that, if $r$ is sufficiently large, then the following holds for all $x,y\in\pyramid{G}$ and all $gH$: if $\dist_\triangle(x,\cone(gH)),\dist_\triangle(y,\cone(gH))\leq 2\delta+Q$, and $\dist_\triangle(x,y)\geq C$, then any geodesic from $x$ to $y$ intersects $\cone(gH)$ in an interior point of $[x,y]$.
\end{lem}

\begin{proof}
Let $x',y'\in\cone(gH)$ satisfy $\dist_\triangle(x,x'),\dist_\triangle(y,y')\leq 2\delta+Q$ and let $\gamma,\gamma'$ be geodesics joining $x,y$ and $x',y'$ respectively.  Let $\alpha,\beta$ be geodesics joining $x,x'$ and $y,y'$ respectively.  Examining the quadrilateral $\alpha\gamma'\beta^{-1}\gamma^{-1}$ provides $a,b\in\gamma$ so that $\dist_\triangle(a,x),\dist_\triangle(b,y)\leq10\delta$, while $\dist_\triangle(a,\cone(gH)),\dist_\triangle(b,\cone(gH))\leq 2\delta+Q$ (by Lemma~\ref{lem:quasiconvex_cones}), and $C\geq\dist_\triangle(a,b)\geq C-20\delta$.  Suppose that the subpath $[a,b]$ of $\gamma$ from $a$ to $b$ does not enter $\cone(gH)$ (except possibly at the endpoint).  Then this path projects to a path of length at most $KC+K$, for $K=K(Q)$, that lies in $gH$ and joins points $a',b'$ respectively at distance $\leq2\delta+Q$ from $a,b$.  Thus, if $r$ is sufficiently large, we have a path of length $2\log_2(KC+K)+1$ in $\cone(gH)$ joining $a',b'$, whence $\dist_\triangle(a,b)\leq4\delta+Q+2\log_2(KC+K)+1$.  On the other hand, $\dist_\triangle(a,b)\geq C-20\delta$, and this is a contradiction provided $C$ was chosen sufficiently large (in terms of $Q$ and $\delta$).
\end{proof}

\begin{lem}[Push-offs are close to $\fontact S$--geodesics]\label{lem:push-off}
Suppose that $r$ is sufficiently large, in terms of $M,\delta$.  Then
there exists $D$, independent of $r$, so that any push-off of a
geodesic in $\pyramid{G}$ that starts and ends in $G$ is a
$(D,D)$--quasigeodesic and lies within Hausdorff distance $D$ (in
$\fontact S$) from a geodesic in $\fontact S$ with the same endpoints.
\end{lem}

\begin{proof}
Let $[a,b]=\alpha_0\beta_1\cdots\alpha_{n-1}\beta_n\alpha_{n+1}$ be a $\pyramid{G}$--geodesic with $a,b\in\fontact S$, where each $\alpha_i$ is a $\fontact S$--geodesic and each $\beta_i$ lies in some $\cone(g_iH)$.  For each $i$, let $\hat\beta_i$ be a $\fontact S$--geodesic joining the endpoints of $\beta_i$, so that $\gamma=\alpha_0\hat\beta_1\cdots\alpha_{n-1}\hat\beta_n\alpha_{n+1}$ is a push-off of $[a,b]$.  Let $H=[a,b]\cup_i\cone(g_iH)$, so that $H$ is $(Q+2\delta)$--quasiconvex by Lemma~\ref{lem:quasiconvex_hull}.  Let $p\colon \pyramid{G}\to H$ be the projection, and fix a $1$--Lipschitz parameterization $\gamma\colon I\to\fontact S$ of $\gamma$.

Assuming $r$ is large enough, we now define a coarsely Lipschitz projection $q\colon \fontact S\to I$, with constants bounded in terms of $M,\delta$. The existence of such map $q$ easily implies that $\gamma$ is a quasigeodesic with constants depending on $M,\delta$ only, and hence that it is also Hausdorff close to a geodesic, as required.  Write $I=A_0\cup B_1\cup\cdots\cup A_{n-1}\cup B_n\cup A_{n+1}$, where $\gamma|_{A_i}=\alpha_i$ and $\gamma|_{B_i}=\hat\beta_i$.  Given $g\in\fontact S$, if $p(g)\in\alpha_i$, then let $q(g)$ be chosen in $A_i$ so that $\gamma(q(g))=p(g)$.  Otherwise, $p(g)\in g_iH$ for some $i$, and we let $q(g)$ be chosen in $B_i$ so that $\hat\beta_i(q(g))$ is the closest-point projection of $p(g)$ on $\hat\beta_i$, i.e., $\dist_{\fontact S}(p(g),\hat\beta_i(q(g)))=\dist_{\fontact S}(p(g),\hat\beta_i)$. 

Our goal is now to show that $q$ is coarsely Lipschitz whenever $r$ is large enough. It suffices to bound $|q(g)-q(h)|$ for $g,h\in\fontact S$ satisfying $\dist_{\fontact S}(g,h)\leq 1$.

Let $g,h\in\fontact S$ satisfy $\dist_{\fontact S}(g,h)\leq 1$.  Then there exists $K=K(Q,\delta)$ so that $\dist_{\triangle}(p(g),p(h))\leq K$. First of all, we show that we can bound $|q(g)-q(h)|$ whenever $r$ is large enough in any of the following cases:
\begin{itemize}
 \item $p(g),p(h)$ both belong to some $g_iH$,
 \item $p(g),p(h)$ each lie on some $\alpha_i$.
\end{itemize}

In fact, in the first case we can use the fact that, provided $r$ is much larger than $K$, we have $\dist_{\fontact S}(p(g),p(h))\leq K'=K'(K)$, combined with the fact that the closest point projection on $\hat\beta_i$ is coarsely Lipschitz. In the second case, the fact that $p(g),p(h)$ are connected by a subgeodesic of $[a,b]$ of length at most $K$ again ensures that $\dist_{\fontact S}(p(g),p(h))\leq K'=K'(K)$ whenever $r$ is large enough.

Up to switching $g,h$, there is only one case left to analyze: Suppose that there exists $i$ so that $p(g)\in g_iH,p(h)\not\in g_iH$; then $p(h)\in g_jH\cup\alpha_j$ for some $j$.  Let $a'\in g_iH$ be the entrance point in $g_iH$ of the subpath of $[a,b]$ joining $\alpha_j$ to $g_iH$.  Then we claim that there exists $C'=C'(\delta,Q)$ so that $\dist_{\fontact S}(a',p(g))\leq C'$. Since a similar statement holds for $h$ as well if $p(h)\in g_jH$, following arguments similar to the ones above we can then easily get the required bound on $|q(g)-q(h)|$ provided $r$ is much larger than $K$ and $C'$.

Consider a geodesic quadrilateral determined by $p(g),p(h),a',c$, in that order, where $c\in[a,b]$ and $[c,a']$ intersects $\cone(g_iH)$ only in $a'$, and either $c,p(h)\in g_jH$ or $c=p(h)$.  Choose $s$ (independent of $r$) so that $\diam(\neb_{Q}(\cone(gH))\cap\neb_{2\delta+2Q}(\cone(g'H)))\leq s$ whenever $gH\neq g'H$.  Suppose by contradiction that $\dist_\triangle(a',p(g))>10\delta+s+C+K$, for $C$ as in Lemma \ref{lem:geodesic_fellow_travel_pyramid}.  Choose $x,y\in[p(g),a']$ with $\dist_\triangle(p(g),x)=3\delta+K$ and $\dist_\triangle(p(g),y)=5\delta+K+s$.  Then, since $\dist_\triangle(x,[p(g),p(h)]), \dist_\triangle(y,[p(g),p(h)])\geq  K+3\delta- K>2\delta$, we have that $x,y$ are $2\delta$--close to $[a',c]$ or $[p(h),c]$.  However, the former is ruled out by Lemma~\ref{lem:geodesic_fellow_travel_pyramid} since $\dist_\triangle(a',x),\dist_\triangle(a',y)>C+2\delta$ and the fact that $[a',c]$ does not have interior points in $\cone (g_iH)$.  Hence we must have $p(h)\neq c$, so that $p(h),c\in g_jH$, and $\dist_\triangle(x,[p(h),c]),\dist_\triangle(y,[p(h),c])\leq 2\delta$, which is impossible, in view of the definition of $s$, since $\dist_\triangle(x,g_iH),\dist_\triangle(y,g_iH)\leq Q$ and $\dist_\triangle(x,y)>s$. Hence, we showed $\dist_{\fontact S}(a',p(g))\leq C'$ for $C'=10\delta+s+C+K$, as required.
\end{proof}

\subsubsection{An alternative HHG structure on $G$}
It will be convenient to add the cosets of $H$ to the HHG structure of
$G$ and, to do so, we also must replace $\fontact S$ with $\pyramid{G}$. The index set of the 
new structure will include $\mathfrak S$ as a proper subset. For each 
element of $W\in\mathfrak S-\{S\}$, the associated hyperbolic space 
$\fontact W$ is the same in both structures; the hyperbolic space 
associated to $S$ will differ, so we denote the two spaces by 
$\fontact_{\mathfrak S} S$ and $\fontact_{\mathfrak T} S$; sometimes 
for emphasis we will, more generally, use the notation 
$\fontact_{\mathfrak S} W$ and $\fontact_{\mathfrak T} W$ to 
emphasize which structure we are considering at the time.

\begin{prop}\label{prop:aux_HHG}
 The following is an HHG structure on $G$:
 \begin{itemize}
  \item the index set, $\mathfrak T$, contains all the elements of 
  $\mathfrak S$ together with one element for each coset $\{gH\}_{g\in G}$;
  \item $\nest$ and $\orth$ restricted to $\mathfrak S$ are unchanged,
  each $gH$ is only nested into $S$ and not orthogonal to anything;
  \item $\fontact_{\mathfrak T} S$ is $\pyramid{G}$, while  
  $\fontact_{\mathfrak T} U = \fontact_{\mathfrak S} U$ for 
  $U\in\mathfrak S-\{S\}$. Set $\fontact_{\mathfrak T} gH=\cayley(gH, \mathcal T \cap H)$;
  \item $\rho^U_V$ is unchanged for $U,V\in\mathfrak S-\{S\}$ (when defined);
  \item for $U\in\mathfrak S$, the map
  $\rho^S_U\colon \pyramid{G}\to\fontact U$ is unchanged on $\fontact
  S\subset\pyramid{G}$, while $\rho^S_U((gh,s))=\rho^S_U(gh)$ for each
  $g\in G,h\in H,s< \infty$ and $\rho^S_U(v_{gH})=\cup_{h\in H}\rho^S_U(gH)$;
  \item $\rho^S_{gH}(x)$ is the set of entrance points in $\cone(gH)$ of all geodesics $[x,v_{gH}]$, when $x\not\in\cone(gH)$, and otherwise $\rho^S_{gH}(x)=gH$, while $\rho^{gH}_S=v_{gH}$;
  \item $\rho^{U}_{gH}$ is $\rho^S_{gH}(\rho^U_S)$ for $U\in\mathfrak 
  S - \{S\}$, while $\rho^{gH}_U=\pi_U(gH)$ for $U\in\mathfrak S-\{S\}$;
   \item $\pi_U$ is unchanged for $U\in \mathfrak S$ and is the composition $\rho^S_{gH}\circ\pi_{S}$ for each $gH$.
 \end{itemize}
\end{prop}

Before proving the proposition, we record two lemmas:

\begin{lem}\label{lem:entry}
There exists $C\geq1$, independent of $r$, so that for each $x,g\in G$, the set of entry points of $\pyramid{G}$--geodesics $[x,v_{gH}]$ in $\cone(gH)$ is within Hausdorff distance $C$ of $\{x'\in gH: \dist_{\fontact S}(x,x')=\dist_{\fontact S}(x,gH)\}$.  Hence there exists $C$ with $\diam(\rho^S_{gH}(a))\leq C$ for all $a,g\in G$.
\end{lem}

\begin{proof}
This follows from $M$--quasiconvexity of $gH$ in $\fontact S$ (Remark~\ref{rem:geom_sep}) and Lemma~\ref{lem:push-off}.
\end{proof}

\begin{lem}\label{lem:bounded_projections_he}
There exists $C\geq1$, independent of $r$, so that:
\begin{enumerate}
 \item $\diam_{\fontact U}(\pi_U(gH))\leq C$ for all $g\in G$ for all $U\in\mathfrak S-\{S\}$;\label{item:bounded_proj_on_s}  
 \item for all $g'H\neq gH$, $x,y\in g'H$, and all geodesics $[x,v_{gH}],[y,v_{gH}]$ of $\pyramid{G}$ from $x,y$ to $v_{gH}$, the entry points $a_x,a_y$ of $[x,v_{gH}],[y,v_{gH}]$ in $\cone(gH)$ satisfy $\dist_{gH}(a_x,a_y)\leq C$.\label{item:bounded_proj_on_other_cosets}
\end{enumerate}
\end{lem}

\begin{proof}
The first assertion follows from bounded geodesic image and the fact
that $H$ acts properly on $\cayley(G,\mathcal T)$.  The second follows from
Lemma~\ref{lem:entry} and Remark~\ref{rem:geom_sep}.
\end{proof}

We can now prove the proposition:

\begin{proof}[Proof of Proposition~\ref{prop:aux_HHG}]
All aspects of Definition~\ref{defn:space_with_distance_formula} involving only the $\nest,\orth,\transverse$ relations (but not the projections) are obviously satisfied.  Note that $G$ acts cofinitely on $\mathfrak T$ and each $g\in G$ induces an isometry $\fontact_{\mathfrak T}U\to\fontact_{\mathfrak T}(gU)$ for each $U\in\mathfrak T$.  Abusing notation slightly, we denote by $gH$ the subgraph of $\cayley(G,\mathcal T)$ spanned by the vertices of $gH$, which is connected since $\mathcal T\cap H$ generates $H$, and which is hyperbolic by Remark~\ref{rem:H_is_hyperbolic}.

\textbf{Projections $\pi_U$ are well-defined and coarsely Lipschitz:}  This is automatic for $U\in\mathfrak S$.  For each $gH$ and $a\in G$, the projection $\pi_{gH}(a)=\rho^S_{gH}(\pi_S(a))=\rho^S_{gH}(a)$ is bounded by Lemma~\ref{lem:entry}.  We now verify that $\pi_{gH}$ is coarsely Lipschitz; it suffices to verify that $\rho^S_{gH}$ is coarsely Lipschitz on $G\subset\fontact S$.  Let $\gamma$ be a geodesic in $\pyramid{G}$ joining $a,b\in G$, and let $\alpha$ be a push-off of $\gamma$, so that $\dist_\triangle(a,b)=|\gamma|\leq|\alpha|$.  By Lemma~\ref{lem:push-off}, $\alpha$ lies at Hausdorff distance at most $D$ from a geodesic $\alpha'$ of $\fontact S$, and the claim follows.

Lemma~\ref{lem:bounded_projections_he} says that $\pi_U(gH)$ is bounded for $g\in G,U\in\mathfrak S$, so $\rho^{gH}_U$ is coarsely constant.

\textbf{Consistency:}  Let $U,V\in\mathfrak T$ and $a\in G$.  If $U,V\in\mathfrak S-\{S\}$, then consistency holds automatically.  Hence suppose that $U=gH$ for some $g\in G$.  If $gH\propnest V$, then $V=S$.  In this case, $\pi_{gH}(a)=\rho^S_{gH}(\pi_S(a))$ by definition, so consistency holds.  If $U\in\mathfrak S$ and $V=S$, then consistency follows easily from consistency in $(G,\mathfrak S)$.  There is no case in which $V\propnest gH$.  

Hence suppose $gH\transverse V$.  Choose $b\in\rho^S_{gH}(a)$, so that
$b$ is the entrance point in $\cone(gH)$ of some geodesic
$[a,v_{gH}]$.  Then
$\dist_U(a,\rho^{gH}_U)=\dist_U(a,\pi_U(gH))\leq\dist_U(a,b)$.  If
$\dist_U(a,b)>E$, then $\rho^U_S$ lies $10E$--close to a $\fontact
S$--geodesic from $a$ to $b$, from which it is easily deduced that
$\dist_{gH}(b,\rho^S_{gH}(\rho^U_S))$ is uniformly bounded, as
required.

The bound on $\dist_W(\rho^U_W,\rho^V_W)$ from Definition~\ref{defn:space_with_distance_formula}.\eqref{item:dfs_transversal} holds automatically when $U,V,W\in\mathfrak S$ and holds vacuously otherwise by the definition of the nesting relation in $\mathfrak T$.

\textbf{Bounded geodesic image:}  If $U,V\in\mathfrak S-\{S\}$ and $U\nest V$, then bounded geodesic image holds because it held in $(G,\mathfrak S)$.  Hence it remains to consider the case where $V=S$.  First suppose that $U\in\mathfrak S$ and that $\gamma$ is a geodesic in $\pyramid{G}$ that does not pass $(E+D+2r)$--close to $\rho^U_S$.  Let $\alpha$ be a push-off of $\gamma$ and let $\alpha'$ be a $\fontact S$--geodesic at Hausdorff distance $\leq D$ from $\alpha$, provided by Lemma~\ref{lem:push-off}.  Then $\alpha'$ cannot pass through the $E$--neighborhood of $\rho^U_S$, for otherwise $\alpha$ would pass through the $(E+D)$--neighborhood of $\rho^U_S$, whence $\gamma$ would pass through the $(E+D+2r)$--neighborhood of $\rho^U_S$ in $\pyramid{G}$.  Hence there exists $E'=E'(D,E)$ so that $\rho^S_U(\alpha')$ has diameter at most $E'$, by bounded geodesic image in $(G,\mathfrak S)$ (here, we mean $\rho^S_U\colon \fontact S\to\fontact U$).  Each point of $\gamma$ maps by $\pi_U$ to a point at distance at most $C$ from a point of $\pi_U(\alpha')$, by Lemma~\ref{lem:bounded_projections_he}, and we are done since $\rho^S_U(\cone(gH))\subseteq\pi_U(gH)$ for each $gH$.

Next suppose that $U=gH$ for some $g\in G$.  Let $x,y\in\pyramid{G}$.  A thin triangle argument shows that if each geodesic $[x,y]$ is sufficiently far from $gH$ in $\pyramid{G}$, then the geodesics $[x,v_{gH}],[y,v_{gH}]$ enter $\cone(gH)$ at points $a_x,a_y\in gH$ with $\dist_\triangle(a_x,a_y)\leq 100\delta$.  Our choice of $r$ ensures that $\dist_{gH}(a_x,a_y)$ is uniformly bounded, since points in $gH$ at distance $<2r$ in $\cone(gH)$ are at uniformly bounded distance (depending on $r$) in $gH$.

\textbf{Large links:}  Let $a,b\in G$ and let $N=\lfloor\dist_\triangle(a,b)\rfloor$.  We will produce uniform constants $K,\lambda'$ and $T_1,\ldots,T_m\in\mathfrak S$ and $g_1H,\ldots,g_nH$ so that $m+n\leq \lambda' N+\lambda'$ and $\dist_U(a,b)\leq K$ unless $U=g_iH$ or $U\nest T_j$ for some $i,j$.

Fix a $\fontact S$--geodesic $\gamma$ from $a$ to $b$.  By
Lemma~\ref{lem:push-off} and Lemma~\ref{lem:entry}, for each
sufficiently large $K_0>2r$ there exists $K_1$ so that either
$\dist_{gH}(a,b)\leq K_0$ or $\gamma$ has a maximal subpath $\gamma_g$
lying in the $K_1$--neighborhood of $gH$ (in $\fontact S$), and the
endpoints $a_g,b_g$ of $\gamma_g$ satisfy $\dist_{gH}(a,a_g)\leq
K_1,\dist_{gH}(b,b_g)\leq K_1$.  Let $G(a,b)$ be the set of $gH$ with
$\dist_{gH}(a,b)>K_0$.  Observe that for all distinct $gH,g'H\in G(a,b)$, we have
$\diam_{\fontact S}(\gamma_g\cap\gamma_{g'})\leq K_2$, where $K_2$
depends on $K_1$ and the geometric separation constants from
Remark~\ref{rem:geom_sep}. We note that $K_{2}$ does not depend on 
$K_{0}$ and thus, by choosing $K_{0}$ large enough compared to 
$K_{2}$, we can ensure that at most two elements of $G(a,b)$ 
simultaneously overlap. Observe that this implies that the 
cardinality of $G(a,b)$ is at most $2\cdot\dist_{\triangle}(a,b)$.

Write $\gamma=\left(\prod_{i=1}^k\alpha_i\beta_i\right)\alpha_{k+1}$,
where each $\beta_i$ is a subpath contained in the union of subpaths
$\gamma_g$, where $gH\in G(a,b)$, and $\interior{\alpha_i}$ is in the complement
of the union of such paths. 
By bounded geodesic image,
Lemma~\ref{lem:bounded_projections_he}, and the large link lemma in
$(G,\mathfrak S)$, there exist $T_1,\ldots, T_m\in\mathfrak S$ such
that $\dist_U(a,b)\leq K_0$ unless $U\nest T_i$ for some $i$, where
$m=\lambda\lfloor\sum_i|\alpha_i|\rfloor+\lambda$.  Hence any 
elements $U\in\mathfrak T$ in which $\dist_{U}(a,b)>\max\{E,K_{0}\}$ 
is nested into one of at most $3\cdot (\lambda\dist_{\triangle}(a,b)+\lambda)$ 
elements of $\mathfrak T-\{S\}$.

\textbf{Partial realization:} Let $\{U_i\}$ be a set of
pairwise--orthogonal elements of $\mathfrak T$ and let
$b_i\in\fontact_{\mathfrak T}U_i$ for each $i$.  We consider two
cases.  First, if each $U_i\in\mathfrak S$, then partial realization
in $(G,\mathfrak S)$ implies that there exists $g\in G$ so that
$\dist_{U_i}(g,b_i)\leq E$ for each $i$ and
$\dist_V(b_i,\rho^{U_i}_V)\leq E$ when $U_i\propnest V$ or
$U_i\transverse V$.  Thus it remains to note that
$\dist_{\fontact_{\mathfrak T}S}(b_i,\rho^{U_i}_V)$ is uniformly
bounded, since $\fontact
S\hookrightarrow\pyramid{G}$ is distance non-increasing.

When $\{U_i\}=\{S\}$, partial realization holds for $\fontact_{\mathfrak T}S$ since it held for $\fontact S$.  Hence it remains to consider a coset $gH$ and some $gh\in gH$.  Obviously $gh\in G$ has the correct projection in $gH$.  If $gH\propnest V$, then $V=S$ and $\rho^{gH}_S=v_{gH}$.  Hence $\dist_{\fontact_{\mathfrak T}S}(gh,\rho^{gH}_S)\leq r$ as required.  If $gH\transverse V$, then $\rho^{gH}_V=\pi_V(gH)\ni\pi_V(gh)$, as required.

\textbf{Uniqueness:}  Let $\kappa\geq0$ be given and let $\theta=\theta(\kappa)$ be the corresponding constant from $(G,\mathfrak S)$, so that, if $\dist_G(a,b)\geq\theta$ then $\dist_{\fontact V}(a,b)\geq\kappa$ for some $V\in\mathfrak S$.  Hence we must consider only $a,b\in G$ such that $\dist_G(a,b)\geq\theta$ but $\dist_{\fontact V}(a,b)\leq\kappa$ if and only if $V\neq S$.  Consider a geodesic $\gamma$ in $\pyramid{G}$ joining $a$ to $b$ and a push-off $\alpha=\alpha_0\beta_0\cdots\alpha_n\beta_n\alpha_{n+1}$ of $\gamma$, where each $\beta_i$ is a $\fontact S$--geodesic joining two points in $\gamma$ lying in some $g_iH$.  Lemma~\ref{lem:push-off} implies that $\alpha$ lies at Hausdorff distance $D$ from a $\fontact S$--geodesic $\alpha'$ joining $a,b$.  By assumption, $\alpha'$ has length at least $\kappa$.  Hence either $|\beta_i|\geq\kappa$ for some $i$ (i.e., $\dist_{g_iH}(a,b)\geq\kappa$) or $n\geq \epsilon\kappa$ for some uniform $\epsilon\geq0$, and thus $\dist_{\fontact_{\mathfrak T}S}(a,b)\geq\epsilon\kappa$.  Thus for each $\kappa\geq0$, we have that $\dist_V(a,b)\geq\kappa$ for some $V\in\mathfrak T$ provided $\dist_G(a,b)\geq\max\{\theta(\kappa),\theta(\epsilon^{-1}\kappa)\}$, i.e., uniqueness holds.    
\end{proof}

\subsection{Proof of Theorem \ref{thm:quotients}}
In light of Lemma~\ref{lem:push-off}, we now enlarge $r$ so that 
$r\geq 10^9CDEQ\delta$, where $C$ exceeds the constants from
Lemma~\ref{lem:entry} and Lemma~\ref{lem:bounded_projections_he}.  Let
$F$ be a finite subset of $H-\{1\}$ chosen so that for any  $N\triangleleft H$ which avoids $F$,
yields  $(\{N^g:g\in G\},\{v_{gH}:g\in
G\})$ is a \emph{very rotating family} (see the proof of
Proposition~\ref{prop:hyperbolic_pyramids}); our choice of $r$ ensures
that $\{v_{gH}\}$ is $200\delta$--separated.
Let
$\pyramid{G}=\pyramid{G}_r$, and let 
$\pyramid{G/\nclose N}=\pyramid{G/\nclose N}_r$,
where $N\triangleleft H$ avoids $F$.  

\subsubsection{Linked pairs}\label{subsubsec:linked_pair}
\begin{defn}[Fulcrum]\label{defn:fulcrum}
 Let $x,y\in \pyramid{G}$ and $N\triangleleft H$. We say that the apex $v_{gH}$ is a $d$--\emph{fulcrum} for $x,y$ if $\dist_\triangle(x,y)=\dist_\triangle(x,v_{gH})+\dist_\triangle(v_{gH},y)$ and there exists $h\in gNg^{-1}$, $x'\in [x,v_{gH}]$, $y'\in[v_{gH},y]$ with the following properties.
 \begin{itemize}
  \item $\dist_\triangle(x',v_{gH}),\dist_\triangle(y',v_{gH})\in [25\delta,30\delta]$,
  \item $\dist_\triangle(x',hy')\leq d$.
 \end{itemize}
A fulcrum is shown in Figure~\ref{fig:fulcrum}.
\end{defn}

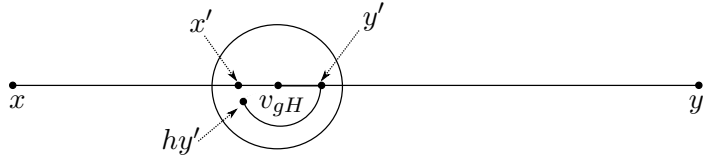
\begin{figure}[h]
\begin{overpic}[width=0.6\textwidth]{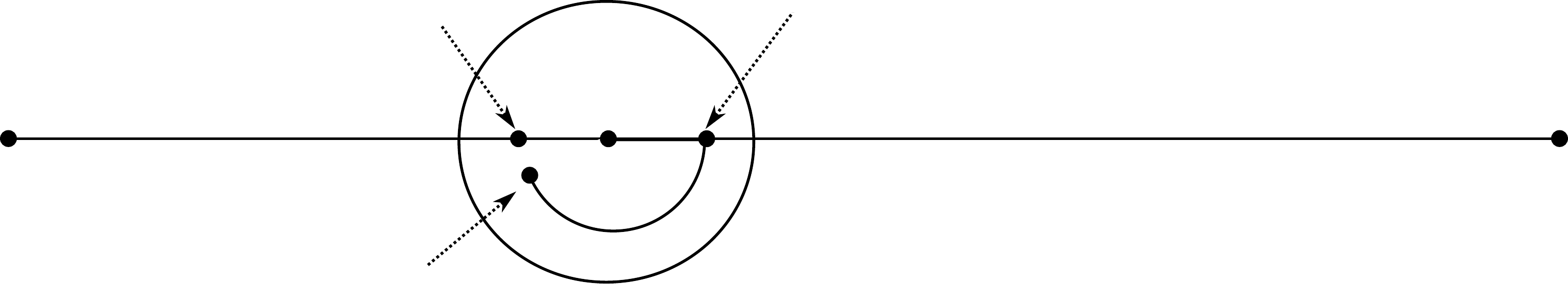}
\put(36,6){$v_{gH}$}
\put(0,6){$x$}
\put(98,6){$y$}
\put(51,18){$y'$}
\put(26,17){$x'$}
\put(22,0){$hy'$}
\end{overpic}
\caption{A fulcrum.}\label{fig:fulcrum}
\end{figure}

Our assumptions on $N$ and $r$ mean that we have the following ``Greendlinger lemma''~\cite[Lemma~5.10]{DGO}, which is formulated in our context as follows:

\begin{lem}[Greendlinger lemma]\label{lem:Greendlinger}
Let $n\in\nclose N-\{1\}$ and let $p\in\pyramid{G}$.  Then one of the following holds:
\begin{enumerate}[(A)]
 \item any $\pyramid{G}$--geodesic $[p,np]$ contains a
 $\delta$--fulcrum for $p,np$; or, 
 \item there exists $v_{gH}$ so that $n\in gNg^{-1}\triangleleft
 gHg^{-1}$ and $\dist_\triangle(p,v_{gH})\leq 25\delta$.
\end{enumerate}
\end{lem}

In~\cite{DGO}, the first conclusion uses $5\delta$, but in fact the conclusion can be made to hold for an arbitrarily small constant, and we use $\delta$ for convenience.

\begin{rem}[]\label{rem:projections_are_points}
For convenience, we can and shall assume that, for all $U\in\mathfrak S-\{S\}$ and all $x\in G$, the sets $\rho^U_S$ and $\pi_S(x)$ consist of single points, and $\rho^{gU}_S=g\rho^U_S$ and $\pi_S(gx)=g\pi_S(x)$ for all $g\in G$.  Indeed, equivariantly replace each relevant bounded set with one of its elements, and adjust the constants of Definition~\ref{defn:space_with_distance_formula} and Subsection~\ref{subsubsec:aut} uniformly if necessary.
\end{rem}

\begin{defn}[Linked pair]\label{defn:linked_pair}
Let $U,V\in\mathfrak S$ (resp.\ $x,y\in G$, resp.\ $U\in\mathfrak S,
x\in G$).  Then $\{U,V\}$ (resp.\  $\{x,y\}$, resp.\  $\{U,x\}$) is
\emph{linked} if there does not exist a
$10\delta$--fulcrum for
$\rho^U_S,\rho^V_S$ (resp.\  $\pi_S(x),\pi_S(y)$, resp.
$\rho^U_S,\pi_S(x)$). We say the pair is \emph{weakly linked} when 
there is no $5\delta$--fulcrum.
\end{defn}

\begin{lem}\label{lem:linked_pairs_exist}
Linked pairs have the following properties:
\begin{enumerate}
 \item for all $[U],[V]\in\mathfrak S/\nclose{N}$, there exists a (weakly) linked pair $U\in[U],V\in[V]$, and the same holds for pairs $\nclose Nx,\nclose Ny$;\label{item:linked_pairs_exist}
 \item for any $g\in G$ the pair $\{gU,gV\}$ is (weakly) linked
 whenever $\{U,V\}$ is (weakly) linked; the same
 holds for linked pairs $\{x,y\}$;\label{item:linked_is_invariant}
 \item if $x\in G$ and $n\in\nclose N$, then $\{x,nx\}$ are not
 weakly linked and, in fact, have a $\delta$--fulcrum (and the same 
 holds for $x$ and $xn$ since $\nclose 
 N$ is normal). \label{item:translates_not_linked}
\end{enumerate}
\end{lem}

\begin{proof}
Choose $x,y\in\pyramid{G}$ and suppose that $v_{gH}$ is a $5\delta$--fulcrum for $x,y$ (i.e., $x,y$ are not weakly linked).  Choose $x',y'\in[x,v_{gH}],[y,v_{gH}]$ and $h\in gHg^{-1}$ so that $\dist_\triangle(x',v_{gH}),\dist_\triangle(y',v_{gH})\in[25\delta,30\delta]$ and $\dist_\triangle(hy',x')\leq 5\delta$.  Then $$\dist_\triangle(x,hy)\leq\dist_\triangle(x,x')+\dist_\triangle(y,y')+5\delta\leq\dist_\triangle(x,y)-45\delta,$$
so, by replacing $x,y$ with $x,hy$, we obtain a closer pair of representatives of $\nclose Nx,\nclose Ny$.  This proves assertion~\eqref{item:linked_pairs_exist} for weakly linked pairs. Repeating exactly the same argument with $10\delta$ replacing $5\delta$ establishes the assertion for linked pairs.  

For all $x,y\in \pyramid{G}$, cosets $g'H$, and $g\in G$, and $d\geq0$, observe that $v_{g'H}$ is a $d$--fulcrum for $x,y$ if and only if $v_{gg'H}$ is a $d$--fulcrum for $gx,gy$, which proves assertion~\eqref{item:linked_is_invariant}.  Assertion~\eqref{item:translates_not_linked} follows from Lemma~\ref{lem:Greendlinger}.
\end{proof}

\subsubsection{Proof of Theorem~\ref{thm:quotients}}\label{subsubsec:proof}
Throughout this section, we say that $N\triangleleft H$ is \emph{sufficiently deep} if $N\cap F=\emptyset$, where $F\subset H-\{1\}$ is the finite subset whose exclusion from $N$ implies that $(\{gNg^{-1}\},\{v_{gH}\})$ is a $200\delta$--separated very rotating family. We now define the hierarchical space structure on $G/\nclose N$.  By choosing $N$ suffciently deep, we can ensure that, if $g,g'\in G$ differ by a generator in a fixed finite generating set, then $g,g'$ are linked.  

\begin{cons}\label{cons:hhg_structure}
The index set and associated hyperbolic spaces are defined by the 
following, where $(G,\mathfrak S)$ is the original HHG structure and 
the modified HHG structure provided by Proposition~\ref{prop:aux_HHG} 
is denoted $(G,\mathfrak T)$:
\begin{enumerate}
 \item the index set $\mathfrak S_N$ is $\mathfrak T/\nclose N$;
 \item for $\mathbf S=\{S\}$ (the $\nclose N$--orbit of $S$), let
 $\fontact \mathbf S=\cayley(G/\nclose N,\mathcal T/\nclose N\cup
 H/N)$ --- note that this is quasi-isometric to $\pyramid{G/\nclose N}$;
 \item $\mathbf U\in \mathfrak S/\nclose N$, let $\fontact\mathbf
 U=\left(\bigsqcup_{U\in\mathbf U}\fontact U\right)/\nclose N$ --- 
 note that this 
 is isomorphic to $\fontact U$ for some (hence any) $U\in\mathbf U$;
 \item for $\mathbf U=\nclose N gH$, let $\fontact\mathbf
 U=\left(\bigsqcup_{ngH\in\mathbf U}\cayley(H,\mathcal T\cap
 H)\right)/\nclose N$ --- note that this is isometric to a Cayley graph of $H/N$;
 
 \item for each $\mathbf U,\mathbf V\in \mathfrak S_N$, let $\mathbf
 U\nest \mathbf V$ (resp.  $\mathbf U\orth \mathbf V$) if there exists
 a linked pair $\{U,V\}\subseteq \mathfrak S$ with $U\in\mathbf U$,
 $V\in\mathbf V$ so that $U\nest V$ (resp.  $U\orth V$).  If $\mathbf 
 U=\nclose N gH$, then we let $\mathbf U\nest \mathbf S$. 
 If neither $\mathbf U\nest \mathbf
 V$, $\mathbf V\nest \mathbf U$ nor $\mathbf U\orth \mathbf V$ holds, 
 then  we let
 $\mathbf U\transverse\mathbf V$ .
\end{enumerate}

The projections are defined by taking all linked pair representatives:

\begin{enumerate}
\setcounter{enumi}{5}
 \item $\pi_{\mathbf U}(g\nclose N)=\left(\bigcup \pi_{U'}(g')\right)/\nclose N$, where the union is taken over all linked pairs $\{U',g'\}$ with $U'\in\mathbf U$, $g'\in g\nclose N$;
 \item similarly, for $\mathbf U \propnest\mathbf V$ or $\mathbf U \transverse \mathbf V$, $\rho^{\mathbf U}_{\mathbf V}$ is defined as $\left(\bigcup \rho^{U'}_{V'}\right)/\nclose N$, where the union is taken over all linked pairs $\{U',V'\}$ with $U'\in\mathbf U$, $V'\in \mathbf V$;
 \item finally, for $\mathbf V \propnest\mathbf U$ and $\nclose Nx\in\fontact \mathbf U$, let $\rho^{\mathbf U}_{\mathbf V}(\nclose Nx)=\left(\bigcup \rho^{U'}_{V'}(x')\right)/\nclose N$, where the union is taken over all linked pairs $\{U',V'\}$ with $U'\in\mathbf U$, $V'\in \mathbf V$ and $x'\in \nclose Nx\cap \fontact U'$.
\end{enumerate}
\end{cons}

Before proceeding with the proof of Theorem~\ref{thm:quotients}, we need several lemmas:

\begin{lem}\label{lem:unique_close_linked_partner}
 If $N$ is sufficiently deep, then for any $\mathbf U,\mathbf V\in \mathfrak S/\nclose{N}-\{\mathbf S_N\}$ and any $U\in\mathbf U$ there exists at most one $V\in\mathbf V$ with $\dist_{\fontact S}(\rho^U_S,\rho^V_S)\leq 10E$, and hence in particular at most one such $V$ with $U\nest V$, $V\nest U$ or $U\orth V$.
\end{lem}

\begin{proof}
It follows from Lemma~\ref{lem:Greendlinger} that, for each $n\in N-\{1\}$ and $x\in\fontact S$, we have $\dist_\triangle(x,nx)\geq 2r$, and hence $\dist_{\fontact S}(x,nx)\geq 2r>10E$. The last assertion follows from Definition \ref{defn:space_with_distance_formula}.\ref{item:dfs_transversal} when $U\nest V$ or $V\nest U$, and from Lemma~\ref{lem:orth_close} when $U\orth V$; in both cases $\rho^U_S$ and $\rho^V_S$ are $E$--close.
\end{proof}

\begin{lem}\label{lem:pass_close_to_apex}
If $N$ is sufficiently deep, then the following holds for each $x,y\in
G$: if $\{U,\pi_S(x)\}$ and $\{U,\pi_S(y)\}$ are linked (resp.\ 
weakly linked) then either:
  \begin{enumerate}
   \item $\{\pi_S(x),\pi_S(y)\}$ is weakly linked (resp.\ has no 
   $\delta$--fulcrum), or
   \item there exists $v_{gH}$ with
   $\dist_\triangle([\pi_S(x),\rho^U_S],v_{gH})\leq 40\delta$ and
   $\dist_\triangle([\pi_S(y),\rho^U_S],v_{gH})\leq 40\delta$.
  \end{enumerate}
The analogous statement holds when replacing $x$ and/or $y$ with an element of $\mathfrak S-\{S\}$.
\end{lem}

\begin{proof}
Let $U,x,y$ be as in the statement and suppose that $\{x,y\}$ is not weakly linked.  By definition, there exists a $5\delta$--fulcrum $v_{gH}$ for $\pi_S(x),\pi_S(y)$.  Consider a geodesic triangle in the $\delta$--hyperbolic space $\pyramid{G}$ with vertices $\pi_S(x),\pi_S(y)$ and $\rho^U_S$, with $v_{gH}\in [\pi_S(x),\pi_S(y)]$. Choose $x'\in[x,v_{gH}],y'\in[v_{gH},y]$ so that $\dist_\triangle(x',v_{gH}),\dist_\triangle(y',v_{gH})\in[25\delta,30\delta]$ and so that $\dist_\triangle(x',hy')\leq 5\delta$ for some $h\in gHg^{-1}$.

\begin{figure}[h]
\begin{overpic}[width=0.6\textwidth]{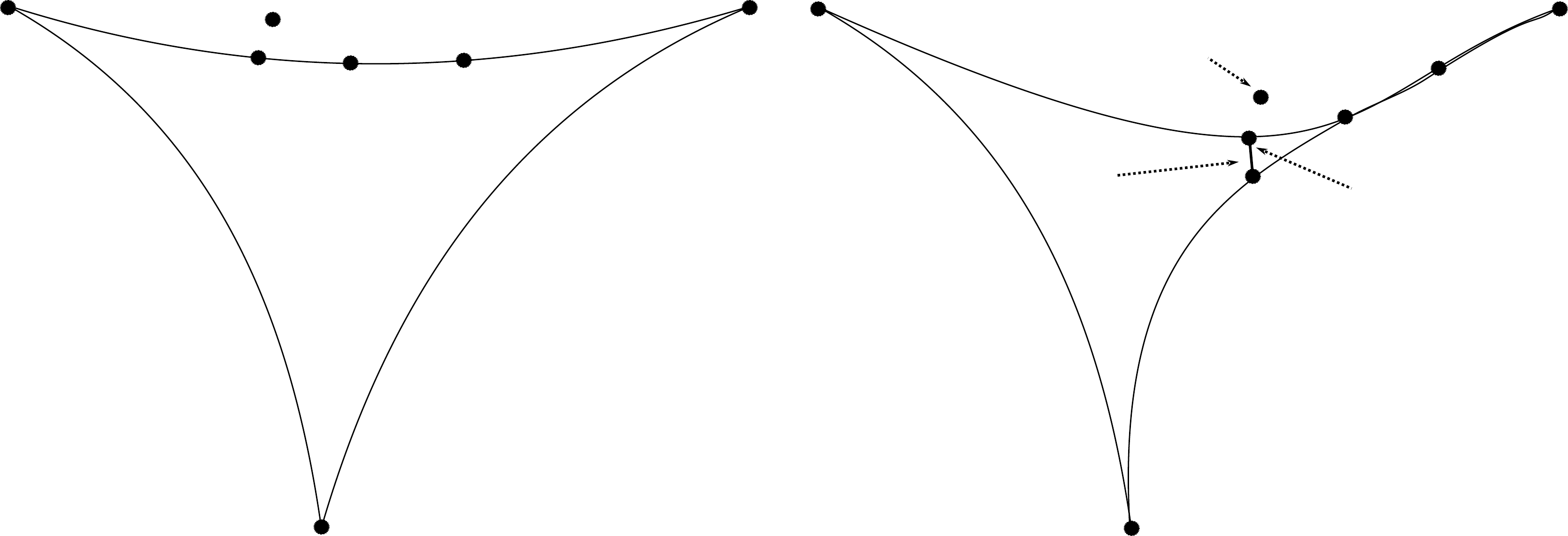}
\put(21,0){$\rho^U_S$}
\put(73,0){$\rho^U_S$}
\put(-1,31){$x$}
\put(46,31){$y$}
\put(51,31){$x$}
\put(98,31){$y$}
\put(12,33){$hy'$}
\put(15,27){$x'$}
\put(21,27){$v_{gH}$}
\put(30,28){$y'$}
\put(65,22){$\leq\delta$}
\put(86,20){$x'$}
\put(73,31){$hy'$}
\put(91,32){$y'$}
\put(87,25){$v_{gH}$}
\end{overpic}
\caption{Finding an illegal fulcrum in the proof of Lemma~\ref{lem:pass_close_to_apex}.}\label{fig:triangle_fulcrum}
\end{figure}
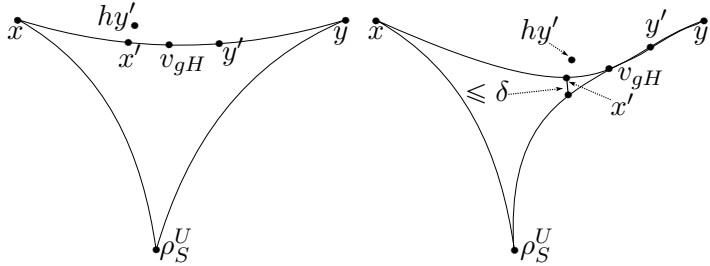

If
$\dist_\triangle(x',[\rho^U_S,\pi_S(x)]),\dist_\triangle(y',[\rho^U_S,\pi_S(x)])>\delta$,
then the very rotating condition~\cite[Lemma 5.5]{DGO} implies 
that $v_{gH}$ is contained in $[\rho^U_S,\pi_S(y)]$ and that $v_{gH}$
is a $10\delta$--fulcrum for $\pi_S(y),\rho^U_S$ (witnessed by the
same element $h\in gHg^{-1}$), contradicting that $U,y$ are linked. 
(See Figure~\ref{fig:triangle_fulcrum}.)  Since the same argument works
for $[\rho^U_S,\pi_S(y)]$, we have
$$\dist_\triangle([\rho^U_S,\pi_S(x)],\{x',y'\})\leq
\delta,\,\,\,\,\dist_\triangle([\rho^U_S,\pi_S(y)],\{x',y'\})\leq
\delta,$$ and $\dist_\triangle(v_{gH},x'),
\dist_\triangle(v_{gH},y')\leq 30\delta$, so the claim follows.
\end{proof}

\begin{lem}\label{lem:same_proj_as_gH}
There exists $K\geq 0$ so that the following holds: if $x,g\in G$ and $U\in\mathfrak S$ satisfy $\dist_\triangle([\pi_S(x),\rho^U_S],v_{gH})\leq 40\delta$ then $\diam_{\fontact U}(\pi_U(x)\cup gH)\leq K$.
\end{lem}

\begin{proof}
Observe that $[\pi_S(x),\rho^U_S]$ must pass through $\cone(gH)$, since $\dist_{\triangle}([\pi_S(x),\rho^U_S],v_{gH})\leq40\delta$ and $r\geq10^9\delta$.  Let $a$ be the entry point of $[x,\rho^U_S]$ in $\cone(gH)$ and let $[x,a]$ be the sub-geodesic joining $x$ to $a$.  Let $[a,b]$ be the sub-geodesic of $[\pi_S(x),\rho^U_S]$ that joins the entry point $a$ of $[\pi_S(x),\rho^U_S]$ in $\cone(gH)$ to the entry point $b$ of $[\rho^U_S,x]$ in $\cone(gH)$.  Since $\dist_\triangle(a,b)\geq 2r-80\delta>1000CDE\delta$, we have $\dist_{\fontact S}(a,b)>1000CDE\delta$.  Combined with Lemma~\ref{lem:push-off}, this shows that $\dist_{\fontact S}(\rho^U_S,c)>100E$ for any $c$ on a $\fontact S$--geodesic from $x$ to $a$.  Hence $\dist_{\fontact S}([x,a],\rho^U_S)>E$, so bounded geodesic image in $(G,\mathfrak S)$ implies that $\dist_U(x,a)\leq E$, whence $\diam_U(\pi_U(x)\cup\pi_U(gH))\leq E+C$, by Lemma~\ref{lem:bounded_projections_he}.
\end{proof}

Combining Lemma \ref{lem:pass_close_to_apex} and Lemma \ref{lem:same_proj_as_gH}, and increasing $C$ if necessary, yields:

\begin{cor}\label{cor:unlinked_have_same_proj}
 There exists $C\geq 0$ so that the following holds for each $x,y\in
 G$ provided $N$ is sufficiently deep.  If $\{U,\pi_S(x)\}$ and
 $\{U,\pi_S(y)\}$ are linked (resp.\ weakly linked) 
 but $\{\pi_S(x),\pi_S(y)\}$ is not
 weakly linked (resp.\ has a $\delta$--fulcrum), then $\dist_{\fontact U}(x,y)\leq C$.  The same
 holds with $x$ and/or $y$ replaced with elements of~$\mathfrak
 S-\{S\}$.
\end{cor}

\begin{lem}\label{lem:hellyish}
Let $\{\mathbf U_i\}_{i=1}^k$ be a totally orthogonal set with $\mathbf U_i\in\mathfrak S/\nclose N$ for all $i$.  Then there exist representatives $U_i\in\mathbf U_i$ so that for all distinct $i,j$, we have $U_i\orth U_j$ and $\{U_i,U_j\}$ is a linked pair.
\end{lem}

\begin{proof}
This follows by induction on $k$, using Lemma~\ref{lem:not_close}.  Indeed, when $k=1$, there is nothing to prove.  Suppose that we can choose the required $U_i$ for $1\leq i\leq k-1$.  For each $i\leq k-1$, choose $U_k^i\in\mathbf U_k$ so that $\{U_i,U^i_k\}$ is a linked pair and $U_i\orth U^i_k$.  Then for each $i,j$, we have $\dist_\triangle(\rho^{U^i_k}_S,\rho^{U^j_k}_S)\leq30E$, so, by Lemma~\ref{lem:not_close}, the $U^i_k$ all coincide, and we are done.
\end{proof}

\begin{lem}\label{lem:not_close}
Let $U\in\mathfrak S$ and let $n\in\nclose N$.  Then either $nU=U$ or $\dist_\triangle(\rho^U_S,\rho^{nU}_S)>100E$.
\end{lem}

\begin{proof}
If $U\neq nU$ and $\dist_\triangle(\rho^U_S,\rho^{nU}_S)\leq100E$, then each geodesic $[\rho^U_S,\rho^{nU}_S]$ in $\pyramid{G}$ fails to pass through any apex, since $r>10^9E$.  In particular, there is no $5\delta$--fulcrum for $\rho^U_S,\rho^{nU}_S=n\rho^U_S$.  Thus, by Lemma~\ref{lem:Greendlinger}, there exists $v_{gH}$ so that $n\in gNg^{-1}$ and $\dist_\triangle(\rho^U_S,v_{gH})\leq25\delta$.  But this is impossible, since $\rho^U_S\in\fontact S$ lies at distance at least $r>10^9\delta$ from any apex.  
\end{proof}

\begin{lem}\label{lem:labelled}
There exists $C'$ so that the following holds.  Let $\{g,g'\}$ be $\dist_\triangle$--minimal representatives of $g\nclose N,g'\nclose N$ and let $U\in\mathfrak T$ satisfy $\dist_U(g,g')>C'$.  Then $\{U,g\}$ and $\{U,g'\}$ are linked.
\end{lem}

\begin{proof}
Consider a geodesic triangle in $\pyramid{G}$ formed by
$g,g',\rho^U_S$.  Suppose $\{U,g\}$ is not linked, so $[\rho^U_S,g']$
contains a $10\delta$--fulcrum $v=v_{g''H}$ for $\{U,g'\}$.  If
$[g,\rho^U_S]$ passes $40\delta$--close to $v$, then
Lemma~\ref{lem:same_proj_as_gH} shows that $\pi_U(g),\pi_U(g')$
coarsely coincide with $\rho^{g''H}_U$.  Otherwise, $[\rho^U_S,g']$
contains a length--$60\delta$ subpath, centered at $v$ and contained
in $\neb_\delta([g,g'])$.  Using the notation of
Definition~\ref{defn:fulcrum}, let $h\in N^{g''},x',y'\in[\rho^U_S,g]$
witness the fact that $v$ is a $10\delta$--fulcrum for $\rho^U_S,g'$,
with $y'$ between $v$ and $g'$.  Choose $x'',y''\in[g,g']$ with
$\dist_\triangle(x',x''),\dist_\triangle(y',y'')\leq\delta$.  Then
$\dist_\triangle(g,hg')\leq\dist_\triangle(g,x'')+\dist_\triangle(y'',g')+12\delta<\dist_\triangle(g,g')$,
contradicting our choice of $g,g'$.
\end{proof}

We are now ready to complete the proof of Theorem~\ref{thm:quotients}:

\begin{proof}[Proof of Theorem \ref{thm:quotients}] 
The claimed hierarchical space structure $(G/\nclose{N},\mathfrak S_N)$ is described in Construction~\ref{cons:hhg_structure}.  Observe that each $\fontact\mathbf U$ is uniformly hyperbolic by definition when $\mathbf U=\nclose NU$ for some $U\in\mathfrak S-\{S\}$.  Moreover, $\fontact\mathbf S$ is hyperbolic by Proposition~\ref{prop:hyperbolic_pyramids}.  If $\mathbf U\in\mathfrak S_N$ arose from a coset of $H/N$, then $\mathbf U$ is necessarily $\nest$--minimal.  Hence, if $(G/\nclose N,\mathfrak S_N)$ is a hierarchical space structure, then it is a relatively hierarchically hyperbolic space structure.  Moreover, $G/\nclose N$ acts on  $\mathfrak S_N$ and, for each $g\in G/\nclose N$ and $\mathbf U\in\mathfrak S_N$, it is easily seen that there is an induced isometry $\fontact\mathbf U\to\fontact g\mathbf U$ so that the required diagrams from Section~\ref{subsubsec:aut} coarsely commute.  Hence it suffices to show that $(G/\nclose N,\mathfrak S_N)$ is a hierarchical space.

We observe that if $\{U,x\}$ is linked and $\{U,nx\}$ is weakly linked, then $\dist_U(x,nx)\leq C$, for $C$ as in Corollary 6.26. In fact, $\{x,nx\}$ has a $\delta$--fuclrum by Lemma 6.18.

\textbf{Verifying
Definition~\ref{defn:space_with_distance_formula}.\eqref{item:dfs_curve_complexes}:}
To finish proving that $(G/\nclose N,\mathfrak S_N)$ satisfies the
projections axiom, we must check that each $\pi_{\mathbf U}$ sends
points to uniformly bounded sets and is uniformly coarsely Lipschitz.
Let $\mathbf U\in\mathfrak S_N$, then for each $g\in G$ we have that 
$\pi_{\mathbf U}(g\nclose N)$ is uniformly bounded by 
Corollary~\ref{cor:unlinked_have_same_proj} and 
Lemma~\ref{lem:linked_pairs_exist}.(\ref{item:translates_not_linked}).

We now show that $\pi_{\mathbf U}$ is coarsely lipschitz.  Let
$g\nclose N,g'\nclose N\in G/\nclose N$.  It suffices to consider the case where $\dist_{G/\nclose{N}}(g\nclose N,g'\nclose N)\leq 1$. Choose representatives $g,g'$ that differ by a single generator of $G$. 
Let $\mathbf U\in\mathfrak T$.  Choose $U\in\mathbf U$ so that $\{U,g\}$ is a linked pair, so that since $N$ was chosen sufficiently deep, $\{U,g'\}$ is weakly linked. Hence, the observation above shows $\dist_U(g,g')\leq C$, as required.

%

\textbf{Verifying Definition~\ref{defn:space_with_distance_formula}.\eqref{item:dfs_nesting},\eqref{item:dfs_orthogonal},\eqref{item:dfs_complexity},\eqref{item:dfs:bounded_geodesic_image}:}  The nesting, orthogonality, finite complexity and bounded geodesic image axioms easily follow from Lemma~\ref{lem:unique_close_linked_partner}.

\textbf{Verifying
Definition~\ref{defn:space_with_distance_formula}.\eqref{item:dfs_transversal}:}
We now prove consistency.  Let $\mathbf U\transverse\mathbf V$ and
$g\nclose N\in G/\nclose N$.  Also, suppose $U\in\mathbf U,
V\in\mathbf V$ and $g$ have the property that $\{U,V\}$ and $\{U,g\}$
are linked; this is justified by Lemma~\ref{lem:linked_pairs_exist}.
If $\{V,g\}$ is weakly linked, then, using the observation above, consistency for $\mathbf U,\mathbf
V,g\nclose N$ follows from consistency for $U,V,g$
(Proposition~\ref{prop:aux_HHG}).  If not, by Corollary
\ref{cor:unlinked_have_same_proj} we have $\dist_{U}(\rho^V_U,g)\leq
C$, and hence the consistency inequality holds.

Now suppose that $\mathbf U\propnest\mathbf V$ and $g\nclose N\in
G/\nclose N$.  Also, suppose $U\in\mathbf U, V\in\mathbf V$ and $g$
have the property that $\{U,V\}$ and $\{U,g\}$ are linked and
$U\propnest V$.  If $\{V,g\}$ is a weakly linked pair, then consistency
follows from consistency for $U,V,g$.  Otherwise, apply
Corollary~\ref{cor:unlinked_have_same_proj} as above.

If $\mathbf U\nest\mathbf V$ and $\mathbf W$ satisfies either $\mathbf V\propnest \mathbf W$ or $\mathbf V\transverse\mathbf W$ and $\mathbf W\not\orth\mathbf U$, then $\dist_{\mathbf W}(\rho^{\mathbf U}_{\mathbf W},\rho^{\mathbf V}_{\mathbf W})$ is uniformly bounded by a similar argument.  This completes the proof of consistency. 
 
\textbf{Verifying
Definition~\ref{defn:space_with_distance_formula}.\eqref{item:dfs_large_link_lemma}:}
Let $g\nclose N,g'\nclose N\in G/\nclose N$ and let $\mathbf
W\in\mathfrak S_N$.  We divide into three cases according to whether
$\fontact \mathbf W=g''H/ N$ for some $g''$, or $\fontact\mathbf W
=\fontact W$ for some $W\in\mathbf W$, or $\fontact
W=\pyramid{G/\nclose N}$.  In the first case, nothing is properly
nested into $\mathbf W$ and we are done.

Consider the second case, and choose $g,g'\in g\nclose N,g'\nclose N$
and $W,W'\in\mathbf W$ so that $\{W,g\}$ and $\{W',g'\}$ are linked
pairs.  By translating, we may assume that $W=W'$, and $\dist_{\mathbf
W}(g\nclose N,g'\nclose N)=\dist_W(g,g')$ by definition.  Hence the
large link lemma in $(G,\mathfrak T)$ provides $\mathbf
T_1,\ldots,\mathbf T_k\propnest \mathbf W$, with $\mathbf T_i$
represented by some $T_i\in\mathfrak S$ so that if $U\propnest W$,
then $\dist_U(g,g')>E$ only if $U\nest T_i$ for some $i$.  Suppose
that $\dist_{\mathbf U}(g\nclose N,g'\nclose N)>E+2C$ for some
$\mathbf U\propnest\mathbf W$.  Then there exist $g_1,g_1'\in g\nclose
N,g'\nclose N$ so that $g_1,g_1'$ are both linked to some $U\in\mathbf
U$ with $U\propnest W$.
Lemma~\ref{lem:linked_pairs_exist}.\eqref{item:translates_not_linked}
and Corollary~\ref{cor:unlinked_have_same_proj} imply that
$\pi_U(g),\pi_U(g_1)$ $C$--coarsely coincide, and the same is true for
$\pi_U(g'),\pi_U(g'_1)$, so $U$ must be nested in some $T_i$. Since 
$\rho^U_S$ and $\rho^{T_i}_S$ coarsely coincide, it follows that 
$U,T_i$ is a linked pair and thus
$\mathbf U\nest\mathbf T_i$.

Consider the third case.  Let $g,g'\in g\nclose N,g'\nclose N$ be minimal-distance (in $\dist_\triangle$) representatives.  Observe that $\{g,g'\}$ is a linked pair and that $\dist_{\triangle}(g,g')=\dist_{\triangle_N}(g\nclose N,g'\nclose N)$, where $\dist_{\triangle_N}$ is the metric on $\pyramid{G/\nclose N}$.  The claim follows from the large link lemma in $(G,\mathfrak T)$ as above.  

\textbf{Verifying
Definition~\ref{defn:space_with_distance_formula}.\eqref{item:dfs_partial_realization}:}
Let $\{\mathbf U_i\}_{i=1}^k$ be a totally orthogonal subset of
$\mathfrak T/\nclose N$. If $\mathbf U_i\not\in\mathfrak S/\nclose N$ for some $i$, then $k=1$ and partial realization obviously holds.  Hence suppose that $\mathbf U_i\in\mathfrak S/\nclose N$ for all $i$.  Then, by Lemma~\ref{lem:hellyish}, for each $i\leq k$, there exists $U_i\in\mathbf U_i$ so that for all distinct $i,j$, we have $U_i\orth U_j$ and $\{U_i,U_j\}$ is a linked pair.  The claim now follows from partial realization in $(G,\mathfrak T)$.

\textbf{Verifying Definition~\ref{defn:space_with_distance_formula}.\eqref{item:dfs_uniqueness}:}  Let $g\nclose N,g'\nclose N\in G/\nclose N$ and let $\kappa\geq 0$.  Suppose that for all $\mathbf U\in\mathfrak T$, we have $\dist_{\mathbf U}(g\nclose N,g'\nclose N)\leq\kappa$.  Let $g,g'\in g\nclose N,g'\nclose N$ be minimal-distance (in $\dist_\triangle$) representatives.  

We now show that $\pi_U(g),\pi_U(g')$ are $(\kappa+C')$--close in every $U\in\mathfrak S$, for $C'$ as in Lemma~\ref{lem:labelled}.  By uniqueness in $(G,\mathfrak T)$, it follows that $\dist_G(g,g')\leq\theta(\kappa+C')$, which implies the required bound on $\dist_{G/\nclose N}(g\nclose N,g'\nclose N)$.

Since we chose $g,g'$ at minimal distance, we have $\dist_\triangle(g,g')\leq\kappa$. Suppose that there exists $U\in\mathfrak T-\{S\}$ with $\dist_U(g,g')>C'$.  Then by Lemma~\ref{lem:labelled}, $U$ is linked to both $g$ and $g'$, and hence $\dist_U(g,g')=\dist_{\mathbf U}(g\nclose N,g'\nclose N)\leq\kappa$, as required.
\end{proof}

\bibliographystyle{alpha}
\bibliography{hhs_asdim}
\end{document}